\documentclass[reqno,12pt,letterpaper]{amsart}
\usepackage{amsmath,amssymb,amsthm,graphicx,mathrsfs,url}
\usepackage[usenames,dvipsnames]{color}
\usepackage[colorlinks=true,linkcolor=Red,citecolor=Green]{hyperref}

\usepackage{tikz}

\usepackage{graphicx}
\usepackage{caption}
\usepackage{subcaption}
\usepackage[colorlinks=true]{hyperref}
\usepackage{color}
\usepackage{amsmath,amsfonts}
\usepackage{calrsfs}
\DeclareMathAlphabet{\pazocal}{OMS}{zplm}{m}{n}

\def\smallsection#1{\smallskip\noindent\textbf{#1}.}

\numberwithin{equation}{section}
\usepackage{amsmath, amsthm}
    \usepackage{dsfont}
    \usepackage{fancybox}
    \usepackage{amssymb}
\newcommand{\HH}{{\pazocal{H}}}
\newcommand{\R}{\mathbb{R}}
\newcommand{\Id}{\operatorname{Id}}

\newcommand{\w}{\omega}
\newcommand{\dd}[2]{\dfrac{\partial #1}{\partial #2}}

\newcommand{\comp}{{\operatorname{comp}}}
\newcommand{\epsi}{\varepsilon}

\newcommand{\WF}{{\operatorname{WF}}}
\newcommand{\Ree}{{\operatorname{Re}}}
\newcommand{\I}{{\operatorname{I}}}

\newcommand{\II}{{\operatorname{II}}}
\newcommand{\III}{{\operatorname{III}}}
\newcommand{\Ha}{{\operatorname{H}}}
\newcommand{\WB}{{\pazocal{W}\pazocal{B}}}
\newcommand{\ttv}{{v_{\operatorname{toy}}}}
\newcommand{\ttsquare}{{\square_{\operatorname{toy}}}}

\newcommand{\Hz}{\hat{z}}

\newcommand{\lr}[1]{\left\langle #1 \right\rangle}
\newcommand{\blr}[1]{\left\langle #1 \right\rangle}
\newcommand{\oY}{{\overline{Y}}}

\newcommand{\FF}{\pazocal{F}}

\newcommand{\supp}{\mathrm{supp}}
\newcommand{\Z}{\mathbb{Z}}
\newcommand{\Vv}{\mathbb{V}}

\newcommand{\tK}{\tilde{K}}
\newcommand{\DD}{\pazocal{D}}
\newcommand{\CC}{\pazocal{C}}

\newcommand{\Hh}{\mathbb{H}}
\newcommand{\ap}{{\operatorname{ap}}}
\newcommand{\MM}{\pazocal{M}}

\newcommand{\SSS}{\pazocal{S}}

\newcommand{\Hr}{{\hat{r}}}
\newcommand{\C}{\mathbb{C}}

\newcommand{\Ss}{\mathbb{S}}
\newcommand{\az}{\alpha}

\newcommand{\p}{\partial}

\newcommand{\te}{\theta}
\newcommand{\vp}{\varphi}

\newcommand{\HT}{{\hat{T}}}
\newcommand{\BB}{{\pazocal{B}}}
\newcommand{\EE}{\pazocal{E}}

\newcommand{\UU}{\pazocal{U}}

\newcommand{\tchi}{{\tilde{\chi}}}

\newcommand{\loc}{{\operatorname{loc}}}

\newcommand{\scatt}{{\operatorname{scatt}}}

\newcommand{\Ht}{{\hat{t}}}

\newcommand{\Haw}{{\operatorname{Haw}}}

\newcommand{\HSS}{\widehat{\SSS}}

\newcommand{\Ee}{\mathbb{E}}
\newcommand{\vac}{{\operatorname{vac}}}

\newcommand{\TTT}{\mathcal{T}}
\newcommand{\Ime}{{\operatorname{Im} \ }}
\newcommand{\system}[1]{\left\{\begin{matrix}
#1
\end{matrix} \right.}

\newcommand{\tv}{{\tilde{v}}}
\newcommand{\oZ}{\overline{Z}}
\newcommand{\matrice}[1]{\left(\begin{matrix}
#1
\end{matrix} \right)}
\newcommand{\HHS}{{\widehat{\SSS}}}
\newcommand{\Hhh}{{\mathcal{H}}}
\newcommand{\AAA}{{\mathcal{A}}}
\newcommand{\Mmm}{{\mathcal{M}}}

\title[A quantitative version of Hawking radiation]{A quantitative version of Hawking radiation}
\author{Alexis Drouot}
\email{alexis.drouot@gmail.com}
\address{Universit\'e Joseph Fourier and University of California, Berkeley.}
  
\newtheorem{thm}{Theorem}

\newtheorem{lem}{Lemma}[section]
\newtheorem{proposition}[lem]{Proposition}

\newtheorem{theorem}[thm]{Theorem}

\begin{document}

\maketitle

\begin{abstract} We present a proof of the existence of the Hawking radiation for massive bosons in the Schwarzchild-de Sitter metric. It provides estimates for the rates of decay of the initial quantum state to the Hawking thermal state. The arguments in the proof include a construction of radiation fields by conformal scattering theory; a semiclassical interpretation of the blueshift effect; and the use of a WKB parametrix near the surface of a collapsing star. The proof does not rely on the spherical symmetry of the spacetime.
\end{abstract}

\section{Introduction} 

The goal of this paper is to describe mathematically the Hawking radiation in the context of a Klein-Gordon field outside a Schwarzchild-de Sitter black hole. Hawking's celebrated prediction states that for outside observers, a black hole appears to emit particles at temperature
\begin{equation*}
T_\Haw = \dfrac{\hbar c^3\kappa_-}{2\pi G k_B} .
\end{equation*}
Here $\kappa_-$ is the surface gravity of the black hole horizon, $\hbar, c, G$ are the three universal constants, and $k_B$ is the Boltzman constant. The discovery of this effect goes back to Hawking's seminal work \cite{Hawking}. This result has had a remarkable effect on the scientific community and has generated since a lot of work.

The first mathematical treatment of the Hawking radiation goes back to a remarkable series of papers of Bachelot \cite{Bachelot0},  \cite{Bachelot2}, \cite{Bachelot1}. This pioneering work describes the apparent emission of non-interacting massive bosonic particles by a spherical black hole. Later Bachelot \cite{Bachelot3} and Melnyk \cite{Melnyk1} studied the emission of fermions by spherical black holes. Also in the case of fermions in a spherical background Melnyk \cite{Melnyk2} showed that the convergence to the Hawking thermal state still occurs while starting with a state at a positive temperature. This is an important physical case -- see Gibbons-Hawking \cite{GibHaw}. The first, and so far the only description of the Hawking radiation in a non-spherical background is due to H\"afner \cite{Hafner}. There the author describes the emission of fermions by Kerr black holes. More recently Bouvier and G\'erard \cite{BouGer} have proposed a toy model studying the Hawking radiation in the case of interacting fermions.

Our goal here is to develop techniques to study the Hawking radiation in a more robust way, which could handle non-symmetric black holes. We investigate the emission of massive spin-$0$ particles by spherical black holes with positive cosmological constant, but we do not rely on the separation of variables that was used in the works cited above. We also avoid the use of compactness arguments, and we privilege geometric aspects. As a result, the proof we give is constructive and we recover a quantitative version of the Hawking radiation. It is shown that as a star collapses to a black hole the initial quantum state converges to the thermal Hawking state exponentially fast. This dramatically improves the rate of convergence to the equilibrium. Note that this result does not apply either to massless fields nor to spaces with vanishing cosmological constant.

This work has been made possible only because of the spectacular growth of interest in the decay of waves on black holes spacetime, see \cite{FKSY06},  \cite{DafRod1}, \cite{FKSY08Err}, \cite{AnBl}, \cite{BonHaf}, \cite{DafRod2}, \cite{Dyatlov1}, \cite{Dyatlov2}, \cite{TT11}, \cite{Dyatlov3}, \cite{Vasy}, \cite{DRS14}, \cite{Ga14a}, \cite{Ga14b}, \cite{MelSaBVas} and references given there. In the case of a non-vanishing cosmological constant the culminating paper is the groundbreaking \cite[Theorem 1.4]{Vasy}. This theorem provides a way to obtain exponential decay for a broad class of perturbations of the Schwarzchild-de Sitter metric. It happens to be crucial to get exponential decay to the thermal state in the framework of Hawking radiation, as achieved here in the case of Schwarzchild-de Sitter black holes. As it was proved that such black holes admit many resonances -- see \cite{SaZw} -- one cannot beat the exponential convergence.

\subsection{Results} Let $\R \times \R \times \Ss$ be the Schwarzchild-de Sitter spacetime (denoted by SdS below) provided with its standard Lorentzian metric and $\square$ be the d'Alembertian operator on SdS -- see \S \ref{sec:1} for precise definitions. A collapsing star spacetime is the submanifold (with boundary) $\MM$ of SdS given by
\begin{equation*}
\MM = \{ (t,x,\w) \in \R^+ \times \R \times \Ss, \ x \geq z_*(t) \}.
\end{equation*}
Here $t \in \R^+ \mapsto z_*(t)$ is a smooth decreasing function with $z_*(t) = -t-A_0 e^{-2\kappa_-t} + O_{C^\infty}(e^{-4\kappa_-t})$ and $z_*(0) < 0$. The surface of the collapsing star is given by $\p\MM=\{(t,x,\w), \ x=z_*(t)\}$. For physical motivations see \S \ref{sec:1}. For $u_0, u_1 \in C^\infty_0(\R \times \Ss)$ and $T$ large enough let $u$ (depending on $T$) be the solution of the Klein-Gordon (KG) boundary value problem
\begin{equation}\label{eq:hgc}
\system{  (\square+m^2) u = 0, \\ \ u|_{\p \MM} = 0, \\ \ u(T) = u_0, \ \p_t u(T) = u_1. }
\end{equation}
Here $m$ is a \textit{positive} number. As explained in \cite{Bachelot1} and reviewed in \S \ref{sub:func} in order to give a mathematical description of the Hawking radiation it suffices to study the behavior of $u|_{t=0}$ on the surface $\Sigma_0=[z_*(0),\infty) \times \Ss$ as $T \rightarrow +\infty$. We define the space of scattering fields $X_\scatt$ as the set of functions $v \in C^\infty(\R \times \Ss)$ such that
\begin{equation*}\begin{gathered}
\exists C \in \R, \ \supp(v) \subset [C,\infty) \times \Ss \text{ and }  \\
\exists \nu > 0, \ v(x,\w) = O_{C^\infty(\R \times \Ss)}\left(e^{-\nu \lr{x}}\right).
\end{gathered}
\end{equation*}
The following theorem describes the behavior of the solution of \eqref{eq:hgc} as $T \rightarrow +\infty$:

\begin{theorem}\label{thm:asymp} Let $u_0, u_1 \in C_0^\infty(\R \times \Ss)$ and $u$ be the solution of \eqref{eq:hgc}. There exists $u_\pm^*\in X_\scatt$ such that on $\Sigma_0$,
\begin{equation*}\begin{gathered}
u(0,x,\w) = \dfrac{r_-}{r} u^*_-\left( \kappa_-^{-1} \ln\left( \dfrac{\gamma_0x}{e^{-\kappa_-T}} \right),\w \right) +  u^*_+(x-T,\w) + \epsi_0(x,\w) \\
\p_tu(0,x,\w) = \dfrac{r_-}{\kappa_- r  x} (\p_xu^*_-)\left( \kappa_-^{-1} \ln\left( \dfrac{\gamma_0x}{e^{-\kappa_-T}} \right),\w \right) + (\p_xu^*_+)(x-T,\w) + \epsi_1(x,\w), \end{gathered}
\end{equation*}
where $\gamma_0 < 0$ and $\epsi_0, \epsi_1$ both belong to $C^\infty(\Sigma_0)$ with
\begin{equation*}\begin{gathered}
|\epsi_0|_{H^{1/2}([z_*(0),1] \times \Ss)} = O(e^{-cT}), \ \ \ \ |\epsi_0|_{C^\infty([0,\infty) \times \Ss)} = O(e^{-cT}),\\
|\epsi_1|_{H^{-1/2}([z_*(0),1] \times \Ss)} = O(e^{-cT}), \ \ \ \ |\epsi_1|_{C^\infty([0,\infty) \times \Ss)} = O(e^{-cT}). \end{gathered}
\end{equation*}
\end{theorem}

This theorem shows the difference between the propagation of KG fields in SdS and the propagation of KG fields in $\MM$: as the time $T$ goes to $+\infty$, a part of $u$ radiates to the cosmological horizon $x=+\infty$ while another part reflects off $\p \MM$ and produces a high frequency term. The latter is responsible for the apparent emission of bosons by the black hole resulting from the collapse. In the following result $\Ee_{\Ha_0,\tau}$ (resp. $\Ee^\pm_{D_x^2,\tau}$) stands for the generating functional of a state at temperature $\tau$ with respect to the KG Hamiltonian $\Ha_0$ (resp. $D_x^2$) in the collapsing star spacetime (resp. near the black hole/cosmological horizons) -- see \S \ref{sub:func} for precise definitions.

\begin{theorem}\label{thm:haw} There exists $\Lambda_0 > 0$ such that for all $0 < \Lambda < \Lambda_0$ and under the assumptions and notations of Theorem \ref{thm:asymp},
\begin{equation*}\begin{gathered}
\Ee_{\Ha_0,\kappa_+/(2\pi)}[u(0),\p_tu(0)] \\ = \Ee^+_{D_x^2,\kappa_+/(2\pi)} [u^*_+,D_x u^*_+] \cdot \Ee^-_{D_x^2,\kappa_-/(2\pi)} [u^*_-,D_x u^*_-] \cdot \left(1 + O(e^{-cT})\right), \ \  T \rightarrow \infty.
\end{gathered}
\end{equation*}
\end{theorem}

This theorem is the mathematical interpretation of Hawking's result. As the time goes to infinity the vacuum state (which has temperature $\kappa_+/(2\pi)$ in SdS) in a spacetime containing a collapsing star splits into two parts. The first one corresponds to a thermal state at temperature $\kappa_+/(2\pi)$ escaping to the cosmological horizon -- as expected. The second part is more surprising: it corresponds to a thermal state at temperature $\kappa_-/(2\pi)$ emerging from the black hole horizon. As seen by an outside observer the resulting black hole emits bosons at temperature $\kappa_-/(2\pi)$. Similar results were obtained in slightly different contexts in \cite{Bachelot2}, \cite{Bachelot1}. The interest of the result, apart for the novelty of the proof, is the exponential rate of convergence. 

\subsection{The Schwarzchild-de Sitter spacetime.}\label{sec:1} Here we define the Schwarzchild-de Sitter spacetime (SdS) and the coordinates systems we will use. We also define collapsing-star spacetimes.

\subsubsection{Definition of the spacetime}\label{sub:2} Let $M$ and $\Lambda$ be two \textit{positive} numbers such that $0 < 9M^2 \Lambda < 1$. The Scharzchild-de Sitter metric on $\R \times \R_+ \times \Ss$ is in coordinates $(t,r,\w)$ given by
\begin{equation}\label{eq:1a}
g =  \dfrac{\Delta_r}{ r^2}dt^2-r^2 \left( \dfrac{dr^2}{\Delta_r} + d\sigma_\Ss^2 \right).
\end{equation}
The metric $d\sigma_\Ss^2$ is the Euclidean metric on the $2$-sphere $\Ss$ and $\Delta_r$ is defined by
\begin{equation*}
\Delta_r = r^2\left( 1-\dfrac{\Lambda r^2}{3} \right) -2M r., \\
\end{equation*}
This polynomial has four distinct roots, one negative, one equal to zero, and two positive ones. We call $0 < r_- < r_+$ the two positive. In this paper we are interested in $r$ lying in some small neighborhood $\UU$ in $\R$ of $[r_-,r_+]$, so that $r_\pm$ are the only roots of $\Delta_r$ on $\UU$. The submanifold $\R \times \UU \times \Ss$ (provided with the metric given by \eqref{eq:1a}) is a smooth Lorentzian manifold. The hypersurfaces $r=r_-$, $r=r_+$ are respectively called the black hole horizon and the cosmological horizon. The surface gravity $\kappa_+$ and $\kappa_-$ of the horizons $r=r_\pm$ are defined by
\begin{equation*}
\kappa_\pm = \mp\dfrac{\Delta_r'(r_\pm)}{2r_\pm^2}.
\end{equation*}
On $\R \times (r_-,r_+) \times \Ss$ we define $\SSS_*$ the system of coordinates $(t,x,\w)$ where $x(r)$ is given by
$x'(r) = r^2/\Delta_r$. The main properties of $x$ are the following asymptotics:
\begin{equation*}\begin{gathered}
x = \mp \dfrac{1}{2\kappa_\pm} \ln|r-r_\pm| + O(1), \ r \text{ near } r_\pm,\\
r=r_\pm + O(e^{\mp 2\kappa_\pm x}), \ x \text{ near } \pm \infty
\end{gathered}
\end{equation*}
-- see \cite[Proposition 4.1]{Dyatlov1} for a more precise statement. As $r$ spans $[r_-,r_+]$ $x$ spans the whole real axis. In $\SSS_*$ the metric $g$ takes the form
\begin{equation*}
g = \dfrac{\Delta_r}{ r^2}(dt^2-dx^2) -r^2  d\sigma_\Ss^2.
\end{equation*}
The smooth Lorentzian manifold $\R \times \R \times \Ss$ provided with the metric $g$ is called the Schwarzchild-de Sitter spacetime. The d'Alembertian operator on this manifold is given by
\begin{equation*}
\square = \dfrac{1}{\Delta_r} \dd{^2}{t^2} - \dfrac{1}{r^2 \Delta_r} \dd{}{x} r^2 \dd{}{x} + \dfrac{\Delta_{\Ss}}{r^2}.
\end{equation*}

\subsubsection{Compactification of the spacetime.}\label{subsec:coo} Fix $K$ a compact interval in $(r_-,r_+)$ and define a smooth function $F_K$ on $(r_-,r_+)$ such that
\begin{enumerate}
\item[$(i)$] For $r \in K$, $F_K(r) = 0$.
\item[$(ii)$] There exists $\mu_K \in C^\infty(\UU)$  with $\mu_K(r_-) = 0$, such that for $r$ in a neighborhood of $r_\pm$,
\begin{equation}\label{eq:yr3}
F_K(r) = - \dfrac{1}{2\kappa_\pm} \ln|r-r_\pm| + \mu_K(r).
\end{equation}
\item[$(iii)$] The level set of $\Ht = t - F_K(r)$ are spacelike -- equivalently $|F_K'(r)| < \Delta_r/r^2$.
\end{enumerate}
Near $r_\pm$ $\Delta_r \sim \mp 2 \kappa_\pm r_\pm^2 (r-r_\pm)$. Differentiating \eqref{eq:yr3} yields
\begin{equation}\label{eq:3d}
F_K'(r) = \pm \dfrac{r^2}{\Delta_r} + \lambda_K'(r), \ r \text{ near } r_\pm
\end{equation}
for some function $\lambda_K \in C^\infty(\UU)$. Point $(iii)$ above  implies that $\mp \lambda_K'(r_\pm) > 0$. It is always possible to construct such a function $F_K$, see for instance \cite{Dyatlov3}. We denote below the system of coordinates $(\Ht,r,\w)$ by $\SSS$. In $\SSS$ the metric $g$ has smooth coefficients up and beyond the event horizons. It takes the form
\begin{equation*}
g =  -r^2 \left( \dfrac{dr^2}{\Delta_r} + d\te^2 \right) - r^2 \sin^2 \te d\vp^2 + \dfrac{\Delta_r}{ r^2}(d\Ht-F_K'(r) dr)^2.
\end{equation*}
The d'Alembertian operator is given in $\SSS$ by
\begin{equation*}
\square = \dfrac{1}{r^2}\left[ - (\p_r - F_K'(r)\p_\Ht) \Delta_r (\p_r - F_K'(r)\p_\Ht) + \Delta_\Ss + \dfrac{r^4}{\Delta_r}\p_t^2 \right].
\end{equation*}
In Figure \ref{fig:31} we give pictorial representations of $\SSS_*$ and $\SSS$.

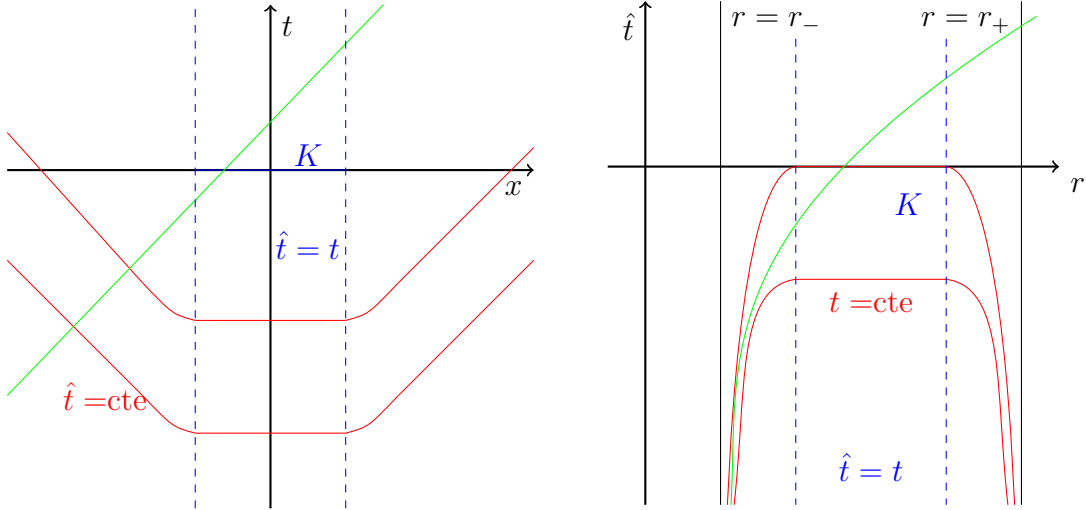
\begin{figure}
\centering
\begin{subfigure}{.5\textwidth}
  \centering
\begin{tikzpicture}
\draw[thick,->] (-3.5,0) -- (3.5,0) node[anchor=north east] {$x$}; \draw[thick,->] (0,-4.5) -- (0,2.2) node[anchor=north west] {$t$};
\draw[blue] (-1,0) -- (1,0);
\node[blue] at (0.5,.2) {$K$};
\draw[red] (-1,-2) -- (1,-2);
\draw[red] (-1,-2) .. controls (-1.3,-1.92) .. (-1.7,-1.5);
\draw[red] (-1.7,-1.5) -- (-3.5, 0.5);
\draw[red] (1,-2) .. controls (1.3,-1.92) .. (1.7,-1.5);
\draw[red] (1.7,-1.5) -- (3.5, 0.3);
\draw[red] (-1,-3.5) -- (1,-3.5);
\draw[red] (-1,-3.5) .. controls (-1.3,-3.42) .. (-1.7,-3);
\draw[red] (-1.7,-3) -- (-3.5,-1.2);
\draw[red] (1,-3.5) .. controls (1.3,-3.42) .. (1.7,-3);
\draw[red] (1.7,-3) -- (3.5,-1.2);
\node[red] at (-2.2,-3) {$\hat{t}=$cte};
\draw[green] (-3.5,-3) -- (1.5,2.2);
\draw[blue,dashed] (-1,-4.5) -- (-1,2.2);
\draw[blue,dashed] (1,-4.5) -- (1,2.2);
\node[blue] at (.5,-1) {$\hat{t}=t$};
\end{tikzpicture}
\end{subfigure}%
\begin{subfigure}{.5\textwidth}
  \centering
\begin{tikzpicture}
\draw[thick,->] (1,-4.5) -- (1,2.2) node[anchor=north east] {$\hat{t}$}; \draw (2,-4.5) -- (2,2.2) node[anchor=north west] {$r=r_-$}; ; \draw (6,-4.5) -- (6,2.2) node[anchor=north east] {$r=r_+$}; \draw[thick,->] (0.5,0) -- (6.5,0) node[anchor=north west] {$r$};
\draw[red] (3,0) -- (5,0);
\node[blue] at (4.5,-.5) {$K$};
\draw[red] (3,0) .. controls (2.2,-.1) and (2.1, -4.6) .. (2.09,-4.5);
\draw[red] (5,0) .. controls (5.8,-0.1) and (5.9, -4.6) .. (5.91,-4.5);
\draw[red] (3,-1.5) .. controls (2.2,-1.6) and (2.35,-3) .. (2.18,-4.5);
\draw[red] (3,-1.5) -- (5,-1.5);
\draw[red] (5,-1.5) .. controls (5.8,-1.6) and (5.65,-3) .. (5.82,-4.5);
\node[red] at (4,-1.8) {$t=$cte};
\draw[green]  (6.2,2) .. controls (1.8,-.8) and (2.24,-2.5) .. (2.14,-4.5);    ;
\draw[blue,dashed] (3,1.7) -- (3,-4.5);
\draw[blue,dashed] (5,1.7) -- (5,-4.5);
\node[blue] at (4,-4) {$\hat{t}=t$};
\end{tikzpicture}
\end{subfigure}
\caption{ On the left (resp. right) we plot in red the level sets of $\Ht$ (resp. $t$) in $\SSS_*$ (resp. $\SSS$). Inside the zone delimited by the blue dashes we have $\Ht=t$. In green we drew the typical trajectory of an untrapped geodesic.}
\label{fig:31}
\end{figure}

\subsubsection{Definition of a collapsing star} In order to describe mathematically the Hawking radiation we give a model of a star collapsing to a black hole. A collapsing-star spacetime is a smooth Lorentzian manifold with boundary $\MM$ given by
\begin{equation*}
\MM = \{(t,x,\w) \in \R \times \R \times \Ss, \ x \geq z_*(t), \ t > 0\}
\end{equation*}
provided with the metric $g$ defined above. The boundary $\BB=\p\MM$ describes the surface of the star. We require $\BB$ to be timelike and the function $t \in \R^+ \mapsto z_*(t)$ to be negative, smooth, decreasing and to satisfy
\begin{equation*}
z_*(t) = -t - A_0 e^{-2\kappa_-t} + O_{C^\infty}(e^{-4\kappa_-t})
\end{equation*}
for some $A_0 > 0$. Such assumptions on $z_*$ are justified in \cite{Bachelot2}. It will be useful to look at the manifold $\MM$ in $\SSS$. In this system of coordinates $\MM$ can be seen as
\begin{equation*}
\left\{(\Ht,r,\w) \in \R \times (r_-,r_+) \times \Ss, \ r \geq z(\Ht) \text{ for } \Ht \in [0,\Ht_\BB], \ \Ht + F_K(r) > 0\right\}.
\end{equation*}
Here $\Ht_\BB$ is a positive number and $\Ht \in [0,\Ht_\BB] \mapsto z(\Ht)$ is a smooth decreasing function such that as $\Ht \rightarrow \Ht_\BB$,
\begin{equation}\label{eq:boundary2}
z(\Ht) = r_- + \az_0(\Ht-\Ht_\BB) + O((\Ht-\Ht_\BB)^2), 
\end{equation}
for some $\az_0 < 0$. In the system of coordinate $\SSS$ this equation can be found in \cite{FH}.

\subsection{Plan of the paper}
We organize the paper as follows. In \S \ref{sec:freeprop} we construct the functions $u_\pm^*$ of Theorem \ref{thm:asymp} using conformal scattering. We give a semiclassical interpretation of the blueshift effect. As we shall see, backwards KG fields asymptotically concentrate near the horizons $r=r_\pm$. This goes together with a radial frequency blow up. This microlocalization result will allow us to study precisely the part of $u$ that reflects off the boundary of the star $\BB$ in \S \ref{sec:mixedprob}. We construct a WKB parametrix and use it to prove Theorem \ref{thm:asymp}. In \S \ref{sec:5} we describe the quantization procedure then prove Theorem \ref{thm:haw}. Despite the fine asymptotics of Theorem \ref{thm:asymp} this last step is quite technical. 

In the figure \ref{fig:dessina} (resp. figure \ref{fig:dessin2}) we plot the propagation of Klein-Gordon fields in $\MM$ starting at time $T \gg 1$ in the system of coordinates $\SSS$ (resp. $\SSS_*$). \begin{center}
\begin{figure}
\begin{tikzpicture}
\draw[thick,->] (0,-1) -- (0,12) node[anchor=north east] {$\hat{t}$}; 
\draw[thick,->] (-1,0) -- (9,0) node[anchor=north west] {$r$};
\draw (1,-1) -- (1,12); 
\draw (8,-1) -- (8,12) node[anchor=north west] {$r=r_+$}; 
\draw[thick] (1,4) .. controls (3.5,2) .. (4,0);
\draw[blue,dashed] (-1,4) -- (9,4);
\node[blue] at (-.8,4.3) {$\hat{t}=\hat{t}_B$};
\draw[red] (3,11) .. controls (1.5,10) and (1.3,6) .. (1.1,3.9);
\draw[red] (6,11) .. controls (4,10) and (2,8.5) .. (1.22,3.82);
\draw[red] (3,11) .. controls (7,9) and (7.8,4) .. (7.8,0);
\draw[red] (6,11) .. controls (7.8,9) and (7.9,4) .. (7.9,0);
\node at (1,12.2) {$r=r_-$};
\draw[red] (3,11) -- (6,11);
\node[red] at (4.5,11.3) {$u_0,u_1$};
\draw[red] (1.1,3.9) .. controls (4,3.7) and (5,1.5) .. (6,0);
\draw[red] (1.22,3.82) .. controls (3.8,3.6) and (4.9,1.5) .. (5.88,0);
\draw[blue,dashed] (-1,11) -- (9,11);
\node[blue] at (-.6,10.6) {$\hat{t}=T$} ;
\node[red] at (2.3,9) {$A$} ;
\node[red] at (6.4,9) {$B$} ;
\node at (4.5,5) {$\CC$} ;
\end{tikzpicture}
\caption{The propagation of KG fields in $\SSS$. In \S \ref{sec:freeprop} we prove that such fields propagate mainly along the zones $A,B$. In $\CC$ the energy decays exponentially. At time $\Ht=\Ht_\BB$ the left part of the KG fields reflects off the boundary of the star and its propagation can be well approximated by a WKB parametrix.}\label{fig:dessina}
\end{figure}
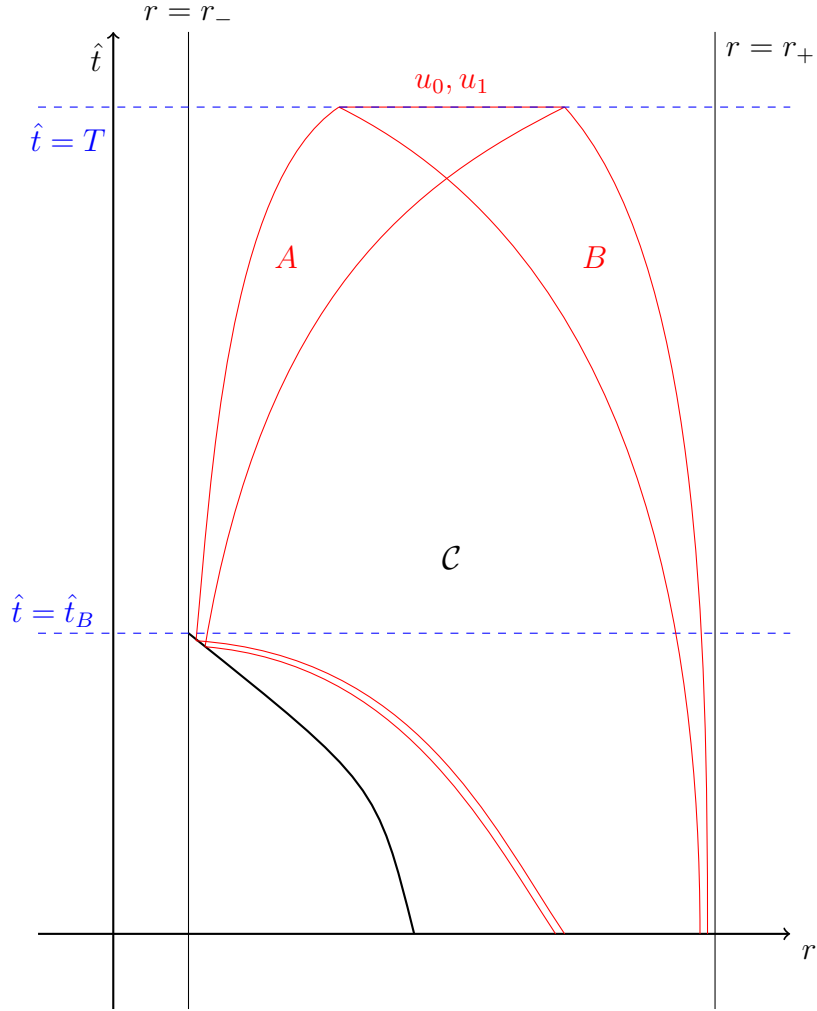
\end{center}
\begin{center}
\begin{figure}
\begin{tikzpicture}
\draw[thick,->] (-7,0) -- (8,0) node[anchor=north east] {$x$}; 
\draw[thick,->] (0,-1)  -- (0,7) node[anchor=north east] {$t$};
\draw[thick] (-7,7) .. controls (-2.5,2.4) and (-2.2, 1.5) .. (-2,0);
\node at (-5.5,5) {$\BB$}; 
\node[blue] at (7,6.5) {$t=T$}; 
\node[red] at (-1,6.5) {$(u_0,u_1)$};
\draw[blue,dashed] (8,6) -- (-6.1,6); 
\draw[red] (-1.5,6) -- (1.5,6);
\draw[red] (-1.5,6) -- (4.5,0);
\draw[red] (1.5,6) -- (7.5,0); 
\draw[red] (-1.5,6) -- (-3.9,3.6);
\draw[red] (-0.3,0) -- (-3.9,3.6);
\draw[red] (1.5,6) -- (-2.63,1.87); 
\draw[red] (-2.63,1.87) -- (-.76,0); 
\end{tikzpicture}
\caption{The propagation of KG fields in $\SSS_*$.}\label{fig:dessin2}
\end{figure}
\end{center}
\subsection{Relation with existing work} Let us describe the relation of this paper with previous work of Fredenhagen-Haag \cite{FH} and Bachelot \cite{Bachelot2}, \cite{Bachelot1}. In \cite{FH} the authors study the Hawking radiation using the same coordinate system as us. In order to perform this derivation they make a few structural assumptions on the asymptotics of KG fields. Theorem \ref{thm:asymp} here \textit{proves} some of these assumptions, clarifying their method. 

In \cite{Bachelot2}, \cite{Bachelot1} the author work in the Schwarzchild metric and derive a result similar to Theorem \ref{thm:haw}. The work is entirely realized in $\SSS_*$ and relies on the separation of variables using spherical harmonics. Here we do not use the spherical symmetry of SdS and rely instead on purely geometric arguments and on \cite[Theorem 1.4]{Vasy} -- which is stable under broad perturbations of the metric. We hope that this work will open the way to a stable treatment of the Hawking radiation.

\subsection{Notations}\label{sub:1}
\begin{itemize}
\item For $z\in \C$, we denote by $\Ree \ z $ and $\Ime z$ the real and imaginary parts of $z$.
\item We denote throughout the paper by $c,c',C>0$ some constants subject to change from a line to another. 
\item If $E$ is a subset of a Lorentzian manifold provided with a time orientation we define $J^+(E)$ its chronological future and $J^-(E)$ its chronological past.
\item The measure $d\w$ is the Lebesgue measure on the $2$-sphere $\Ss$.
\item The space $C^\infty$ with respect to a system of coordinate $\{x^i\}$ is provided with the topology induced by the following convergence. A sequence of functions $f_n$ is said to go to $0$ in $C^\infty$ with respect to $\{x^i\}$ if for all $\az = (\az_1 ..., \az_d)$, $\p_{x_1}^{\az_1} ... \p_{x_d}^{\az_d} f_n$ converges to $0$ in $L^\infty$.
\item If $s \in \R$, the space $H^s$ with respect to a system of coordinate is the standard $L^2$-based Sobolev space.
\item If $X$ is a function space whose topology is given by a family of seminorms and $f,g \in X$, $g \neq 0$, we write $f = O_X(g)$ if $fg^{-1}$ is bounded for all the seminorms.
\item We use the $\pm/\mp$ notation: any time the symbols $\pm/\mp$ appear in an equation, this equation has two meanings: one for the upper subscript, one for the lower subscript. For instance, $f(x) = \pm 1 \text{ for } \mp x \geq 1$ means $f(x) = 1 \text{ for } -x \geq 1 
$ and $f(x) = -1 \text{ for } x \geq 1$.
\item We say that $f \in C^\infty(\R \times \Ss)$ belongs to the class of symbols $S^\delta$ if for all $\az, \beta>0$, $\p_x^\az \Delta_\Ss^\beta f = O(\lr{x}^{\delta-\az})$.
\end{itemize}

\smallsection{Acknowledgment} I would like to thank Dietrich H\"afner for suggesting the project and for his invaluable help and guidance throughout the conception of this paper. I would also like to thank Semyon Dyatlov and Maciej Zworski for many fruitful discussions and ideas. I am indebted to Michal Wrochna for a partial proofreading of the manuscript. Finally, I gratefully acknowledge financial support from the Fondation CFM pour la recherche.

\section{The free propagation}\label{sec:freeprop} 

Fix $K$ a compact set of $(r_-,r_+)$, and $F_K, \Ht$ associated to $K$ as described in \S \ref{subsec:coo}.  Fix $u_0, u_1  \in C_0^\infty((r_-,r_+) \times \Ss)$ with support in $K \times \Ss$. Here we study the scattering of solutions of
\begin{equation}\label{eq:fy6}
\system{(\square+m^2) u = 0 \\ u|_{\Ht=0} = u_0, \ \ \p_\Ht u|_{\Ht=0} = u_1.}
\end{equation}
In \S \ref{subsec:2.1} we recall the decay estimates of the literature. In \S \ref{subsec:1a} we define the radiation fields of $u_0, u_1$ and we study their qualitative properties. In \S \ref{subsec:2.4} we give a microlocal description of the blueshift effect. In this section unless precised otherwise the Sobolev norms are measured with respect to $r,\w$.

\subsection{Pointwise backward decay}\label{subsec:2.1} Decay of KG fields in SdS have been studied in \cite{BonHaf}, \cite{DafRod1} and  \cite{MelSaBVas}. \cite[Theorem $1.3$]{MelSaBVas} states:

\begin{proposition}\label{prop:res} There exists $\nu >0$ such that if $u$ solves $\eqref{eq:fy6}$ then
\begin{equation}\label{eq:hDs}
u(\Ht,r,\w) = O_{C^\infty(\R^+ \times \UU \times \Ss)}(e^{- \nu \Ht}).
\end{equation}
\end{proposition}

The implicit constant in \eqref{eq:hDs} depends only on high-order Sobolev norms of $(u_0,u_1)$ -- see \cite[Theorem $2$]{Dyatlov2} and \cite[Theorem $1.4$]{Vasy} for a precise description of this dependence in a much more general framework. The fact that $\nu$ is positive is due to the absence of the zero resonance in the case of the Klein-Gordon equation with strictly positive mass -- see \cite{Dyatlov1}.

Decay for forwards solutions beyond event horizons is possible because the coordinate system $\SSS$ allows the linear waves to cross the horizons in finite time. As a result all the mass of the solution escapes $[r_-,r_+] \times \Ss$. Things behave differently backwards, as all the mass of the solution remains between $r_-$ and $r_+$. The Klein-Gordon equation is not time-reversible in $\SSS$ though it is time reversible in $\SSS_*$. This induces the following invariance. Let $\gamma$ be the map
\begin{equation*}
\gamma : (\Ht,r,\w) \mapsto (-\Ht-2F_K(r), r,\w).
\end{equation*}
As $\square$ is invariant under $t \mapsto -t$, the pull-back operator $\gamma^* : u \mapsto \gamma^*u$ satisfies $[\gamma^*,\square+m^2]=0$. Since $t=\Ht$ on $K\times \Ss$ we can apply Proposition \ref{prop:res} to $\gamma^* u$. It gives

\begin{proposition}\label{prop:res2} There exists $\nu >0$ such that if $u$ solves $\eqref{eq:fy6}$,
\begin{equation*}
u(\Ht,r,\w) = O_{C^\infty(\{(\Ht,r,\w): \ \Ht+2F_K(r) \leq 0 \})}(e^{\nu(\Ht+2F_K(r))}).
\end{equation*}
\end{proposition}

\subsection{Radiation fields}\label{subsec:1a} Here we define the radiation fields of $u_0, u_1$. The main idea comes from \cite{Friedlander}. There the author constructs the \textit{radiation field of expanding waves} in the Euclidean setting. We adapt this construction to the hyperbolic setting - which happens to be much simpler. Similar ideas can be found in \cite{JP} for the Schwarzchild metric. Fix $\chi_\pm : [r_-,r_+] \rightarrow [0,1]$ such that
\begin{equation}\label{eq:kjx}
\chi_++\chi_- = 1, \ \chi_\pm = 1 \ \operatorname{near} \ r_\pm.
\end{equation}
We recall that $r \mapsto x(r)$ is the radial coordinate defined in \S \ref{sub:2}

\begin{proposition}\label{prop:scatt} If $u$ solves \eqref{eq:fy6} there exist $u_\pm \in X_\scatt$ such that \begin{equation}\label{eq:gfC}
u(\Ht,r,\w) = \left(\sum_{+/-}\chi_\pm (r) u_\pm(-\Ht-2F_K(r),\w)\right) + \epsi(\Ht,x,\w), \ \ \Ht \leq 0
\end{equation}
where $\epsi$ is a smooth function on $\R^-\times\R\times\Ss$ satisfying $\epsi = O_{C^\infty(\R^- \times \R \times \Ss)}(e^{c\Ht})$.
\end{proposition}

\begin{proof} We start by constructing $u_\pm$ and show that they belong to the scattering class $X_\scatt$. Let $v=\gamma^* u$. As $\gamma^*$ commutes with $\square + m^2$ and as $F_K(r)=0$ on $\supp(u_0,u_1)$, $v$ satisfies
\begin{equation*}
\system{ (\square+m^2) v = 0, \\ v|_{\Ht=0} = u_0, \p_t v|_{\Ht=0} = -u_1.}
\end{equation*}
In $\SSS$ the operator $\square$ is a second order differential operator with smooth coefficients on $\R \times \UU \times \Ss$ thus $v$ is smooth on $\R \times \UU \times \Ss$. This shows that the trace of $v$ along $r=r_\pm$ is smooth. Define $u_\pm$ by $u_\pm(x,\w) = v(x,r_\pm,\w)$. 

Since $v$ has support in $J^+(\{0\} \times K\times \Ss ) \cup J^-(\{0\} \times K\times \Ss)$ the functions $u_\pm$ vanish for $x\leq C_K$ for some positive constant $C_K$. By Proposition \ref{prop:res} $v$ and its derivatives decay exponentially -- thus so do $v(x,r_\pm,\w)$. This implies $u_\pm = O_{C^\infty}(e^{-cx})$ and therefore $u_\pm$ belong to $X_\scatt$.

Define $\epsi$ by 
\begin{equation}\label{eq:3v}
\epsi(\Ht,x,\w) = \sum_{+/-} \chi_\pm(r) \left(v (-\Ht-2F_K(r),r,\w) - v(-\Ht-2F_K(r),r_\pm,\w)\right).
\end{equation}
so that \eqref{eq:gfC} is satisfied. We show now that $\epsi = O_{C^\infty(\R^- \times \R \times \Ss}(e^{c\Ht})$. By \eqref{eq:3v} $\epsi$ vanishes if $-\Ht - 2F_K(r) \leq C_K$ and it suffices to work in the set  $-\Ht - 2F_K(r) \geq C_K$. Fix $A$ a (large) number. If $x$ lies in $[-A,A]$, $F_K(r)$ is uniformly bounded and Proposition \ref{prop:res2} implies
\begin{equation}\label{eq:3e}
\epsi(\Ht,x,\w) = O_{C^\infty(\R^- \times [-A,A] \times \Ss)}(e^{c\Ht}).
\end{equation}
Concentrate on $x \in (-\infty,-A]$. For such $x$ there exists a smooth function $\varphi$ with $r=r_-+\varphi(e^{\kappa_-x})$ with $\varphi(0)=0$ -- see \cite[Proposition 4.1]{Dyatlov1}. Equation \eqref{eq:3d} implies
\begin{equation*}
F_K'(r) = - \dfrac{r^2}{\Delta_r} + \lambda_K'\left(r_-+\varphi(e^{\kappa_-x})\right) = -x'(r) + \lambda_K'\left(r_-+\varphi(e^{\kappa_-x})\right).
\end{equation*}
Consequently there exists a smooth function $\psi$ such that $2F_K(r) = -2x-\psi(e^{\kappa_-x})$.  Again if $A$ is large enough then for $x < -A$, $\chi_+(r) = 0$ and $\chi_-(r) = 1$. It follows that that for $x < -A$,
\begin{equation}\label{eq:3b}\begin{gathered}
\epsi(\Ht,x,\w) =  v (-\Ht-2F_K(r),r,\w) - v(-\Ht-2F_K(r),r_-,\w) \\
 = \int^{r_-+\varphi(e^{\kappa_-x})}_{r_-} (\p_r v)(-\Ht + 2x + \psi(e^{\kappa_+x}),\rho,\w) d\rho \\
 = \int^1_0 (\p_r v)(-\Ht + 2x + \psi(z),r_-+s\varphi(z),\w) \varphi(z) ds
 \end{gathered}
\end{equation}
where $z=e^{\kappa_-x}$. This shows
\begin{equation}\label{eq:3a}\begin{gathered}
|\epsi(\Ht,x,\w)| \leq |\varphi(z)| \sup_{s \in [0,1]} \left|(\p_r v)(-\Ht + 2x + \psi(z),r_-+s\varphi(z),\w)\right| \\
\leq C e^{\kappa_-x} e^{\nu(\Ht - 2x-\psi(z))} \leq C e^{\nu \Ht -2\nu x + \kappa_- x}.\end{gathered}
\end{equation}
In the last line we used that $|\varphi(z)| \leq C |z| = C e^{-\kappa_-x}$, Proposition \ref{prop:res} applied to $v$, and the fact that $\psi$ is uniformly bounded. Since $-C_K \geq \Ht + 2F_K(r)$ we have $\Ht-2x \leq 0$, yielding
\begin{equation*}
\nu \Ht -2\nu x + \kappa_- x \leq \inf_{\te \in [0,1]} \te \nu(\Ht - 2 x) + (1-\te)\kappa_-x.
\end{equation*} 
Optimize this with $\te = \kappa_-/(\kappa_-+2\nu)$ to obtain
\begin{equation*}
\nu \Ht -2\nu x + \kappa_- x \leq \dfrac{\kappa_-\nu}{\kappa_-+2\nu} \Ht.
\end{equation*}
This inequality together with the estimate \eqref{eq:3a} shows that $\epsi(\Ht,x,\w) = O(e^{c\Ht})$ when $x < -A$. The same estimate holds while considering angular or time derivatives of $\epsi$. We now focus on radial derivatives of $\epsi$. By an immediate recursion essentially using that $\p_xz=-\kappa_-z$ and $\varphi(0)=0$ there exists smooth functions $p_{\az,\beta}$ independent of $v$ such that
\begin{equation}\label{eq:3c}\begin{gathered}
\dd{^\az}{x^\az} \left( (\p_r v)(-\Ht + 2x + \psi(z),r_-+s\varphi(z),\w) \varphi(z) \right) \\ = z \sum_{\beta_1+\beta_2 \leq \az} p_{\az,\beta}(x,s) (\p_\Ht^{\beta_1} \p_r^{\beta_2+1} v)(-\Ht + 2x + \psi(z),r_-+s\varphi(z),\w).
\end{gathered}
\end{equation}
To give estimates on $\p_x^\az \epsi$ we differentiate \eqref{eq:3b} $\az$ times with respect to $x$ using the identity \eqref{eq:3c}. Proposition \ref{prop:res} applied to $v$ yields
\begin{equation*}\begin{gathered}
|\p_x^\az\epsi(\Ht,x,\w)| \leq C |z| \sup_{s \in [0,1], \beta_1+\beta_2 \leq \az} \left|(\p_\Ht^{\beta_1}\p_r^{\beta_2+1} v)(-\Ht + 2x + \psi(z),r_-+s\varphi(z),\w)\right| \\
\leq C e^{\kappa_-x} e^{\nu(\Ht - 2x-\psi(z))} \leq C e^{\nu \Ht -2\nu x + \kappa_- x}.\end{gathered}
\end{equation*}
We conclude that $\p_x^\az\epsi = O(e^{c\Ht})$ by the same arguments as before. Similar estimates hold for $x > A$. This concludes the proof. \end{proof}

This admits a formulation that does not depend on the compact set $K$ in the system of coordinates $\SSS_*$. By Proposition \eqref{prop:scatt} the pullback of $u$ by the map
\begin{equation*}
\begin{matrix}
  \R \times \R \times \Ss & \rightarrow  & \R \times (r_-,r_+) \times \Ss \\
 (t,x,\w) & \mapsto & (\Ht,r,\w)
\end{matrix} 
\end{equation*}
-- still denoted by $u$ -- satisfies
\begin{equation*}
u(t,x,\w) = u_+(-\Ht-2F_K(r),\w) + \epsi(\Ht,x,\w), \ \ x > A
\end{equation*}
where $A$ is some large enough number. As seen before $u$ is supported in the set $-\Ht-2F_K(r) \geq C_K$. Thus $t=\Ht + F_K(r) \leq (\Ht - C_K)/2$ and $e^{c\Ht} = O(e^{ct})$ on the support of $u$. In addition equation \eqref{eq:3d} implies
\begin{equation*}
-\Ht - 2 F_K(r) = - t - F_K(r) = -t - x -\lambda_K(r_+) + O(e^{-2\kappa_+x}), \  x > A.
\end{equation*}
Functions in the scattering class $X_\scatt$ have rapid decay and thus
\begin{equation*}\begin{gathered}
u(t,x,\w) = u_+(-\Ht - 2 F_K(r),\w) + O_{C^\infty(\R^- \times [A,\infty) \times \Ss)}(e^{ct}) \\ = u_+(-t-x-\lambda_K(r_+),\w) + O_{C^\infty(\R^- \times [A,\infty) \times \Ss)}(e^{ct}).
\end{gathered}
\end{equation*}
A similar estimate holds for $x$ near $-\infty$. This yields

\begin{proposition}\label{prop:scatt2} If $u$ solves \eqref{eq:fy6} for $t \leq 0$ there exists $u_\pm^* \in X_\scatt$ such that
\begin{equation*}
u(t,x,\w) = \sum_{+/-} u_\pm^*(-t \mp x,\w) + + O_{C^\infty(\R^- \times \R \times \Ss)}(e^{ct}).
\end{equation*}
\end{proposition}

The scattering fields $u_\pm^*$ of Proposition \ref{prop:scatt2} are directly obtained from the action of the pull back operator induced by $x \mapsto x-\lambda_K(r_\pm)$ on the scattering fields of Proposition \ref{prop:scatt}.

\begin{proposition}\label{prop:Hbound} Let $u$ be a solution of \eqref{eq:fy6} expressed in $\SSS$. Then 
\begin{equation*}
  |u(\Ht,\cdot)|_{H^{1/2}} + |\p_\Ht u (\Ht,\cdot)|_{H^{-1/2}} = O(1), \ \Ht \rightarrow -\infty.
\end{equation*}
\end{proposition}

This result is critical: for $(u_0,u_1) \neq 0$ and $\delta > 0$, 
\begin{equation}\label{eq:2q}\begin{gathered}
\liminf_{\Ht \rightarrow -\infty} |u(\Ht,\cdot)|_{H^{1/2+\delta}([r_-,r_+] \times \Ss)} + |\p_\Ht u (\Ht,\cdot)|_{H^{-1/2+\delta}([r_-,r_+] \times \Ss)} = \infty, \\
\limsup_{\Ht \rightarrow -\infty}  |u(\Ht,\cdot)|_{H^{1/2-\delta}([r_-,r_+] \times \Ss)} + |\p_\Ht u (\Ht,\cdot)|_{H^{-1/2-\delta}([r_-,r_+] \times \Ss)}  = 0.\end{gathered}
\end{equation}

\begin{proof} We start by proving that if $\epsi$ is given by Proposition \ref{prop:scatt} then the function $\epsilon$ given by $\epsilon(\Ht,r,\w) = \epsi(\Ht,x,\w)$ satisfies
\begin{equation}\label{eq:3g}
  |\epsilon(\Ht,\cdot)|_{H^{1/2}([r_-,r_+] \times \Ss)} + |\p_\Ht \epsilon (\Ht,\cdot)|_{H^{-1/2}([r_-,r_+] \times \Ss)}  = O(e^{c\Ht}).
\end{equation}
By the identity \eqref{eq:3e} and Proposition \ref{prop:res} $\epsilon(\Ht,\cdot)$ (and its derivatives) are $O(e^{c\Ht})$ as $\Ht \rightarrow -\infty$ uniformly on compact subsets of $(r_-,r_+) \times \Ss$. Therefore we fix $\delta > 0$ and we concentrate on the zones $|r-r_\pm| \leq \delta$ -- let's say $r-r_- \leq \delta$. The estimate $\epsi(\Ht,x,\w) = O(e^{c\Ht})$ of the proof of Proposition \ref{prop:scatt} implies $\epsilon(\Ht,r,\w) = O(e^{c\Ht})$. Now we work on higher order derivatives. We have
\begin{equation*}
\epsilon(\Ht,r,\w) = v(-\Ht-2F_K(r),r,\w) - v_-(-\Ht-2F_K(r),r_-,\w) \\
 = \int_{r_-}^r (\p_rv)(-\Ht-2F_K(r),\rho,\w)d\rho.
\end{equation*}
Angular and time derivatives are estimated as in the proof of Proposition \ref{prop:scatt}. For the radial derivative we have
\begin{equation}\label{eq:3f}
\p_r\epsilon(\Ht,r,\w) = (\p_rv)(-\Ht-2F_K(r),r,\w) + \int_{r_-}^r F_K'(r)(\p_\Ht\p_rv)(-\Ht-2F_K(r),\rho,\w)d\rho.
\end{equation}
The first term in the RHS of \eqref{eq:3f} is uniformly bounded. Using that $F_K'(r) |r-r_-|=O(1)$ near $r_-$ the second term can be estimated by
\begin{equation*}
\left|\int_{r_-}^r F_K'(r)(\p_\Ht\p_rv)(-\Ht-2F_K(r),\rho,\w)d\rho\right| \leq C \sup_{[r_-,r]}(\p_\Ht\p_rv)(-\Ht-2F_K(r),\rho,\w) = O(1).
\end{equation*}
It follows that $\p_r \epsilon(\Ht,r,\w) = O(1)$ near $r_-$ -- and a similar estimate holds near $r_+$. Interpolating with the pointwise bound $\epsilon(\Ht,r,\w) = O(e^{c\Ht})$ shows that $|\epsilon(\Ht,\cdot)|_{H^{1/2}} = O(e^{c\Ht})$. In order to prove that $|\p_\Ht\epsilon(\Ht,\cdot)|_{H^{-1/2}} = O(e^{c\Ht})$ we use that $H^{1/2}([r_-,r_+] \times \Ss) \hookrightarrow L^3([r_-,r_+] \times \Ss)$. This implies (by duality) $ L^{3/2}([r_-,r_+] \times \Ss) \hookrightarrow H^{-1/2}([r_-,r_+] \times \Ss)$. Consequently,
\begin{equation*}\begin{gathered}
|\p_\Ht\epsilon(\Ht,\cdot)|_{H^{-1/2}} \leq C\left(\int_{[r_-,r] \times \Ss}\left|\int_{r_-}^r (\p_\Ht\p_rv)(-\Ht-2F_K(r),\rho,\w)d\rho\right|^{3/2} drd\w\right)^{2/3} \\
\leq C |\p_\Ht\epsilon(\Ht,\cdot)|_\infty = O(e^{c\Ht}).
\end{gathered}
\end{equation*}

Now because of Proposition \ref{prop:scatt} and of \eqref{eq:3g} it suffices to prove that the functions
\begin{equation*}
(r,\w) \mapsto u_\pm(-\Ht-2F_K(r),\w)
\end{equation*}
are uniformly bounded in $H^{1/2}([r_-,r_+] \times \Ss)$ as $\Ht \rightarrow -\infty$. The functions $u_\pm$ are uniformly bounded and the space $[r_-,r_+] \times \Ss$ is compact so it is clear that $(r,\w) \mapsto u_\pm(-\Ht-2F_K(r),\w)$ are uniformly bounded in $L^2$. Derivatives with respect to $\w$ will not cause any problem. For the derivative with respect to $r$, note first that on  a neighborhood $V$ of $K\times \Ss$ in $(r_-,r_+) \times \Ss$, the solution decays exponentially. On $U = [r_-,r_+] \setminus V$ we use the weighted $H^1-L^2$ interpolation inequality
\begin{equation*}\begin{gathered}
 \left(\int_U  |D_r^{1/2} (u_-(-\Ht-2F_K(r),\w))|^2 dr d\w \right)^2 \\   \lesssim \int_U \Delta_r |D_r (u_-(-\Ht-2F_K(r),\w))|^2 dr d\w \cdot \int_U \dfrac{1}{\Delta_r} |u_-(-\Ht-2F_K(r),\w)|^2 dr d\w \\
 \lesssim  \int_U \Delta_r  F_K'(r)^2|D_x u_-(-\Ht-2F_K(r),\w)|^2 dr d\w \cdot \int_U \dfrac{1}{\Delta_r} |u_-(-\Ht-2F_K(r),\w)|^2 dr d\w.\end{gathered}
\end{equation*} 
Since $U \cap K \times \Ss = \emptyset$ the function $F_K'$ does not vanish on $U$ and one can make the substitution $x=-\Ht-2F_K(r)$, $dx=-2F_K'(r) dr$. It yields
\begin{equation*}
\left(\int_U |D_r^{1/2} (u_-(x,\w))|^2 dr d\w \right)^2 \lesssim \int_U \Delta_r  |F_K'(r)||(D_x u_-)(x,\w)|^2 dx d\w \cdot \int_U \dfrac{|u_-(x,\w)|^2}{\Delta_r |F_K'(r)|}  dx d\w.
\end{equation*}
Since $\Delta_r F_K'(r)$ is bounded from above and below and $u_- \in X_\scatt$ the RHS is finite. This completes the proof.\end{proof}

\subsection{Semiclassical description of the blueshift effect}\label{subsec:2.4}

Here we state a microlocalization result that will prove to be of great importance in the next section:

\begin{lem}\label{lem:lalala} Let $h = e^{\kappa_- \Ht}$ and $\chi_-$ satisfying \eqref{eq:kjx}. If $u$ solves \eqref{eq:fy6} in $\SSS$ then 
\begin{equation}\label{eq:kd3}\begin{gathered}
\chi_-(r)u(\Ht,r,\w) = u_-\left( \kappa_-^{-1} \ln\left( \dfrac{r-r_-}{h} \right) ,\w \right) + O_{H^{1/2}([r_-,r_+] \times \Ss)}(h^c) \\
\chi_-(r)(\p_\Ht u)(\Ht,r,\w) = O_{H^{-1/2}([r_-,r_+] \times \Ss)}(h^c).\end{gathered}
\end{equation}
\end{lem}

A similar result holds near $r=r_+$. The semiclassical wavefront set of
\begin{equation*}
u_h : (r,\w) \mapsto  u_-\left( \kappa_-^{-1} \ln\left( \dfrac{r-r_-}{h} \right) ,\w \right)
\end{equation*}
satisfies $\WF_h(u_h) \subset \{ (r_-,\w,\xi_r,0) \}$ -- see \cite{Zworski} for definitions. This lemma indicates that the only singularities of KG fields that emerge in the limit $\Ht \rightarrow \infty$ are located near the horizons and directed in the radial direction. Since the radiation fields are generically non vanishing such singularities do arise. This provides a semiclassical description of the blueshift effect. 

\begin{proof}[Proof of Lemma \ref{lem:lalala}] We focus first on the first line of \eqref{eq:kd3}. By Proposition \ref{prop:scatt} and \eqref{eq:3g} it suffices to prove that
\begin{equation}\label{eq:kd5}
u_-(-\Ht-2F_K(r),\w) = u_-\left( \kappa_-^{-1} \ln\left( \dfrac{r-r_-}{h} \right),\w \right) + O_{H^{1/2}}(h^c).
\end{equation}
If $h=e^{\kappa_- \Ht}$ then
\begin{equation*}
-\Ht - 2 F_K(r) = -\kappa_-^{-1} \ln(h) + \kappa_-^{-1} \ln(r-r_-) + \mu_K(r) = \kappa_-^{-1} \ln\left( \dfrac{r-r_-}{h}\right) + \mu_K(r).
\end{equation*}
Define $f(y,\w) = u_-(\kappa_-^{-1}\ln(y/h),\w)$, $\chi(y) = \chi_-(y+r_-)$ and $\zeta$ the smooth function so that $e^{\mu_K(r_-+y)} = 1 + y \zeta(y)$ -- recall $\mu_K(r_-) = 0$. The function $f$ belongs to $S^{-\delta}$ for some $\delta \in [0,1/2]$ as $u_- \in X_\scatt$. In addition $y \mapsto \mu_K(y+r_-)$ is bounded from below on $\supp(\chi_-)$ and thus there exists $\eta > 0$ with $\eta < 1+y\zeta(y)$. We can then apply Lemma \ref{lem:1e} in the Appendix. It shows
\begin{equation*}
f\left(\dfrac{y+y^2\zeta(y)}{h},\w\right) = f\left(\dfrac{y}{h},\w\right) + O_{H^{1/2}}(h^c).
\end{equation*}
Coming back to the variables $u_-, r$ and $\mu_K$ yields the estimate \eqref{eq:kd5}.

We now prove the second line of \eqref{eq:kd3}. By Proposition \ref{prop:scatt} and \eqref{eq:3g} it suffices to prove that
\begin{equation}\label{eq:3r}
(\p_\Ht u_-)(-\Ht-2F_K(r),\w) = (\p_x u_-)\left( \kappa_-^{-1} \ln\left( \dfrac{r-r_-}{h} \right),\w \right) + O_{H^{-1/2}}(h^c).
\end{equation}
Here again we apply Lemma \ref{lem:1e} with $f(y,\w) = (\p_xu_-)(\kappa_-^{-1}\ln(y/h),\w)$ and $\chi_-, \zeta$ as above. It gives
\begin{equation*}
f\left(\dfrac{y+y^2\zeta(y)}{h},\w\right) = O_{H^{-1/2}}(h^c).
\end{equation*}
Coming back to the variables $u_-, r$ and $\mu_K$ yields the estimate \eqref{eq:3r}. This completes the proof of Lemma \ref{lem:lalala}. \end{proof}

\section{The mixed problem}\label{sec:mixedprob}

Fix $K$ a compact set in $(r_-,r_+)$ and $F_K, \Ht$ as described in \S \ref{subsec:coo}. Let $u_0, u_1 \in C_0^\infty((r_-,r_+) \times \Ss)$ with support in $K \times \Ss$. In this section we study the behavior as $T \rightarrow \infty$ of the solution of 
\begin{equation}\label{eq:1c}
\system{ (\square +m^2) u = 0, \ \ \  u|_\BB = 0, \\ u(T) = u_0, \ \ \p_\Ht u(T) = u_1 }
\end{equation}
for times $\Ht \in [0,\Ht_\BB]$. For times $\Ht \in [\Ht_\BB,T]$ the boundary plays no role and in the limit $T \rightarrow \infty$, $u$ at time $\Ht = \Ht_\BB$ is described by the results of \S \ref{sec:1}. In \S \ref{subsec:b1/2} we prove that the solutions of \eqref{eq:1c} are uniformly bounded in a suitable space. In \S \ref{sec:hg} we construct a radial coordinate $\Hr$ so that $\Ht + \Hr$ corresponds mainly to the phase function for an appropriate WKB parametrix. In \S \ref{sec:4.1} we study the reflection of KG fields off $\BB$ by comparing them with the solution of a toy model for short times. In \S  \ref{subsec:gbs} we construct a parametrix that precisely describes the solution of \eqref{eq:1c} at time $\Ht=0$ away from $r=r_+$. In \S \ref{sec:Haw} we prove Theorem \ref{thm:asymp}.

\subsection{The $H^{1/2}$ bound.}\label{subsec:b1/2} 
We first recall \cite[Lemma $24.1.5$]{Hormander}.

\begin{proposition}\label{prop:nrj} For all $s \in \R$ there exists $C > 0$ with the following. For every smooth function $u$ vanishing on $\BB$ and $\Ht \in [0,\Ht_\BB]$,
\begin{equation}\begin{gathered}
|u(\Ht)|_{H^s} + |\p_\Ht u(\Ht)|_{H^{s-1}} \\ \leq C \left( |u(\Ht_\BB)|_{H^s} + |\p_\Ht u(\Ht_\BB)|_{H^{s-1}} +\int_{\Ht}^{\Ht_\BB} |(\square+m^2) u(\tau)|_{H^{s-1}} d\tau \right).
\end{gathered}
\end{equation}
\end{proposition}

\noindent Together with Proposition \ref{prop:Hbound} this yields a global energy estimate for solutions of \eqref{eq:hgc}:

\begin{proposition} For every $T \geq \Ht_\BB$ the boundary value problem \eqref{eq:1c} admits a unique smooth solution. In addition there exists a constant $C$ such that for all $\Ht \in [0,T]$,
\begin{equation}\label{eq:fds}
|u(\Ht)|_{H^{1/2}}+|\p_t u(\Ht)|_{H^{-1/2}} \leq C.
\end{equation} 
\end{proposition}

\begin{proof} By \cite[Theorem $24.1.1$]{Hormander} solutions of \eqref{eq:1c} with smooth compactly supported initial data exist, are smooth, and are unique in the class of smooth functions. By Proposition \ref{prop:Hbound}, $u$ satisfies
\begin{equation}\label{eq:dsq}
|u(\Ht)|_{H^{1/2}} + |\p_t u(\Ht)|_{H^{-1/2}} \leq C 
\end{equation}
for $\Ht \in [\Ht_\BB,T]$ and a constant $C$ independent of $T, \Ht$. The estimate \eqref{eq:fds} follows from \eqref{eq:dsq} and Proposition \ref{prop:nrj}. \end{proof}

\subsection{Definition of a new radial coordinate}\label{sec:hg} In this section we introduce a new radial coordinate $\Hr$ motivated by the following heuristic arguments. Lemma \ref{lem:lalala} indicates that the only singularities of $u$ emerging in the limit $T \rightarrow \infty$ must be located near the horizons and directed radially. We expect these singularities to reflect off the boundary $\BB$ according to the optic laws of reflection. We would like to chose $\Hr$ so that after the reflection the singularities propagate along the hypersurface $\{\Ht + \Hr = 0\}$ -- which must then be lightlike. This induces the eikonal equation
\begin{equation}\label{eq:h4e}
\system{ g(\nabla(\Ht+\Hr),\nabla(\Ht+\Hr)) = 0 \\ \Hr(r_-) = 0, \ \  \p_r \Hr = \lambda_K'(r_-).}
\end{equation}
The first equation is implied  by the lightlike condition on the hypersurface $\{\Ht + \Hr = 0\}$. The second equation is arbitrary. The third equation comes from the laws of reflection.

\begin{lem} Equation \eqref{eq:h4e} admits a unique a unique solution which is given by $\Hr = x + F_K(r) + \Ht_\BB.$
\end{lem}

The uniqueness part of this lemma is a standard result in the theory of eikonal equations. A computation using the explicit formula (or using that the zero set of $t+x$ is lightlike) yields the global existence. We omit the proof of this result. From $\Hr = x + F_K(r) + \Ht_\BB$ we see that $\p_r \Hr (r_-) = \lambda_K'(r_-) >0$ and that $r \mapsto \Hr$ never vanishes. Consequently this map is a smooth diffeomorphism on its range and $(\Ht,\Hr,\w)$ is a global coordinate system that we denote below by $\HHS$. By construction this coordinate system is more appropriate than $\SSS$ to study solutions of \eqref{eq:1c} for $\Ht \in [0,\Ht_\BB]$. In Figure \ref{fig:dessin3} we drew the graph of a radial geodesic propagating along the horizon $r=r_-$ for $\Ht \geq \Ht_\BB$ and reflecting off $\BB$ at $\Ht=\Ht_\BB$ in both $\SSS$ and $\HHS$.

\noindent \textbf{Remark.} Although in SdS $\Hr$ can be defined directly using the above explicit formula we chose to give a geometric definition of $\Hr$ that does not rely on the spherical symmetry. It is then possible to generalize $\Hr$ to perturbations of SdS.

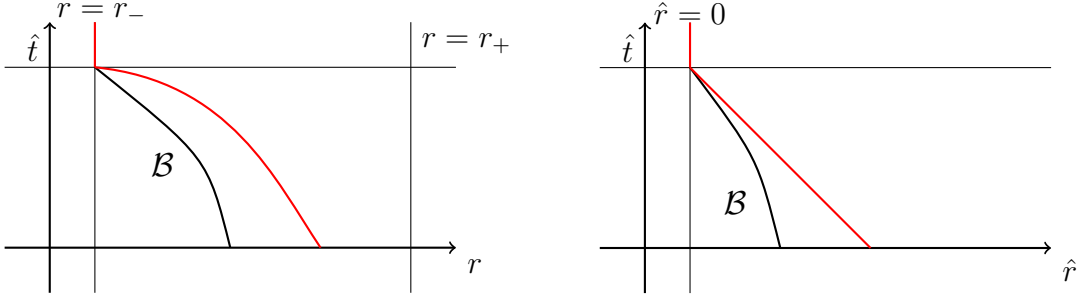
\begin{figure}
\centering
\begin{subfigure}{.5\textwidth}
  \centering
  \begin{tikzpicture}
\draw[thick,->] (0,-1*.6) -- (0,5*.6) node[anchor=north east] {$\hat{t}$}; 
\draw[thick,->] (-1*.6,0) -- (9*.6,0) node[anchor=north west] {$r$};
\draw (1*.6,-1*.6) -- (1*.6,5*.6); 
\draw (8*.6,-1*.6) -- (8*.6,5*.6) node[anchor=north west] {$r=r_+$}; 
\draw[thick] (1*.6,4*.6) .. controls (3.5*.6,2*.6) .. (4*.6,0);
\draw (-1*.6,4*.6) -- (9*.6,4*.6);
\node at (1.5,1.1) {$\BB$};
\draw[red,thick] (1*.6,5*.6) -- (1*.6,4*.6);
\node at (1.15*.6,5.2*.6) {$r=r_-$};
\draw[red,thick] (1*.6,4*.6) .. controls (4*.6,3.7*.6) and (5*.6,1.5*.6) .. (6*.6,0);
\end{tikzpicture}
\end{subfigure}%
\begin{subfigure}{.5\textwidth}
  \centering
\begin{tikzpicture}
\draw[thick,->] (0,-1*.6) -- (0,6*.5) node[anchor=north east] {$\hat{t}$}; 
\draw[thick,->] (-1*.6,0) -- (9*.6,0) node[anchor=north west] {$\Hr$};
\draw (1*.6,-1*.6) -- (1*.6,5*.6); 
\draw[thick] (1*.6,4*.6) .. controls (2.5*.6,2*.6) .. (3*.6,0);
\draw (-1*.6,4*.6) -- (9*.6,4*.6);
\node at (1.2,.6) {$\BB$};
\draw[red,thick] (1*.6,5*.6) -- (1*.6,4*.6);
\node at (1*.6,5.2*.6) {$\Hr=0$};
\draw[red,thick] (1*.6,4*.6) -- (5*.6,0);
\end{tikzpicture}
\end{subfigure}
\caption{Propagation of a backward geodesic reflecting off the boundary $\BB$ in $\SSS$ on the left and in $\HHS$ on the right.}
\label{fig:dessin3}
\end{figure}

In $\HHS$ the wave operator $\square$ is given by
\begin{equation*}
\square  = \dfrac{1}{r^2}\left( - \left(\p_r \Hr \dd{}{\Hr} - F_K'(r)\dd{}{\Ht} \right) \Delta_r \left(\p_r \Hr \dd{}{\Hr} - F_K'(r)\dd{}{\Ht} \right)  + \Delta_{\Ss} + \dfrac{r^4}{\Delta_r} \dd{^2}{\Ht^2} \right).
\end{equation*}
It has principal symbol
\begin{equation*}
\sigma_p(\square)  = \dfrac{1}{r^2}\left[ - \Delta_r((\p_r\Hr) \xi_\Hr - F_K'(r)\xi_\Ht )^2  - \xi_\w^2 + \dfrac{r^4}{\Delta_r} \xi_\Ht^2 \right]
\end{equation*}
where $\xi_\Ht, \xi_\Hr, \xi_\w$ are the dual variables of $\Ht,\Hr,\w$ ($\xi_\w$ is simply the symbol of the operator $\sqrt{\Delta_\Ss}$). The next lemma reformulates Lemma \ref{lem:lalala} in $\HHS$.

\begin{lem}\label{lem:gcS} Let $\chi_-$ satisfying \eqref{eq:kjx} and $u_-$ in $X_\scatt$. Then as $h \rightarrow 0$,
\begin{equation}\label{eq:3s}
\chi_-(r)u_- \left( \kappa_-^{-1} \ln \left(\dfrac{r-r_-}{h}\right),\w \right) = \chi_-(r)
u_- \left( \kappa_-^{-1} \ln \left(\dfrac{\Hr}{\lambda_K'(r_-)h}\right),\w \right) + O_{H^{1/2}}(h^c).
\end{equation}
\end{lem}

\begin{proof} Let $y = \lambda_K'(r_-)^{-1}\Hr$. Since $\Hr(r_-) = 0$ and $\p_r \Hr = \lambda_K'(r_-)$ there exists a smooth function $\zeta$ such that
\begin{equation*}
r - r_- = y + y^2 \zeta(y).
\end{equation*}
Since $r \mapsto \Hr(r_-)$ is a diffeomorphism with $\Hr(r_-) = 0$ there exists $\eta > 0$ such that $r-r_- > \eta y$ when $r \in \supp(\chi_-)$. It follows that $\eta y < y+y^2 \zeta(y)$. If $f(y,\w) = u_-(\kappa_-^{-1}\ln(y/h),\w)$ and $\chi=\chi_-$ then $f$ belongs to $S^{-\delta}$ for some $\delta \in (0,1/2)$. Therefore Lemma \ref{lem:1e} applies and gives \eqref{eq:3s} -- after coming back to the variables $u_-, \Hr$ and $\chi$. \end{proof}

To study the solutions of \eqref{eq:1c} for $\Ht \in [0,\Ht_\BB], T \rightarrow \infty$ modulo error terms in $H^{1/2} \oplus H^{-1/2}$ it suffices to study the limit as $h \rightarrow 0$ of 
\begin{equation*}
\system{(\square+m^2) u = 0, \ \ \ u|_\BB = 0 \\ u(\Ht_\BB,\Hr,\w) = u_- \left( \kappa_-^{-1} \ln \left(\dfrac{\Hr}{\lambda_K'(r_-)h}\right),\w \right), (\p_\Ht u)(\Ht_\BB) = 0}
\end{equation*}
where $u_- \in X_\scatt$ -- see Lemma \ref{lem:lalala} and Proposition \ref{prop:nrj}. We consequently introduce the auxiliary boundary value problem
\begin{equation}\label{eq:1d}
\system{(\square+m^2) v = 0, \ \ v|_\BB=0, \\ v(\Ht_\BB,\Hr,\w) = v_0\left( \dfrac{\Hr}{h},\w \right), (\p_\Ht v)(\Ht_\BB) = 0.}
\end{equation}
Fix $A_0 > 0$. We say that $v_0$ satisfies hypothesis \ref{eq:3h} if 
\begin{equation}\tag{$\operatorname{H}_\ell$}\label{eq:3h}\stepcounter{equation}
v_0 \in C_0^\infty(\R \times \Ss), \ \supp(v_0) \subset (0,\ell) \times \Ss, \ |v_0|_{C^2} \leq A_0.
\end{equation}
In \S \ref{sec:4.1}-\ref{subsec:gbs} we study \eqref{eq:1d} assuming that \ref{eq:3h} holds. We obtain an asymptotic of $v$ uniform in $\ell$ and $h$. In \S \ref{sec:Haw} we prove Theorem \ref{thm:asymp} mainly by passing to the limit $\ell \rightarrow \infty$ in a sense depending on $h$. 

\subsection{Study of a toy model}\label{sec:4.1}  
We study here the local structure of the reflection on $\BB$. We use the same toy model as \cite{Bachelot1}. In $\HHS$ the boundary $\BB$ is described by $\Hr = \Hz(\Ht)$, where $\Hz$ is smooth, decaying and satisfies
\begin{equation*}
\Hz(\Ht) = \az_0\lambda_K'(r_-) (\Ht - \Ht_\BB) + O(\Ht-\Ht_\BB)^2,
\end{equation*}
and $\az_0 < 0$ is given in \eqref{eq:boundary2}. As heuristically described earlier in the limit $T \rightarrow \infty$ the solutions of \eqref{eq:1c} on $\Ht = \Ht_\BB$ are microlocalized radially near the black hole horizon -- so in $\HSS$ they are microlocalized near $\Hr=0$. We consequently introduce the quantization of $\sigma_p(\square)$ at $\Hr=0, \xi_\w=0$:
\begin{equation*}
\ttsquare = 2\lambda_K'(r_-)(-\p_\Ht \p_\Hr+ \p_\Ht^2).
\end{equation*}
Let $\ttv$ be the solution of the boundary value problem 
\begin{equation}\label{eq:r6t}
\system{ \ttsquare \ttv = 0,  \ \ \ \ttv|_\BB = 0, \\  \ttv(\Ht_\BB,\Hr,\w) = v_0\left(\dfrac{\Hr}{h},\w \right), \ \p_\Ht \ttv(\Ht_\BB) = 0.}
\end{equation}
Here we prove the following

\begin{proposition}\label{prop:1} Let $c_\BB = |\Hz'|_\infty$ and $c > c_\BB$. Assume that $v_0$ satisfies \ref{eq:3h}. 
\begin{enumerate}
\item[$(i)$] For $\Ht \in [\Ht_\BB-c\ell h, \Ht_\BB-c_\BB \ell h]$ the solution $\ttv$ of \eqref{eq:r6t} satisfies
\begin{equation*}\begin{gathered}
\ttv(\Ht,\Hr,\w) = v_0\left( y/h,\w\right) + O_{H^{1/2}}(\ell^2 h^{1/2}), \\
\p_\Ht \ttv(\Ht,\Hr,\w) = \dfrac{\beta_0}{h} (\p_xv_0)\left( y/h,\w\right) + O_{H^{-1/2}}(\ell^2 h^{1/2}) \\
\end{gathered}
\end{equation*}
where $\beta_0 = (1+\lambda_K'(r_-)\az_0)^{-1} \lambda'_K(r_-)\az_0$ and $y = \beta_0(\Hr + \Ht - \Ht_\BB)$. \item[$(ii)$] If $v$ is solution of \eqref{eq:1d} then for every $\Ht$ with $\Ht_\BB - \Ht = O(\ell h)$,
\begin{equation*}
|(v-\ttv)(\Ht)|_{H^{1/2}} + |\p_\Ht (v-\ttv)(\Ht)|_{H^{-1/2}} = O(\ell^{8/3} h^{2/3}).
\end{equation*}\end{enumerate}
\end{proposition}

We start with a preliminary lemma:

\begin{lem}\label{lem:1a} For all $s \in [\Hz(0),\Ht_\BB]$, the equation $\Hz(\Ht_1)+\Ht_1 = s$ admits a unique solution $\Ht_1(s)$. The map $s \mapsto \Ht_1(y)$ is smooth and 
\begin{equation}\label{eq:3j}
\Ht_1(s) = \Ht_\BB + \dfrac{1}{1+\lambda_K'(r_-)\az_0}(s- \Ht_\BB) + O((s-\Ht_\BB)^2).
\end{equation}
\end{lem}

\begin{proof} Since $\BB$ is timelike the equation $\Hz(\Ht_1)+\Ht_1 = s$ admits a unique solution for each $s \in [\Hz(0),\Ht_\BB]$. By the implicit function theorem $s \mapsto \Ht_1(s)$ is smooth. We must have $\Ht_1(\Ht_\BB)=\Ht_\BB$ and $\Ht_1'(\Ht_\BB) (\Hz' (\Ht_\BB) +1) = 1$. Since $\Hz'(\Ht_\BB) = \az_0 \lambda_K'(r_-)$ we obtain
$\Ht_1'(\Ht_\BB)= (1+\lambda_K'(r_-)\az_0)^{-1}$. This shows \eqref{eq:3j}.\end{proof}

\begin{proof}[Proof of Proposition \ref{prop:1}] We start with $(i)$. The operator $\ttsquare$ has constant coefficients and consequently the solution $\ttv$ of \eqref{eq:r6t} can be derived in closed form. For times $\Ht \in \TTT = [\Ht_\BB-c\ell h, \Ht_\BB - c_\BB \ell h]$ it is given by
\begin{equation}\label{eq:hFb}
\ttv(\Ht,\Hr,\w) = v_0\left(\dfrac{\Hr + \Ht-\Ht_1(\Hr+\Ht)}{h}, \w\right).
\end{equation}
As $v_0$ satisfies \ref{eq:3h} we have $\supp(\ttv(\Ht_\BB)) \subset [0,\ell h] \times \Ss$. Therefore $\supp(\ttv(\Ht)) \subset [\Hz(\Ht),\ell h+O(\Ht)] \times \Ss$ and if $\Ht - \Ht_\BB = O(\ell h)$ then $\supp(\ttv(\Ht)) \subset [\Hz(\Ht),O(\ell h)] \times \Ss$. The identity \eqref{eq:3j} combined with \eqref{eq:hFb} yields
\begin{equation*}
\ttv(\Ht,\Hr,\w) = v_0\left( \dfrac{y+O(y^2)}{h},\w\right), \ \ \ \Ht \in \TTT
\end{equation*}
where $\beta_0 = (1+\lambda_K'(r_-)\az_0)^{-1} \lambda'_K(r_-)\az_0$ and $y = \beta_0(\Hr +\Ht-\Ht_\BB)$. Since $O(y^2)=O(\ell^2h^2)$ on $\supp(\ttv)$ we can apply an analog of Lemma \ref{lem:1e} and obtain
\begin{equation*}
\ttv(\Ht,\Hr,\w) = v_0\left( y/h,\w\right) + O_{H^{1/2}}(\ell^2 h^{1/2}), \ \ \ \Ht \in \TTT.
\end{equation*}
A similar calculation shows
\begin{equation*}
\p_\Ht \ttv(\Ht,\Hr,\w) = \dfrac{\beta_0}{h} (\p_x v_0)\left( y/h,\w\right) + O_{H^{-1/2}}(\ell^2 h^{1/2}), \ \ \  \Ht \in \TTT.
\end{equation*}

We now prove $(ii)$. Define $w = v-\ttv$. It satisfies $w|_\BB = 0$, $w(\Ht_\BB) = \p_\Ht w(\Ht_\BB) = 0$ and
\begin{equation*}
(\square+m^2) w = - (\square+m^2) \ttv = -(\square- \ttsquare+m^2)\ttv.
\end{equation*}
Proposition \ref{prop:nrj} for $w$ and $s=1/2$ gives
\begin{equation*}
|w(\Ht)|_{H^{1/2}} + |\p_t w(\Ht)|_{H^{-1/2}} \lesssim \int_{\Ht}^{\Ht_\BB}|(\square - \ttsquare+m^2)\ttv(\tau)|_{H^{-1/2}} d\tau.
\end{equation*}
Because of the embedding $L^{3/2} \hookrightarrow H^{-1/2}$ point $(ii)$ follows from an estimate on $|(\square+m^2 - \ttsquare)\ttv(\tau)|_{3/2}$ for $\Ht_\BB-\tau = O(\ell h)$. Write
\begin{equation*}
\square - \ttsquare + m^2 = \az_{\Ht \Ht} \p_\Ht^2 +\az_{\Ht \Hr} \p_\Ht\p_\Hr +\az_{\Hr\Hr} \p_\Hr^2 + r^{-2} \Delta_\Ss + P(D).
\end{equation*}
Here $\az_{ij}$ are smooth functions on $[0,\infty) \times \Ss$ vanishing at $\Hr = 0$ and $P(D)$ is a first order differential operator with smooth bounded coefficients. Recall that for $\Ht_\BB - \tau = O(\ell h)$ we have $\supp(\ttv) \subset [0,O(\ell h)] \times \Ss$. Thus $\az_{\Ht\Ht}(\Hr,\w) = O(\ell h)$ on $\supp(\ttv)$. It follows that
\begin{equation*}
|\az_{\Ht \Ht} \p_\Ht^2 \ttv(\tau)|_{3/2} = O(\ell h) |\p_\Ht^2 \ttv(\tau)|_{3/2}.
\end{equation*}
Given the explicit form of $\ttv$ while differentiating twice with respect to $\Ht$ a factor $h^{-2}$ appears. While integrating over $\Hr$ a factor $(\ell h)^{2/3}$ appears. Thus $|\p_\Ht^2 \ttv(\tau)|_{3/2} = O(\ell^{2/3} h^{-4/3})$ and
\begin{equation*}
|\az_{\Ht \Ht} \p_\Ht^2 \ttv(\tau)|_{3/2} = O(\ell h) O(\ell^{2/3} h^{-4/3}) = O(\ell^{5/3} h^{-1/3}).
\end{equation*}
By the same arguments we obtain a similar bound for $|\az_{ij} \p_i \p_j \ttv(\tau)|_{3/2}$, $i,j \in \{\Ht, \Hr\}$. In order to evaluate $|r^{-2} \Delta_\Ss \ttv(\tau)|_{3/2}$ we note that $r^{-2}$ is uniformly bounded and that differentiating with respect to $\w$ does not make us loose any power of $h$. Consequently $|r^{-2}\Delta_\Ss \ttv(\tau)|_{3/2} = (\ell h)^{2/3}$. The operator $P(D)$ is a differential operator of order $1$ and has bounded coefficients. When differentiating one time $\ttv$ a factor $h^{-1}$ appears; when taking the $L^{3/2}$-norm a factor $(\ell h)^{2/3}$ appears. Therefore $|P(D)\ttv(\tau)|_{3/2} \leq |D \ttv(\tau)|_{3/2} = O(\ell^{2/3} h^{-1/3})$. Grouping all these estimates together we obtain
\begin{equation*}
\left|(\square - \ttsquare + m^2) \ttv (\tau)\right|_{H^{-1/2}} = O(\ell^{5/3} h^{-1/3}), \ \ \  \tau -\Ht_\BB = O(\ell h).
\end{equation*}
Integrating over $\tau \in [\Ht,\Ht_\BB]$ yields
\begin{equation*}
\int_{\Ht}^{\Ht_\BB} |(\square - \ttsquare+m^2)\ttv(\tau)|_{H^{-1/2}} d\tau = O(\ell^{8/3} h^{2/3}).
\end{equation*}
This concludes the proof. \end{proof}

\subsection{Global study of the reflection.}\label{subsec:gbs} 
Here we construct an approximate solution of \eqref{eq:1d} for $v_0$ satisfying \ref{eq:3h}. Let $\tv$  be the solution of 
\begin{equation}\label{eq:trouloulou}
\system{ (\square+m^2) \tv = 0 \\ \tv(\Ht_\BB-c_\BB \ell h) = v_0\left( \beta_0\left( \dfrac{\Hr}{h} -c_\BB \ell\right), \w\right) \\ (\p_\Ht \tv)(\Ht_\BB - c_\BB \ell h) =\dfrac{\beta_0}{h}(\p_x v_0)\left(\beta_0\left( \dfrac{ \Hr}{h} -c_\BB \ell\right), \w\right) .}
\end{equation}
We first construct an approximate solution of \eqref{eq:trouloulou} of the form 
\begin{equation}\label{eq:approx}
\tv_\ap(\Ht,\Hr,\w) = \int_{\R \times \R} e^{\frac{i}{h}(\phi(\Ht,\Hr,p_\Hr) - \Hr' \xi_\Hr)} b(\Ht,\Hr,\xi_\Hr) v_0\left(\beta_0 \left( \dfrac{\Hr'}{h}-c_\BB \ell \right), \w\right) d\Hr' d\xi_\Hr
\end{equation}
by the WKB method -- see \cite{Zworski} for a comprehensive introduction. Because of \S \ref{sec:hg} the phase function $\phi$ is given by $\phi(\Ht, \Hr, \xi_\Hr) = (\Hr+\Ht-\Ht_\BB+c_\BB\ell h) \xi_\Hr$. The amplitude $b$ satisfies the transport equation
\begin{equation}\label{eq:lh5}\begin{gathered}
\system{\left[(\p_\xi p)(x,\phi'_x) \p_x + p_1(x,\phi'_x) \right]b =0 \\  b(\Ht_\BB-c_\BB \ell h,\Hr,p_\Hr) = 1,} \\
 p = - \Delta_r \left((\p_r \Hr) \xi_\Hr - F_K' \xi_\Ht \right)^2 + \dfrac{r^4\xi_\Ht^2}{\Delta_r }, \
p_1 = - (\p_r \Hr) \left( (\p_\Hr \Delta_r \p_r \Hr) \xi_\Hr - (\p_\Hr \Delta_r F_K') \xi_\Ht\right).\end{gathered}
\end{equation}
Equation \eqref{eq:lh5} reduces to
\begin{equation}
-\p_\Hr b + \p_\Ht b - \dfrac{\p_\Hr r}{r} b = 0
\end{equation}
thus $rb$ depends only on $\Ht+\Hr$. Using $b(\Ht_\BB-c_\BB \ell h) = 1$ the function $b$ is explicitly given by
\begin{equation*}
b(\Ht,\Hr) = \dfrac{\az(\Hr+\Ht-\Ht_\BB+c_\BB \ell h)}{\az(\Hr)}, \ \ \ \az(\Hr) = r.
\end{equation*}
The Fourier inversion formula applied to \eqref{eq:approx} yields
\begin{equation*}
\tv_\ap(\Ht,\Hr,\w) = \dfrac{\az(\Hr+\Ht-\Ht_\BB+c_\BB \ell h)}{r} v_0\left(y/h, \w\right)
\end{equation*}
where we recall that $y=\beta_0(\Hr+\Ht-\Ht_\BB)$. When $(y/h,\w)$ belongs to $\supp(v_0)$ we must have $y = O(\ell h)$ and thus $\az(\Hr+\Ht-\Ht_\BB-c_\BB \ell h) = \az(0) + O(\ell h) = r_-+O(\ell h)$. We define
\begin{equation*}
v_\ap(\Ht,\Hr,\w) = \dfrac{r_-}{r} v_0\left(y/h, \w\right).
\end{equation*}

\begin{proposition}\label{prop:jG4} If $v_0$ satisfies \ref{eq:3h} and $v$ solves \eqref{eq:1d} then uniformly in $\ell \geq 1, h \leq 1$
\begin{equation}\label{eq:jf2}
|(v-v_\ap)(0)|_{H^{1/2}} + |\p_t(v-v_\ap)(0)|_{H^{-1/2}} = O(\ell^{8/3} h^{2/3}).
\end{equation}
\end{proposition}

\begin{proof} We start by proving that
\begin{equation}\label{eq:jf3}
|(v-\tv_\ap)(\Ht)|_{H^{1/2}} + |\p_t(v-\tv_\ap)(\Ht)|_{H^{-1/2}} = O(\ell^{8/3} h^{2/3}).
\end{equation}
This holds at time $\Ht = \Ht_\BB - c_\BB \ell h$ by Proposition \ref{prop:1}. According to Proposition \ref{prop:nrj} \eqref{eq:jf3} will follow from the estimates
\begin{equation}\label{eq:3x}
\system{  (\square+m^2) \tv_\ap = O_{H^{-1/2}}((\ell h)^{2/3}), \ \ \ \ \tv_\ap|_\BB = 0, \\
\tv_\ap(\Ht_\BB-c_\BB \ell h) = v(\Ht_\BB-c_\BB \ell h) + O_{H^{1/2}}(\ell^{8/3} h^{2/3}). }
\end{equation}
By construction of $\tv_\ap$ we know that $\square \tv_\ap$ takes the form
\begin{equation*}
\square \tv_\ap = f\left(r,\Hr+\Ht-\Ht_\BB+c_\BB \ell h, y/h,\w\right)
\end{equation*}
where the RHS vanishes when $y/h \notin [0,\ell]$. In particular no negative power of $h$ appears while applying the operator $\square$ to $v_\ap$ and
\begin{equation*}
|\square \tv_\ap|_{H^{-1/2}} \leq C |\square \tv_\ap|_{3/2} \leq C (\ell h)^{2/3}.
\end{equation*}
This proves the first estimate of \eqref{eq:3x}. The boundary condition $\tv_\ap|_\BB = 0$ follows from of the fact that $\BB$ is a timelike surface and that the level sets of $\Ht+\Hr$ are lightlike. The conditions at $\Ht = \Ht_\BB - c_\BB \ell h$ directly follow from point $(ii)$ of Proposition \ref{prop:1}. This yields \eqref{eq:3x} and thus \eqref{eq:jf3}. 

We now prove that 
\begin{equation}\label{eq:3z}
|\tv_\ap(\Ht)- v_\ap(\Ht)|_{H^{1/2}} + |\p_\Ht(\tv_\ap - v_\ap)(\Ht)|_{H^{-1/2}} =O((\ell h)^{1/2}).
\end{equation}
Since $\az(0) = r_-$ there exists a smooth function $\zeta$ such that $\az(\Hr+\Ht-\Ht_\BB-c_\BB \ell h) = (y-c_\BB \ell h) \zeta(y-c_\BB \ell h)$. Thus when $(y/h,\w) \in \supp(v_0)$ we have $\az(\Hr+\Ht-\Ht_\BB-c_\BB \ell h) = O(\ell h)$. This implies
\begin{equation*}
|v_\ap - \tv_\ap|_2 \leq C\ell h \left(\int_{\R \times \Ss} v_0(y/h,\w)^2 dyd\w\right)^{1/2} = O((\ell h)^{3/2}).
\end{equation*}
Similarly $\az'(\Hr+\Ht-\Ht_\BB-c_\BB \ell h) = \zeta(y-c_\BB \ell h) + (y-c_\BB \ell h) \zeta'(y-c_\BB \ell h)$. This is uniformly bounded for $(y/h,\w) \in \supp(v_0)$. The product rule for derivatives imply
\begin{equation*}
|\p_r(v_\ap - \tv_\ap)|_2 = O((\ell h)^{1/2}).
\end{equation*}
The angular derivatives can be estimated the same way as $|v_\ap - \tv_\ap|_2$. Interpolating between these two bounds yields $|v_\ap - \tv_\ap|_{H^{1/2}} = O(\ell h)$. We know estimate $|\p_\Ht(\tv_\ap - v_\ap)(\Ht)|_{H^{-1/2}}$. It suffices to control $|\p_\Ht(\tv_\ap - v_\ap)(\Ht)|_{3/2}$. The formula for $\az'$ above and the product rule for derivatives imply
\begin{equation*}
|\p_\Ht(\tv_\ap - v_\ap)(\Ht)|_{3/2} = O((\ell h)^{2/3}).
\end{equation*}
This proves \eqref{eq:3z} which together with \eqref{eq:jf3} implies \eqref{eq:jf2}.\end{proof}

\subsection{Proof of Theorem \ref{thm:haw}.}\label{sec:Haw} This section is devoted to the proof of Theorem \ref{thm:haw}.It consists in two main steps. In \S \ref{sub:3}-\ref{sub:4} we use an auxiliary hypersurface $\Sigma$ to connect the results of \S \ref{sec:freeprop} with the one of this section. In \S \ref{sub:5} we use the decay of scattering fields to generalize Proposition \ref{prop:jG4} in the limit $\ell \rightarrow \infty$ (in a sense depending on $h$). We work in $\SSS_*$.

\subsubsection{Preliminaries}\label{sub:3} We first recall a few facts. By standard theorems of Leray, if $\Sigma_1, \Sigma_2$ are two smooth spacelike hypersurfaces there is a well defined operator $U(\Sigma_2,\Sigma_1)$ (resp. $U_\BB(\Sigma_2,\Sigma_1)$) 
\begin{equation*}
\begin{matrix}
 C^\infty_0(\Sigma_1, \R^2) & \rightarrow & C_0^\infty(\Sigma_2,\R^2) \\
  (f_0,f_1) & \mapsto & (f|_{\Sigma_2}, \p_tf|_{\Sigma_2})
\end{matrix}
\end{equation*}
where $f$ is the unique solution of
\begin{equation*}
\system{(\square +m^2) f = 0 \\ f|_{\Sigma_1} = f_0, \ \p_tf|_{\Sigma_1} = f_1} \ \ 
\left( \text{resp. } \system{(\square +m^2) f = 0, \ \ \ f|_\BB = 0,  \\ f|_{\Sigma_1} = f_0, \ \p_tf|_{\Sigma_1} = f_1}\right).
\end{equation*}
For two points $p,q$ in $\MM$ we define the causal distance between $p$ and $q$ as
\begin{equation*}
d(p,q) = \inf\{ L(\gamma), \gamma \text{ timelike geodesic connecting } p  \text{ and } q \},
\end{equation*}
where $L(\gamma)$ is the length of the geodesic. If $\Sigma_1$ and $\Sigma_2$ are two spacelike surfaces we define $d(\Sigma_1,\Sigma_2) = \sup d(p,q)$ where the infimum is realized over all points $p \in \Sigma_1, q \in \Sigma_2$. If $d(\Sigma_1,\Sigma_2) < \infty$, the operators $U(\Sigma_2,\Sigma_1), U_\BB(\Sigma_2,\Sigma_1)$ are continuous from $C^\infty(\Sigma_1)$ to $C^\infty(\Sigma_2)$. In what follows we denote simply by $t_0$ (resp. $\Ht_0$) the Cauchy surface $\{t=t_0\}$ (resp. $\{\Ht=\Ht_0\}$). Fix $u_0, u_1 \in C^\infty_0(\R \times \Ss)$. Theorem \ref{thm:asymp} is a result on the limit of $U_\BB(0,T) (u_0,u_1)$ as $T \rightarrow \infty$. 

We now define an auxiliary surface $\Sigma$ that connects the results of \S \ref{sec:freeprop} with Proposition \ref{prop:jG4}. Let $\tK$ be a compact set containing $\supp(u_0,u_1)$ and $[z_*(0),1]$. Let $K$ be the image of $\tK$ under the map $x \mapsto r$. Following \S \ref{subsec:coo} we define a coordinate $\Ht$ satisfying $t=\Ht$ on $K$. The star crosses the black hole horizon at a time $\Ht_\BB > 0$. Finally fix $A \in \R$ with 
\begin{equation}\label{eq:hg3}
J^-(\{\Ht=\Ht_\BB, x \leq A \}) \cap \{ t= 0 \} \subset \{ x \leq 1 \}.
\end{equation}

There exists $\Sigma \subset \R \times \R \times \Ss$ a hypersurface such that:
\begin{enumerate}
\item[$(i)$] $\Sigma$ is smooth and spacelike;
\item[$(ii)$] For $x \leq A$, $\Sigma$ is contained in $\Ht=\Ht_\BB$;
\item[$(iii)$] There exists $B>0$ such that for $x \geq A+B$, $\Sigma$ is contained in $\{t=0\}$.
\end{enumerate}
$\Sigma$ is constructed as follows: for $x \leq A$ it is given by $\Ht = \Ht_\BB$. For $x \geq A$ it is continued smoothly by the equation $t=t(x)$. Here the function $x \mapsto t(x)$ is smooth nonincreasing and satisfies $0 \geq t'(x) > -1$ when $t(x) \neq 0$, and vanishes for $x \geq A+B$ for some positive constant $B$. 

Let $T \gg 1$ and $h = e^{-\kappa_-T}$. For $\ell$ with $\kappa_-^{-1} \ln(\ell) - T\leq A$ let $\chi_\I, \chi_\II, \chi_\III \in C_0^\infty(\R)$ such that $|\chi'_{\I,\II,\III}| \leq 1$ and 
\begin{equation*}
\begin{gathered}
\chi_\I + \chi_{\II} + \chi_{\III} = 1, \ \supp(\chi_\III) \subset [A+B,\infty),\\
\supp(\chi_{\II}) \subset[ \kappa_-^{-1} \ln(\ell) - T-1, A+B+1], \ \ \supp(\chi_\I) (-\infty,\kappa_-^{-1} \ln(\ell) - T].
\end{gathered}
\end{equation*}
Note that the condition $x \geq \kappa_-^{-1} \ln(\ell) - T$ implies $\Hr\geq c_0 \ell h$ for some $c_0 > 0$. We can thus take $\chi_\I$ of the form
\begin{equation}\label{eq:4g}
\chi_\I(x) = \chi\left(\dfrac{\Hr}{\ell h}\right)
\end{equation}
where $\Hr = \Hr(x)$ is the radial coordinate defined in \S \ref{sec:hg} and $\chi(x)$ is equal to $1$ for $|x| \leq c_0/2$ and vanishes for $x \geq c_0$. Below we will chose $\ell$ depending on $h$ in a suitable way. We draw on Figure \ref{fig:sigma} the hypersurface $\Sigma$ as well as the support of the three functions $\chi_{\I,\II,\III}$.

\begin{center}
\begin{figure}
\begin{tikzpicture}
\draw[thick,->] (-8,0) -- (8,0) node[anchor=north east] {$x$}; 
\draw[thick,->] (-2,-2)  -- (-2,7) node[anchor=north east] {$t$};
\node[brown] at (-6.2,6.3) {$x = \kappa_-^{-1} \ln(\ell) - T$};
\node[blue] at (-4.7,5.3) {$x=A$};
\node[brown] at (3,6.3) {$x=A+B$};
\draw[brown,dashed,thick] (3,-2) -- (3,6);
\draw[brown,dashed,thick] (-6.2,-2) -- (-6.2,6);
\draw[blue,dashed,thick] (-4.7,-2) -- (-4.7,5);
\draw[thick,brown,->] (-8,-1.5) -- (-6.2,-1.5);
\draw[thick,brown,->] (8,-1.5) -- (3,-1.5);
\draw[thick,brown,->] (0,-1) -- (-7.2,-1);
\draw[thick,brown,->] (0,-1) -- (4,-1);
\node[brown] at (6,-2) {$\supp(\chi_{\III})$};
\node[brown] at (0,-1.5) {$\supp(\chi_{\II})$};
\node[brown] at (-7,-2) {$\supp(\chi_{\I})$};
\draw[thick] (-8,6) .. controls (-4.5,2.4) and (-4.2, 1.5) .. (-4,0);
\node at (-6.5,4) {$\BB$};
\draw[blue,thick] (-8,6.2) .. controls (-5,2.8) and (-4.6, 2.35) .. (-4.2,2.4);
\draw[red] (8,6.2) .. controls (4.6,2.8) and (4.2, 2.35) .. (2,2.4);
\node[red] at (6,5) {$\Ht=\Ht_\BB$};
\draw[blue,thick] (-2,2.4) -- (-4.2,2.4);
\draw[red] (-2,2.4) -- (2,2.4);
\draw[blue,thick] (-2,2.4) .. controls (-1.5,2.4) and (-.5,1.8) .. (0,1.2);
\draw[blue,thick] (2,0) .. controls (1.5,0) and (0.5,0.6) .. (0,1.2);
\draw[blue,thick] (2,0) -- (7,0);
\node[blue,thick] at (-3,2.8) {$\Sigma$};
\end{tikzpicture}
\caption{The Cauchy surface $\Sigma$ in $\SSS_* $.}\label{fig:sigma}
\end{figure}
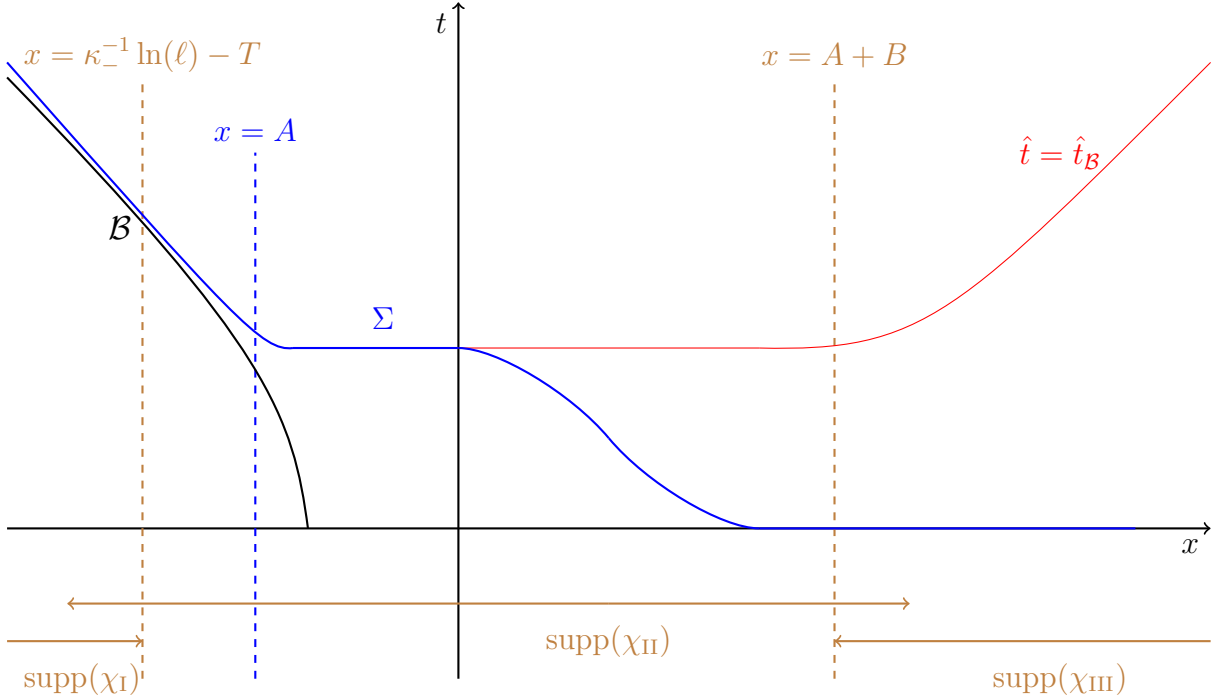
\end{center}

Since $J^+(\Sigma) \cap \BB = \emptyset$ $\Ht=t$ on $\supp(u_0,u_1)$ we can write
\begin{equation*}
U_\BB(0,T) (u_0,u_1) = U_\BB(0,\Sigma) U_\BB(\Sigma,T)(u_0,u_1) = U_\BB(0,\Sigma) U(\Sigma,\HT)(u_0,u_1).
\end{equation*}
Here $\HT$ denotes the Cauchy surface $\Ht = T$. Consequently $U_\BB(0,T) (u_0,u_1)$ splits into three terms:
\begin{equation*}\begin{gathered}
U_\BB(0,T) (u_0,u_1) =  U_\BB(0,\Sigma) \chi_\I U(\Sigma,\HT)(u_0,u_1) \\
  + U_\BB(0,\Sigma) \chi_{\II} U(\Sigma,\HT) (u_0,u_1) 
 + U_\BB(0,\Sigma) \chi_{\III} U(\Sigma,\HT)(u_0,u_1).
 \end{gathered}
\end{equation*}
Below we study separately each of the terms in the RHS.

\subsubsection{Study of the third term.}
Start with $U_\BB(0,\Sigma) \chi_{\III} U(\Sigma,\HT)(u_0,u_1)$. On $\supp(\chi_{\III})$, $\Sigma$ is contained in $\{ t=0 \}$. Thus $U_\BB(0,\Sigma)\chi_{\III} = \chi_{\III}$ and $\chi_{\III} U(\Sigma,T) = \chi_{\III} U(0,T)$. It follows that
\begin{equation*}
U_\BB(0,\Sigma) \chi_{\III} U(\Sigma,T)(u_0,u_1) = \chi_{\III} U(0,\HT) (u_0,u_1).
\end{equation*} 
No boundary is involved in this first term. The results of \S \ref{subsec:1a} apply with remainder term small in the space $C^\infty$ and there is no complication due to the change of time-foliation. Proposition \ref{prop:scatt2} implies
\begin{equation}\label{eq:3m}\begin{gathered}
\left(\chi_{\III} U(0,\HT) (u_0,u_1)\right)(x,\w) = \chi_{\III}(x) u^*_+(x-T,\w) + O_{C^\infty}(e^{-cT}) \\ 
 = u^*_+(x-T,\w) + O_{C^\infty}(e^{-cT}).
 \end{gathered}
\end{equation}

\subsubsection{Study of the second term.}\label{sub:4}
Here we estimate the term $U_\BB(0,\Sigma) \chi_{\II} U(\Sigma,\HT)(u_0,u_1)$. On $\supp(\chi_\II) \cap \Sigma$ we have $\Ht \leq \Ht_\BB$ and $r_+-\delta_0 \geq r, \ r-r_- \gtrsim \ell h$ for some $\delta_0 > 0$. Since $h = e^{-\kappa_- T}$ we have along $\Sigma$
\begin{equation*}
-T+\Ht +2 F_K'(r) \leq - T - \kappa_-^{-1}\ln |r-r_-| + C \leq -\kappa_-^{-1} \ln(\ell) + C.
\end{equation*}
Proposition \ref{prop:res2} implies that $\chi_{\II} U(\Sigma,\HT)(u_0,u_1) =  O(\ell^{-c_0})$. Along $\Sigma \cap \supp(\chi_\II)$ is at finite distance from $\{t=0\}$; consequently the operator $U_\BB(0,\Sigma)\chi_{\II}$ is continuous on $C^\infty$. It yields
\begin{equation}\label{eq:3n}
U_\BB(0,\Sigma) \chi_{\II} U_\BB(\Sigma,T)(u_0,u_1) =  O_{C^\infty}(\ell^{-c_0}).
\end{equation}

\subsubsection{Study of the first term.}\label{sub:5}
Here we study the term $U_\BB(0,\Sigma) \chi_\I U(\Sigma,\HT)(u_0,u_1)$. Note that $\Sigma \cap \supp(\chi_\I) \subset \{ \Ht = \Ht_\BB \}$. This together with \eqref{eq:hg3} implies
\begin{equation*}
J^-(\supp(\chi_\I) \cap \Sigma) \cap \{ t= 0 \} \cap \{ x \geq z_*(0) \} \subset \{ 0 \}  \times [z_*(0),1] \times \Ss \subset \{\Ht=0\}.
\end{equation*}
This yields
\begin{equation*}
U_\BB(0,\Sigma) \chi_\I U(\Sigma,\HT)(u_0,u_1) = U_\BB(\hat{0},\Ht_\BB) \chi_\I U(\Ht_\BB,\HT)(u_0,u_1).
\end{equation*}

Lemma \ref{lem:lalala} and Lemma \ref{lem:gcS} imply that 
\begin{equation*}
\left(\chi_\I U(\Ht_\BB,\HT)(u_0,u_1)\right)(x,\w) = \left(\chi_I(x) u_-\left( \kappa_-^{-1}\ln\left( \dfrac{\Hr}{\lambda_K'(r_-)h} \right),\w \right), 0\right) + \chi_\I O_{H^{1/2} \oplus H^{-1/2}}(e^{-cT}).
\end{equation*}
By Proposition \ref{prop:nrj} the operator $U_\BB(\hat{0},\Ht_\BB)$ is continuous from $H^{1/2} \oplus H^{-1/2}$ to itself. Thus
\begin{equation*}\begin{gathered}
U_\BB(\hat{0},\Ht_\BB) \chi_\I U(\Ht_\BB,\HT)(u_0,u_1) = \\ U_\BB(\hat{0},\Ht_\BB) \left(\chi_I(x) u_-\left( \kappa_-^{-1}\ln\left( \dfrac{\Hr}{\lambda_K'(r_-)h} \right),\w \right), 0\right) + O_{H^{1/2} \oplus H^{-1/2}}(e^{-cT}).
\end{gathered}
\end{equation*} 
Because of \eqref{eq:4g} the function
\begin{equation*}
(x,\w) \mapsto \chi_\I(x) u_-\left( \kappa_-^{-1}\ln\left( \dfrac{x}{\lambda_K'(r_-)} \right),\w \right)
\end{equation*}
satisfies the hypothesis \ref{eq:3h}. The support of $U_\BB(\hat{0},\Ht_\BB) \chi_\I U(\Ht_\BB,\HT)(u_0,u_1)$ is contained in $[z_*(0),1] \times \Ss$ thus $F_K(r) = 0$ on this set. Since $\Hr = x+F_K(r)-\Ht_\BB$ this implies $\Hr = x -\Ht_\BB$ on $[z_*(0),1] \times \Ss$. Proposition \ref{prop:jG4} yields
\begin{equation*}\begin{gathered}
\left(U_\BB(\hat{0},\Ht_\BB) \chi_\I U(\Ht_\BB,\HT)(u_0,u_1)\right)(x,\w) = \\
\left(v_\ap\left(\dfrac{\beta_0 x}{h},\w\right), \ \p_x \left(v_\ap\left(\dfrac{\beta_0 x}{h},\w\right)\right)\right) + O_{H^{1/2} \oplus H^{-1/2}} (\ell^{8/3} h^{2/3}) + O_{H^{1/2} \oplus H^{-1/2}}(h^c), \\
v_\ap(\beta_0 x/h,\w) = \chi\left(\dfrac{x}{h\ell}\right) u_-\left( \kappa_-^{-1}\ln\left( \dfrac{\beta_0 x}{h}\right), \w\right).
\end{gathered}
\end{equation*}
Rigorously speaking here the Sobolev norms are measured with respect to $\HHS$ but as $\p_x \Hr$ is bounded from above and below on $[z_*(0),1]$ we can also measure them with respect to $\SSS_*$. We now remove the term $\chi(x/(h\ell))$ in the expression of $v_\ap$. The function $(x,\w) \mapsto u_-(\kappa_-^{-1} \ln(\beta_0 x),\w)$ belongs to the symbol class $S^{-\delta}$ for some $\delta > 0$.  Thus Lemma \ref{lem:1g} in the appendix implies
\begin{equation*}
 \chi_0(x)\rho\left(\dfrac{x}{h \ell}\right) u_-\left( \kappa_-^{-1}\ln\left( \dfrac{\beta_0 x}{h}\right), \w\right) = \chi_0(x) u_-\left( \kappa_-^{-1}\ln\left( \dfrac{\beta_0x}{h}\right), \w\right) + O_{H^{1/2}}(h^{\delta/2} + \ell^{-\delta/2})
\end{equation*} 
where $\rho=1-\chi$ and $\chi_0$ is supported on a small neighborhood of $[z_*(0),1]$ and equal to $1$ on $[z_*(0),1]$. Since $u_- \in X_\scatt$ it vanishes for small values of $x$, leading to
\begin{equation*}
u_-\left( \kappa_-^{-1}\ln\left( \dfrac{\beta_0x}{h}\right), \w\right) = 0, \ x \geq \dfrac{1}{2}, \ h \text{ small enough.}
\end{equation*}
It follows that on $\Sigma_0 = [z_*(0),\infty) \times \Ss$,
\begin{equation*}
 \chi_0(x)\rho\left(\dfrac{x}{h \ell}\right) u_-\left( \kappa_-^{-1}\ln\left( \dfrac{\beta_0 x}{h}\right), \w\right) =  u_-\left( \kappa_-^{-1}\ln\left( \dfrac{\beta_0x}{h}\right), \w\right) + O_{H^{1/2}}(h^{\delta/2} + \ell^{-\delta/2}).
\end{equation*}
A similar identity holds for the time derivative -- we skip the details:
\begin{equation*}
\chi_0(x)\rho\left(\dfrac{x}{h \ell}\right) \p_x \left(u_-\left( \kappa_-^{-1}\ln\left( \dfrac{\beta_0 x}{h}\right), \w\right)\right) =  \p_x \left(u_-\left( \kappa_-^{-1}\ln\left( \dfrac{\beta_0x}{h}\right), \w\right)\right) + O_{H^{-1/2}}(h^{\delta/2} + \ell^{-\delta/2}).
\end{equation*}
Since the results of Proposition \ref{prop:jG4} are uniform in $\ell$ and $h$ we can now chose $\ell$ depending on $h$. Take $\ell = h^{-\eta}$ where $\eta$ is a small positive exponent so that $\ell^{8/3} h^{2/3} = \ell^{-\delta/2}$. It yields
\begin{equation}\label{eq:4d}\begin{gathered}
\left(U_\BB(\hat{0},\Ht_\BB) \chi_\I U(\Ht_\BB,\HT)(u_0,u_1)\right)(x,\w) = \\
\left(u_-\left( \kappa_-^{-1}\ln\left( \dfrac{\beta_0 x}{h}\right), \w\right), \dfrac{1}{\kappa_- x} (\p_xu_-)\left( \kappa_-^{-1}\ln\left( \dfrac{\beta_0 x}{h}\right), \w\right) \right) + O_{H^{1/2} \oplus H^{-1/2}}(h^c).\end{gathered}
\end{equation}
Switching $u_-$ to the scattering field $u^*_-$ constructed in Proposition \ref{prop:scatt2} according to $u^*_-(x,\w) = u_-(x-\lambda_K(r_-),\w)$ simply changes to the value of $\beta_0 < 0$ to some $\gamma_0 < 0$. With $\ell = h^{-\eta}$ \eqref{eq:3n} becomes 
\begin{equation*}
U_\BB(0,\Sigma) \chi_{\II} U_\BB(\Sigma,T)(u_0,u_1) =  O_{C^\infty}(h^c).
\end{equation*}
This estimate together with \eqref{eq:3m} and \eqref{eq:4d} completes the proof of Theorem \ref{thm:asymp}.

\subsection{Comments.}
Theorem \ref{thm:asymp} is one of the main novelties of this work. The proof relies exclusively on geometric arguments: conformal scattering for times $\Ht \in [\Ht_\BB,T]$, construction of a radial coordinate following the reflection of the singularity for $\Ht \in [0,\Ht_\BB]$, and WKB approximation. One can try to apply these arguments on non-spherical spacetimes but in such cases the solution $\Hr$ of \eqref{eq:h4e} might not be global. For small perturbation of SdS one should still be able to prove that $\Hr$ is globally defined.

In \cite{Bachelot1} the author works exclusively in $\SSS_*$. In this system of coordinate the reflection of the KG field takes place in the time-zone $t \in [0,T/2]$. Using $\HHS$ we needed to apply a WKB parametrix only on a bounded time-zone while in Bachelot's case a similar analysis must be realized on a time-zone of length $T \sim -\ln(h)$ (where $h^{-1}=e^{\kappa_-T}$ is the order of the frequency). This scale of time is known as the Erhenfest time. Local WKB paramtrices do not give good results in Bachelot's situation and even global theorems like \cite[Theorem $11.12$]{Zworski} cannot be applied.

\section{Proof of Theorem \ref{thm:haw}}\label{sec:5}
In this section we prove Theorem \ref{thm:haw}.

\subsection{Preliminaries}\label{sub:9u} In the coordinate system $\SSS_*$ the Klein-Gordon operator $\square + m^2$ is given by
\begin{equation}\label{eq:defH}\begin{gathered}
\square + m^2 = \dfrac{1}{\Delta_r} \dd{^2}{t^2} - \dfrac{1}{r^2 \Delta_r} \dd{}{x} r^2 \dd{}{x} + \dfrac{\Delta_{\Ss}}{r^2}+m^2 = \Delta_r^{-1}(\p_t^2 + r^{-1} \Ha r), \\
\Ha = D_x^2 + \dfrac{\Delta_r \Delta_{\Ss}}{r^2}+V, \ \ \ V=  \dfrac{(\p_x^2 r)}{r} + m^2 \Delta_r.\end{gathered}
\end{equation}
For $t \in [0,\infty]$ we introduce the spaces
\begin{equation*}\begin{gathered}
\HH_t^0  = L^2([z_*(t),\infty) \times \Ss, r^2 dx d\w), \\
\HH_t^2  = \{ u \in \HH_t^0, \ \Ha u \in \HH_t^0, \ u(z_*(t),\w) = 0 \},\end{gathered}
\end{equation*}
where by convention $z_*(\infty) = -\infty$ and there is no boundary condition in the definition of $\HH_\infty^2$. The operator $\Ha_t$ formally equal to $\Ha$ but with domain $\HH_t^2$ is selfadjoint on $\HH_t^0$ with spectrum in $[0,\infty)$. For $s \in \R$ we define $\HH^{2s}_t$ the domain of the operator $\Ha_t^s$. We provide this space with the norm $|\cdot|_{\HH_t^{2s}} = |\Ha^s_t \cdot|_{\HH^0_t}$. For technical reasons we concentrate on the case $V > 0$, where $V$ is the potential given in \eqref{eq:defH}. The following lemma shows that this is realized as long as $\Lambda$ is small enough.

\begin{lem}\label{lem:1b} Fix $m, M > 0$. There exists $\Lambda_0, C$ such that for all $0 < \Lambda \leq \Lambda_0$, for all $r \in (r_-,r_+)$, $V \geq C \Delta_r$.
\end{lem}

\begin{proof} Note that 
\begin{equation*}
V = \Delta_r \left( \dfrac{1}{r^3} \p_r \dfrac{\Delta_r}{r^2} +m^2 \right) = \Delta_r \left( - \dfrac{2\Lambda}{3r^2} +\dfrac{2M}{r^5} +m^2 \right)\geq \Delta_r \left(- \dfrac{2\Lambda}{3r_-^2} +m^2\right).
\end{equation*} 
See $r_-$ -- the smallest positive root of $\Delta_r$ -- as a function of $\Lambda$. As $\Lambda \rightarrow 0$ $\Delta_r/r$ converges uniformly to $r - 2M$. By Hurwitz's theorem, two of the roots of $\Delta_r/r$ escape any compact subset of $\R$ and the last one converges to $2M$. We also know that $\Delta_r/r$ has one negative root and two positive roots $r_- < r_+$. Since $2M > 0$ the negative root of $\Delta_r/r$ must diverge to $-\infty$ and the largest positive root must diverge to $+\infty$. By elimination $r_- \rightarrow 2M$. Therefore 
\begin{equation*}
V \geq \Delta_r \left( - \dfrac{2\Lambda}{12M^2} + m^2 +o(\Lambda)\right) \geq \dfrac{1}{2} m^2 \Delta_r
\end{equation*}
for $\Lambda$ small enough. \end{proof}

In the rest of the paper we work for $\Lambda$ in the range $(0,\Lambda_0]$. Hence $V \geq C e^{-\kappa_+x}$ for every $x \geq z_*(0)$. To simplify notation we write $\HH = \HH_0^0$, $\lr{\cdot, \cdot} = \lr{\cdot,\cdot}_{\HH_0^0}$ and $|\cdot| = |\cdot|_{\HH_0^0}$. Let $\beta>0$ and $\varphi, \psi$ be given by
\begin{equation}\label{eq:2r}
\varphi(z) = z^{1/2} \coth(\beta z^{1/2}), \ \ \ \psi(z) = z^{-1/2} \coth(\beta z^{1/2}).
\end{equation}
The function $\varphi$ (resp. $\psi$) is holomorphic (resp. meromorphic with a simple pole at $0$) in a neighborhood of $[0,\infty)$. To give a mathematical formulation of the Hawking radiation we need to study the operators $\varphi(\Ha_0)$ and $\psi(\Ha_0)$ defined by the spectral theorem -- see \S \ref{sub:func}. Below we show that the domains of $\varphi(\Ha_0), \psi(\Ha_0)$ contain the space 
\begin{equation*}
X_0 = \{u \in C^\infty_0(\Sigma_0), \ \ u(z_*(0),\w) = 0\}.
\end{equation*}
We will use the Lowner-Heinz inequality which we recall here. If $(B,D(B))$ is a selfadjoint operator and $(A,D(A))$ is a symmetric positive operator with $D(B) \subset D(A)$. If $A \leq B$ on $D(B)$ then for all $0 \leq p \leq 1$, $A^p \leq B^p$, and for all $-1 \leq p \leq 0$, $B^p \leq A^p$.

\begin{lem}\label{lem:1f} Let $s \in [1/4,1]$ and $F$ be a continuous function. Assume that for every $z \in [0,\infty)$ we have $F(z) \leq c \lr{z}^s$. Then the domain of the selfadjoint operator $F(\Ha_0)$ defined by the spectral contains the set $X_0$ and for every $v \in X_0$,
\begin{equation*}
|F(\Ha_0)v| \leq C |v|_{H^{2s}(\Sigma_0)}.
\end{equation*}
\end{lem}

\begin{proof} Since $F(z) \leq c \lr{z}^s$, $\DD(\Ha_0^s) \subset \DD(F(\Ha_0))$. In addition the operator $L_0 = D_{x,0}^2 + \lr{x}^{-2} \Delta_\Ss +1$ is selfadjoint on $\HH$ with domain 
\begin{equation*}
\DD(L_0) = \{ v \in H^2([z_*(0),\infty) \times \Ss), v(z_*(0),\w) = 0\} \subset \HH_0^2 = \DD(\Ha_0).
\end{equation*}
Moreover on $\DD(L_0)$ we have $\Ha_0 \leq L_0$. We now apply the Lowner-Heinz theorem. The operator $L_0^s$ has the domain
\begin{equation*}
\DD(L_0^s) = \{ v \in H^{2s}(\Sigma_0), \ v(z_*(0),\w) = 0\}. 
\end{equation*}
Consequently $X_0 \subset \DD(L_0^s) \subset \DD(F(\Ha_0))$ and for every $v \in X_0$,
\begin{equation*}
|F(\Ha_0) v| \lesssim |L_0 v| \lesssim |v|_{H^{2s}}.
\end{equation*}
This ends the proof of the lemma.\end{proof}

Applying this result to $F = \varphi$ and $s=1$ shows that the domain of $\varphi(\Ha_0)$ contains the space $X_0$. Now comes the more difficult task of proving that the domain of $\psi(\Ha_0)$ contains $X_0$. The function $\psi$ has a pole at $0$. Since
\begin{equation*}
\coth(x) = \sum_{k = -\infty}^\infty \dfrac{x}{x^2+k^2 \pi^2},
\end{equation*}
we have
\begin{equation}\label{eq:4h}
\psi(z) = \dfrac{1}{\beta z} + \phi(z), \ \ \ \phi(z) = \dfrac{1}{\beta}\sum_{k \in \Vv \setminus \{0\} } \dfrac{1}{z + k^2}
\end{equation}
where $\Vv = \pi \beta^{-1} \cdot \Z$. The series
\begin{equation*}
\phi(\Ha_0) = \beta^{-1} \sum_{k \in \Vv \setminus \{0\}} ( \Ha_0 + k^2)^{-1}
\end{equation*}
converges for the topology of bounded operators on $\HH$ as
\begin{equation*}
\left| (\Ha_0+k^2)^{-1} \right|_{\BB(\HH)} \leq \dfrac{1}{d(\sigma(\Ha_0), -k^2)} \lesssim k^{-2}.
\end{equation*}
The operator $\phi(\Ha_0)$ is thus a bounded operator. In order to prove that $X_0$ is contained in $\DD(\psi(\Ha_0))$ it is sufficient to show that $X_0$ is contained in $\DD(\Ha_0^{-1})$. Let $D_{x,0}^2$ be the unbounded operator formally equal to $D_x^2$ but with domain
\begin{equation*}
\DD(D_{x,0}^2) = \{u \in \HH, D_x^2u \in \HH, \ u(z_*(0),\w)=0 \}.
\end{equation*}
It is a selfadjoint operator on $\HH$. In the following lemma we give some estimates on the resolvent of this operator.

\begin{lem}\label{lem:1c} Let $\az > 0$. The operator
\begin{equation}\label{eq:1u}
e^{-\az x} (D_{x,0}^2-\lambda^2)^{-1} e^{-\az x} : \HH \rightarrow \HH
\end{equation}
well defined for $\Ime \lambda > 0$, continues to a holomorphic family of operators for $\Ime \lambda > -\az/2$. Moreover for every $v \in X_0$,
\begin{equation}\label{eq:1v}
\lr{D_{x,0}^{-2} v,v} \leq C \int_\Ss \left|\int_{z_*(0)}^\infty \lr{x} |v(x,\w)| dx \right|^2 d\w.
\end{equation}
\end{lem}

\begin{proof} For $\Ime \lambda >0$ the kernel of the one-dimensional operator $R_0(\lambda) = D_{x,0}^2 - \lambda^2$ is the holomorphic function of $\lambda$ given by 
\begin{equation}\label{eq:kerne}
K_\lambda(x,y) = i\dfrac{e^{i\lambda|x-y|} - e^{i\lambda|x+y-2z_*(0)|}}{2\lambda}
\end{equation}
-- see \cite{Simon}. It extends to an entire function on $\C$; in particular
\begin{equation*}
\lim_{\lambda \rightarrow 0} K_\lambda(x,y) = \min(x,y) - z_*(0).
\end{equation*}
Equation \eqref{eq:1u} follows then from Schur's lemma. For \eqref{eq:1v} we note that
\begin{equation*}
\begin{gathered}
\lr{D_{x,0}^{-2} v, v} = \int_{\Sigma_0} \int_{z_*(0)}^\infty \left( \min(x,y) - z_*(0) \right) v(y,\w) v(x,\w) dy dx d\w \\
         \lesssim \int_\Ss \int_{z_*(0)}^\infty \int_{z_*(0)}^\infty \lr{x} \lr{y} |v(x,\w)| |v(y,\w)| dy dx d\w 
        \lesssim \int_\Ss \left| \int_{z_*(0)}^\infty \lr{x} |v(x,\w)| dx \right|^2 d\w.
        \end{gathered}
\end{equation*}
This ends the proof. \end{proof}

We now give a sense to $\Ha_0^{-1}$ on $X_0$. By the same analysis as in the proof of \cite[Proposition 2.1]{SaZw}, the operator
\begin{equation*}
 (\Ha_0-\lambda^2)^{-1} : \HH \rightarrow \HH
\end{equation*}
initially defined for $\Ime \lambda > 0$, admits a meromorphic continuation to $\C$ as an operator $\HH_\comp \rightarrow \HH_\loc$. Its poles are called resonances.

\begin{lem} The operator $\Ha_0$ has no resonance at $0$ thus $\Ha_0^{-1}$ is well defined on $X_0$. In addition,
\begin{enumerate}
\item[$(i)$] For every $R$ there exists $C_R$ such that for all $v \in X_0$ supported in $[z_*(0),R] \times \Ss$,
\begin{equation}\label{eq:1l}
\lr{\Ha_0^{-1}v,v} \leq C_R |v|_{H^{-1}(\Sigma_0)}^2.
\end{equation}
\item[$(ii)$] There exists $C$ such that for every $v \in X_0$,
\begin{equation}\label{eq:1m}
\lr{\Ha_0^{-1}v,v} \leq C \int_\Ss \left| \int_{z_*(0)}^\infty \lr{x} |v(x,\w)| dx \right|^2 d\w.
\end{equation}
\end{enumerate}
\end{lem}

\begin{proof} We start by proving that $\Ha_0$ has no resonance at $0$. Let $\Ha_{0,\ell}$ be the selfadjoint operator 
\begin{equation*}
\Ha_{0,\ell} = D_{x,0}^2 + V + \dfrac{\Delta_r}{r^2}\ell(\ell+1)
\end{equation*}
on $L^2([z_*(0),\infty),dx)$ with domain $H^2_0([z_*(0),\infty))$. This is a $1$D Schr\"odinger operator with a positive potential. The family  of operators $V^{1/2}(\Ha_{0,\ell}-\lambda^2)^{-1}V^{1/2}$ initially defined for $\Ime \lambda > 0$ extends to a meromorphic family on $\C$, see \cite{Simon}. The poles of this meromorphic continuation are exactly the zero of the Fredholm determinant
\begin{equation*}
d(\lambda) = \det(\Id + V^{1/2} R_0(\lambda) V^{1/2}).
\end{equation*}
By \cite[Theorem 2]{Simon} this determinant admits  an explicit expansion of the form
\begin{equation}
d(\lambda) = 1 + \sum_{n=1}^\infty d_n(\lambda).
\end{equation}
Each of the coefficients satisfy $d_n(0) \geq 0$. In particular $d(0) \geq 1$ and $\Ha_{0,\ell}$ has no resonance at $0$. Assume now that $\Ha_0$ has a resonance at $0$ and call $m$ its multiplicity. Then there exists a family of operators $\lambda \mapsto A(\lambda)$ that is holomorphic near $0$ and and $\Pi$ an operator of finite rank such that near $0$,
\begin{equation*}
\lambda^{m-1}(\Ha_0-\lambda^2)^{-1} = \dfrac{\Pi}{\lambda} + A(\lambda).
\end{equation*} 
Define $\pi_\ell$ the orthogonal projector on the eigenspace $\ker(\Delta_\Ss-\ell(\ell+1))$. If $\Ime\lambda \gg 1$ and $\ell \neq \ell'$ then
\begin{equation*}
\pi_\ell (\Ha_0-\lambda^2)^{-1} \pi_{\ell'} = 0,
\end{equation*}
which implies by meromorphic continuation that $\pi_{\ell} \Pi \pi_{\ell'} = 0$. If $v \in \HH_\comp$ and $v_\ell$ is the function given by $v_\ell = \pi_\ell v$ then
\begin{equation*}
 \pi_{\ell} (\Ha_0-\lambda^2)^{-1} \pi_{\ell} v =  (\Ha_{0,\ell}-\lambda^2)^{-1}v_{\ell},  \ \ \Ime \lambda \gg 1.
\end{equation*}
By meromorphic continuation this yields
\begin{equation*}
\left(\pi_{\ell} \dfrac{\Pi}{\lambda}\pi_{\ell}  + \pi_{\ell} A(\lambda) \pi_{\ell}\right) v =  \lambda^{m-1}  (\Ha_{0,\ell}-\lambda^2)^{-1} v_{\ell},  \ \ \lambda \text{ near } 0.
\end{equation*}
As $0$ is not a resonance of $\Ha_{0,\ell}$ both sides must be holomorphic and thus $\pi_{\ell} \Pi \pi_{\ell} = 0$. Thus $\Pi=0$ and $\Ha_0$ has no resonance at $\lambda=0$.

Now we prove  $(i)$. Let $v \in X_0$ with support in $[z_*(0),R] \times \Ss$. Write $v = \sum_{\ell} v_\ell$ where $\pi_\ell v=v_\ell$. By the Lowner-Heinz inequality we have $(\Ha_{0,\ell}+\epsilon)^{-1} \leq (\Delta_r(\Delta_\Ss+c)+\epsilon)^{-1} \leq \Delta_r^{-1}(\Delta_\Ss+c)^{-1}$ for $c$ small enough and $\epsilon > 0$. The term $\Delta_r^{-1}$ is uniformly bounded for $x \in [z_*(0),R]$ by a constant $C_R$. Thus
\begin{equation*}\begin{gathered}
\lr{(\Ha_0+\epsilon)^{-1} v, v}   = \sum_{\ell=0}^\infty \lr{(\Ha_{0,\ell}+\epsilon)^{-1} v_\ell, v_\ell}   \leq \sum_{\ell=0}^\infty \lr{(\Delta_\Ss+c)^{-1} v_\ell, \Delta_r^{-1} v_\ell} \\
 \leq C_R \sum_{\ell=0}^\infty \dfrac{|v_\ell|^2}{\ell(\ell+1)+c}  = C_R \lr{(\Delta_\Ss+c)^{-1}v,v}.\end{gathered}
\end{equation*}
In addition by the Lowner-Heinz inequality, $\lr{(\Ha_0+\epsilon)^{-1} v, v} \leq \lr{(D_{x,0}^2+\epsilon)^{-1}v,v}$. The operators $\Delta_\Ss$ and $D_{x,0}^2$ commute. The inequality
\begin{equation*}
\min(\dfrac{1}{\xi^2+\epsilon}, \dfrac{1}{\ell(\ell+1)+c}) = \dfrac{1}{\max(\xi^2+\epsilon, \ell(\ell+1)+c)} \leq \dfrac{2}{\xi^2 + \ell(\ell+1)+c+\epsilon}
\end{equation*}
implies then
\begin{equation*}\begin{gathered}
\lr{(\Ha_0+\epsilon)^{-1} v, v} \leq C_R \min\left( \lr{(\Delta_\Ss+c)^{-1}v,v}, \lr{(D_{x,0}^2+\epsilon)^{-1}v,v}\right)  \\
    \leq 2C_R \lr{(D_{x,0}^2 + \Delta_\Ss+c+\epsilon)^{-1}v,v}.\end{gathered}
\end{equation*} 
Make $\epsilon \rightarrow 0$ to conclude: $\lr{(\Ha_0+\epsilon)^{-1} v, v} \leq C_R|v|^2_{H^{-1}(\Sigma_0)}$.

We now prove \eqref{eq:1m}. By the Lowner-Heinz inequality if $v \in C_0^\infty(\Sigma_0)$ then
\begin{equation*}\begin{gathered}
\lr{\Ha_0^{-1}v,v}   = \lim_{\epsi \rightarrow 0^+} \lr{(\Ha_0+\epsi)^{-1}v,v} \\
      \leq \lim_{\epsi \rightarrow 0^+} \lr{(D_{x,0}^2+\epsi)^{-1}v,v}     = \lr{D_{x,0}^{-2}v,v}.\end{gathered}
\end{equation*}
Equation \eqref{eq:1m} follows then from \eqref{eq:1v}.
\end{proof}

Note that although we used a decomposition in spherical harmonics in the proof this result is still stable under suitable non-spherical perturbations of the metric. Indeed the meromorphic continuation of the resolvent still holds if the perturbation preserves the structure of the horizons $r=r_\pm$ (see \cite{SaZw}). The resonances are stable under small metric perturbations. The estimates $(i)$ and $(ii)$ will still work for small perturbations of $\Ha_0$ using a comparison operator that is spherically symmetric. We now conclude with an estimate for the operator $\psi(\Ha_0)$. 

\begin{lem}\label{lem:9q} For every $R>0$ there exists $C_R$ such that for every $v \in X_0$ with support in $[z_*(0),R] \times \Ss$
\begin{equation*}
\lr{\psi(\Ha_0)v,v} \leq C_R|v|_{H^{-1/2}(\Sigma_0)}^2.
\end{equation*}
\end{lem}

\begin{proof} By reproducing the proof of \eqref{eq:1l}, $\lr{(\Ha_0+1)^{-1/2}v,v} \leq C_R\lr{(\Delta_\Ss+c)^{-1/2}v,v}$. Similarly $\lr{(\Ha_0+1)^{-1/2}v,v} \leq \lr{(D_{x,0}^2+1)^{-1/2}v,v}$. The operators $D_{x,0}^2$ and $\Delta_\Ss$ commute. In addition,
\begin{equation*}
\min\left( \dfrac{1}{\xi^2+1}, \dfrac{1}{\ell(\ell+1)+c}\right) \leq \dfrac{2}{\xi^2 + \ell(\ell+1)+c+1}.
\end{equation*}
It follows that
\begin{align*}
\lr{(\Ha_0+1)^{-1/2}v,v} \leq C_R |v|_{H^{-1/2}(\Sigma_0)}^2.
\end{align*}
In the notation of \eqref{eq:4h} the function $\phi$ satisfies $\phi(z) \leq C(z+1)^{-1/2}, \ z \geq 0$. By the spectral theorem 
\begin{equation*}\begin{gathered}
\lr{\psi(\Ha_0) v,v} = \lr{\Ha_0^{-1}v,v} + \lr{\phi(\Ha_0)v,v}  \leq \lr{\Ha_0^{-1}v,v}+ c\lr{(1+\Ha_0)^{-1/2}v,v} \\
     \leq C_R |v|_{H^{-1}(\Sigma_0)}^2 + C_R |v|_{H^{-1/2}(\Sigma_0)}^2 \leq C_R |v|_{H^{-1/2}(\Sigma_0)}^2.\end{gathered}
\end{equation*}
In the second line we used \eqref{eq:1l} and the embedding $H^{-1/2}(\Sigma_0) \hookrightarrow H^{-1}(\Sigma_0)$. This completes the proof.\end{proof}


\subsection{Quantization process}\label{sub:func} We recall here the quantization process of bosonic scalar quantum fields as done in \cite{Bachelot2}, \cite{Bachelot1}.

\subsubsection{Principle of algebraic quantum field theory} Let us recall briefly the fundamental principles of algebraic quantum field theory. For a more complete introduction see \cite{Bachelot2} and \cite{BraRob}. Given a Hilbert space $\mathfrak{H}$ a Weyl quantization of $\mathfrak{H}$ is a pair $(\FF(\mathfrak{H}), \Mmm)$ with the following:
\begin{itemize}
\item[$(i)$] $\FF(\mathfrak{H})$ is a Hilbert space,
\item[$(ii)$] $\Mmm$ is a map from $\mathfrak{H}$ to the space of unitary operators on $\FF(\mathfrak{H})$ satisfying the canonical commutation relation (CCR)
\begin{equation*}
\forall f,g \in \mathfrak{H}, \ \ \ \ \Mmm(f+g) = e^{-i\text{Im}\lr{f,g}_\mathfrak{H}} \Mmm(f) \Mmm(g),
\end{equation*}
\item[$(iii)$] The restriction of $\Mmm$ to any finite dimensional subspace is continuous.
\end{itemize}
The algebra of observable $\AAA(\FF(\mathfrak{H}), \Mmm)$ is the $\C^*$-algebra generated by the elements $\Mmm(f), f \in \mathfrak{H}$. It is unique in the sense given by the Von-Neumann uniqueness theorem and we write $\AAA(\mathfrak{H}) = \AAA(\FF(\mathfrak{H}),\Mmm)$. A state is a positive normalized linear form $\w$ on $\AAA(\mathfrak{H})$. The generating functional of $\w$ is the map $\Ee_\w : \mathfrak{H} \rightarrow \C$ such that 
\begin{equation*}
\forall f \in \mathfrak{H}, \ \ \ \ \Ee_\w(f) = \w(\Mmm(f)).
\end{equation*}
It satisfies the three following conditions:
\begin{enumerate}
\item[$(i)$] $\Ee(0)=1$,
\item[$(ii)$] The restriction of $\Ee$ to any finite dimensional subspace of $\mathfrak{H}$ is continuous,
\item[$(iii)$] For every $\lambda_j \in \C, f_j \in \mathfrak{H}$,
\begin{equation*}
\sum_{j,k=1}^n \lambda_j \overline{\lambda_k} \Ee(f_j-f_k) e^{-2i\text{Im}\lr{f_j,f_k}_\mathfrak{H}} \geq 0.
\end{equation*}
\end{enumerate}
Now given $\Ee : \mathfrak{H} \rightarrow \C$ satisfying $(i), (ii), (iii)$ there exists a unique state $\w$ and a unique Weyl quantization $(\FF(\mathfrak{H}),\Mmm)$ of $\mathfrak{H}$ and cyclic vector $F \in \FF(\mathfrak{H})$ with
\begin{equation*}
\Ee(f) = \lr{F, \Mmm(f) F}_{\FF(\mathfrak{H})} = \w(\Mmm(f)).
\end{equation*}
 
We now study the dynamic of states. Let $\mathfrak{H}_t$ be a family of Hilbert space provided with a family of symplectic propagator $\UU(t,t') : \mathfrak{H}_{t'} \rightarrow \mathfrak{H}_t$. Let $(\FF(\mathfrak{H}_0),\Mmm_0)$ be a quantization of $\mathfrak{H}_0$ and $\AAA(\mathfrak{H}_0)$ be the algebra of observable. The propagator $\UU(t,t')$ induces a quantization of $\mathfrak{H}_t$ with algebra of observable generated by the operators
\begin{equation*}
\Mmm_t(f) = \Mmm_0(\UU(0,t) f), \ \ f \in \mathfrak{H}_t.
\end{equation*}
Consequently if $\w_0$ is a state on $\mathfrak{H}_0$ we define the evolution of $\w_0$ by the map $t \mapsto \w_t$ with $\w_t(\Mmm_t(f)) = \w_0(\Mmm_0(\UU(0,t) f))$, $f \in \mathfrak{H}_t$. Similarly if $\Ee_0$ is a generating functional we define the evolution of $\Ee_0$ by the map $t \mapsto \Ee_t$ with $\Ee_t(f) = \Ee_0(\Mmm_0(\UU(0,t) f))$.

\subsubsection{Quantization in the framework of Hawking radiation} We now relate the previous algebraic consideration to the framework of the Hawking radiation. The evolution of a scalar bosonic field in a collapsing-star spacetime is described by the Klein-Gordon equation
\begin{equation*}
(\p_t^2 + r^{-1}\Ha_t r)u = 0
\end{equation*}
where $\Ha_t$ is the selfadjoint operator defined in \S \ref{sub:9u}. This equation admits the following selfadjoint formulation:
\begin{equation}\label{eq:2o}
i\dfrac{df(t)}{dt} = \Hh_t f(t), \ \ \ \ f(t) = (u(t,\cdot),i\p_tu(t,\cdot)), \ \ \ \ \Hh_t = \matrice{0 & \Id \\ r^{-1}\Ha_tr  & 0}.
\end{equation}
The operator $r\Ha_tr^{-1}$ is selfadjoint on the space $\Hhh^s_t = r^{-1}\HH^s_t$ provided with the scalar product $\lr{\cdot,\cdot}_{\Hhh^s_t} = \lr{r \cdot, r\cdot}_{\HH^{s}_t}$. The operator $\Hh_t$ is selfadjoint on the space $\mathfrak{H}^s_t = \Hhh^{s-1/2}_t \oplus \Hhh^{s+1/2}_t$  provided with the scalar product $\lr{\cdot,\cdot}_{\mathfrak{H}^s_t} = \lr{\cdot, \cdot}_{\Hhh^{s-1/2}_t \oplus \Hhh^{s+1/2}_t}$. If $f$ belongs to $C_0^\infty(\R \times \Ss)$ then $f(T,z_*(T),\w) = 0$ for $T$ large enough and we write $t \mapsto f(t)$ the solution of \eqref{eq:2o} with initial condition $f(T) = f$. We finally define the propagator $\UU(t,T)$ connecting the spaces $\mathfrak{H}_s^t$ by requiring $f(t) = \UU(t,T)f$.

An initial Bose-Einstein state at temperature $(2\beta)^{-1}$ with respect to $\Ha_0$ is given by the generating functional
\begin{equation*} 
\Ee_{\Ha_0,(2\beta)^{-1}}[f] = \exp\left( -\dfrac{1}{2} \lr{f,\matrice{\coth\left(\beta  \left(r^{-1}\Ha_0 r\right)^{1/2}\right) & 0 \\ 0 & \coth\left(\beta  \left(r^{-1}\Ha_0 r\right)^{1/2}\right)} f}_{\mathfrak{H}^s_t} \right)
\end{equation*}
where $f=(f_1,f_2)$ -- see \cite{Bachelot2}. The asymptotic evolution of such a state is described by the limit 
\begin{equation*}
\lim_{T \rightarrow \infty}\Ee_{\Ha_0,(2\beta)^{-1}}[\UU(0,T)f].
\end{equation*}
According to \eqref{eq:2q} there is heuristically only one choice of space $\mathfrak{H}^s_0$ so that this limit is be non-trivial, corresponding to $s=0$. We now fix $\mathfrak{H}=\mathfrak{H}_0^0$. The vacuum state at $t=0$ in the collapsing-star spacetime has temperature prescribed by the cosmological horizon and thus its generating functional is $\Ee_\vac = \Ee_{\Ha_0, \kappa_+/(2\pi)}$ -- see \cite{GibHaw}.  A mathematical description of the Hawking radiation consists in studying the asymptotic of the initial vacuum state as the time goes to infinity, that is, computing the limit as $T \rightarrow \infty$ of $\Ee_\vac[\UU(0,T) f]$. If $\varphi_\pm, \psi_\pm$ are given by \eqref{eq:2r} with $\beta=\beta_\pm= \pi/\kappa_\pm$ we have
\begin{equation*}\begin{gathered}
\Ee_\vac[\UU(0,T) f] \\ 
= \exp\left( -\dfrac{1}{2} \lr{\matrice{ru(0) \\ r\p_t u(0)},\matrice{\coth\left(\beta_+  \left(r^{-1}\Ha_0 r\right)^{1/2}\right) & 0 \\ 0 & \coth\left(\beta_+  \left(r^{-1}\Ha_0 r\right)^{1/2}\right)} \matrice{ru(0) \\ r\p_t u(0)}}_{\mathfrak{H}} \right) \\
 = \lim_{T \rightarrow \infty} \exp\left(  -\dfrac{1}{2}\lr{ru(0), \varphi_+(\Ha_0) ru(0)}   -\dfrac{1}{2}\lr{r\p_t u(0), \psi_+(\Ha_0) r \p_tu(0)}\right).
\end{gathered}
\end{equation*}
Here $u(0), \p_tu(0)$ are defined by \eqref{eq:2o}. To formulate Theorem \ref{thm:haw} we finally define $\Ee_{D_x^2,\kappa_\pm/(2\pi)}^\pm$ the generating functional of a Bose-Einstein state at temperature $\kappa_\pm/(2\pi)$ near the horizons $r = r_\pm$:
\begin{equation*}
\Ee_{D_x^2,\kappa_\pm/(2\pi)}^\pm[g] = \exp\left(  -\dfrac{1}{2}\lr{r_\pm g_1, \varphi_\pm(D_x^2) r_\pm g_1}   -\dfrac{1}{2}\lr{r_\pm g_2, \psi_\pm (D_x^2) r_\pm g_2}\right).
\end{equation*}
Because of these definition Theorem \ref{thm:haw} is equivalent to:

\begin{theorem}\label{thm:4} There exists $\Lambda_0$ such that for all $\Lambda \in (0,\Lambda_0]$ and under the assumptions and notations of Theorem \ref{thm:asymp},
\begin{equation*}\begin{gathered}
\lr{ru(0), \varphi_+(\Ha_0)r u(0)}  + \lr{r\p_t u(0), \psi_+(\Ha_0)r \p_tu(0)} \\ = \sum_{+/-} \lr{r_\pm u_\pm^*, \varphi_\pm(D_x^2)r_\pm u_\pm^*}  + \lr{r_\pm D_xu_\pm^*, \psi_\pm(D_x^2)r_\pm D_xu_\pm^*} + O(e^{-cT}).
 \end{gathered}
\end{equation*}
\end{theorem}

Below we prove this theorem.

\subsection{A decomposition suitable for quantum field theory spaces}\label{subsec:a} The following lemma modifies slightly the statement of Theorem \ref{thm:asymp} to apply it to the proof of Theorem \ref{thm:4}.

\begin{lem}\label{lem:1d} If $u$ solves \eqref{eq:hgc} we can write
\begin{equation*}\begin{matrix}
ru(0) = r_- u_\BB + \epsi^0_\BB + r_+ u_{\WB} + \epsi^0_{\WB}  \\
r\p_t u (0) = r_- v_\BB + \epsi^1_\BB + r_+ v_{\WB} + \epsi_{\WB}^1
\end{matrix} 
\end{equation*}
where
\begin{itemize}
\item[$(i)$] The terms $\epsi^{0,1}_{\BB,\WB}$ satisfy:
\begin{equation*}
\begin{gathered}
\supp(\epsi_\BB^0, \  \epsi_\BB^1) \subset \{ z_*(0) \leq x \leq 1 \}, \ \ \ \ \ \ |\epsi_\BB^0|_{H^{1/2}} =O(e^{-cT}), \ \ \ \ \ |\epsi_\BB^1|_{H^{-1/2}} =O(e^{-cT}),\\
\supp(\epsi_{\WB}^0, \ \epsi_\BB^1) \subset [0, T+C], \ \ \ \ \ \  |\epsi_{\WB}^0|_{C^\infty} =O(e^{-cT}), \ \ \ \ \ \  |\epsi_\WB^1|_{C^\infty} =O(e^{-cT}).
\end{gathered}
\end{equation*}
\item[$(ii)$] The terms $u_\BB, v_\BB$ satisfy  
\begin{equation*}
u_\BB(x,\w) = \chi(x) u^*_-\left( \kappa_-^{-1} \ln\left( \dfrac{\gamma_0x}{h} \right),\w \right), \ \ \ \ v_\BB = \p_x u_\BB, \ \ h=e^{-\kappa_-T}
\end{equation*}
for some $\chi \in C_0^\infty(\R)$ equal to $1$ near $x=0$ and with support in $(z_*(0,1)$. Moreover,
\begin{equation}\label{eq:estimatesg}
|u_\BB| = O(e^{-cT}) \operatorname{ and } |\Delta_\Ss u_\BB| = O(e^{-cT}).
\end{equation}
\item[$(iii)$] The terms $u_{\WB}, v_{\WB}$ satisfy $v_\WB = \p_x u_\WB$ and
\begin{equation*}
u_\WB(x,\w) =   \tchi(x) u^*_+(x-T,\w)
\end{equation*}
for some $\tchi \in C^\infty(\R)$ that is $0$ for $x \leq 0$ and $1$ for $x \geq 1$.
\end{itemize}
\end{lem}

\begin{proof} This lemma is a simple reformulation of Theorem \ref{thm:asymp} using as $T \rightarrow \infty$ $u$ splits in two parts, one that is localized at $x=0$, and one that is localized near $x = T$.\end{proof}

Because of this lemma we expect the following to hold:
\begin{equation}\label{eq:2d}\begin{gathered}
\lr{ru(0),\varphi_+(\Ha_0)ru(0)} + \lr{r\p_tu(0),\psi_+(\Ha_0)r\p_tu(0)} = \\ r_-^2 \lr{u_\BB,\varphi_+(\Ha_0)u_\BB} + r_+^2\lr{u_\WB,\varphi_+(\Ha_0)u_\WB} \\ + r_-^2 \lr{v_\BB,\psi_+(\Ha_0)v_\BB} + r_+^2\lr{v_\WB,\psi_+(\Ha_0)v_\WB}   + O(e^{-cT}).\end{gathered}
\end{equation}
We separate the proof of Theorem \ref{thm:4} into two main parts. In \S \ref{subsec:2} we compute the four limits
\begin{equation*}\begin{gathered}
\lim_{T \rightarrow \infty} \lr{u_\BB,\varphi_+(\Ha_0)u_\BB}, \ \ \ \
\lim_{T \rightarrow \infty} \lr{u_\WB,\varphi_+(\Ha_0)u_\WB}, \\ 
\lim_{T \rightarrow \infty} \lr{v_\BB, \psi_+(\Ha_0)v_\BB}, \ \ \ \
\lim_{T \rightarrow \infty} \lr{v_\WB, \psi_+(\Ha_0)v_\WB}.\end{gathered}
\end{equation*}
In \S \ref{subsec:e} we prove that \eqref{eq:2d} holds.

\subsection{Computation of the limits}\label{subsec:2}
\subsubsection{Computation of the first limit}

\begin{lem}\label{lem:lim1} We have:
\begin{equation}\label{eq:2e}
\lr{u_\BB,\varphi_+(\Ha_0)u_\BB} = \lr{\varphi_-(D_x^2) u^*_-, u^*_-} + O(e^{-cT}).
\end{equation}
\end{lem}

\begin{proof} There exists $C$ large enough such that for all $\lambda \geq 0$,
\begin{equation*}
\lambda^{1/2} \leq \varphi_+(\lambda) \lesssim \lambda^{1/2} + C.
\end{equation*}
The spectral theorem implies
\begin{equation*}
\lr{u_\BB , \Ha_0^{1/2} u_\BB} \leq \lr{u_\BB,\varphi_+(\Ha_0)u_\BB} \leq \lr{\Ha_0^{1/2} u_\BB , u_\BB} + C|u_\BB|^2.
\end{equation*}
Because of \eqref{eq:estimatesg}, $|u_\BB| = O(e^{-cT})$ and it suffices to compute the limit of $\lr{u_\BB , \Ha_0^{1/2} u_\BB}$ as $T \rightarrow \infty$.  For $c$ large enough we have $D_{x,0}^2 \leq \Ha_0 \leq D_{x,0}^2 + c(\Delta_\Ss + 1)$ and thus the Lowner-Heinz inequality implies 
\begin{equation*}
\lr{u_\BB, |D_{x,0}| u_\BB} \leq \lr{u_\BB , \Ha_0^{1/2} u_\BB} \leq \lr{u_\BB,(D_{x,0}^2 + c\Delta_\Ss+c)^{1/2}u_\BB}.
\end{equation*}
Because of the inequality $(\xi+\ell)^{1/4} \leq \xi^{1/4} + \ell^{1/4}$ and the fact that the operators $D_{x,0}^2$ and $\Delta_\Ss+1$ commute,
\begin{equation*}
\lr{u_\BB,(D_{x,0}^2 + c\Delta_\Ss+c)^{1/2}u_\BB} \leq \lr{u_\BB,|D_{x,0}|u_\BB} + c^{1/2}\lr{u_\BB,(\Delta_\Ss+1)^{1/2}u_\BB}.
\end{equation*}
The estimate \eqref{eq:estimatesg} shows that $\lr{u_\BB,(\Delta_\Ss+1)^{1/2}u_\BB} = O(e^{-cT})$. Thus the evaluation of \eqref{eq:2e} is reduced to the evaluation of $\lr{u_\BB,|D_{x,0}|u_\BB}$ as $T \rightarrow \infty$.

Let $\EE : \HH \rightarrow L^2(\R \times \Ss,dxd\w)$ be the extension operator defined by
\begin{equation}\label{eq:2s}
\EE v(x,\w) = \system{ v(x,\w) & \operatorname{for} \ x \geq z_*(0) \\ -v(2z_*(0) - x,\w) & \operatorname{for} \ x \leq z_*(0)}.
\end{equation}
This operator is a partial isometry: for every $v \in \HH$, $|\EE v|_{L^2(\R \times \Ss, dx d\w)}^2 = 2|v|^2$. The operator $\EE$ intertwines $D_{x,0}^2$ with $D_x^2$: $\EE D_{x,0}^2  = D_x^2 \EE$. Consequently if $\mu : \R \rightarrow \R$ then
\begin{align}\label{eq:2t}\begin{split}
\lr{\mu(D_{x,0}^2) v, v} & = \dfrac{1}{2} \lr{  \mu(D_x^2)\EE v, \EE v} \\
     & = \dfrac{1}{4\pi} \int_{\R \times \Ss}  \mu(\xi^2) \left|\int_\R e^{-ix\xi} \EE v(x,\w) dx \right|^2 d\w d\xi \\
     & = \dfrac{1}{\pi} \int_{\R \times \Ss} \mu(\xi^2) \left| \int_0^\infty \sin(x\xi) u(x+z_*(0),\w) dx \right|^2 d\w d\xi.\end{split}
\end{align}
Apply \eqref{eq:2t} to $\mu(\xi) = |\xi|^{1/2}$ and $v=u_\BB$ to obtain
\begin{align*}
\lr{u_\BB,|D_{x,0}|u_\BB} & = \dfrac{1}{\pi}\int_{\R \times \Ss} |\xi| \left| \int_0^\infty \sin(x\xi) u_\BB(x+z_*(0),\w) dx \right|^2 d\xi d\w \\
   & = \dfrac{1}{\pi}\int_{\R \times \Ss} |\xi| \left| \int_0^\infty \sin(x\xi) \chi(x+z_*(0))   u^*_- \left( \kappa_-^{-1} \ln \left( \gamma_0 \dfrac{x+z_*(0)}{h} \right),\w \right) dx \right|^2 d\xi d\w \\
   & = \dfrac{1}{\pi} \int_{\R \times \Ss} |\xi| \left| \int_\R \sin \left( x\xi - \dfrac{z_*(0)\xi}{h} \right) \chi(hx) u^*_-\left( \kappa_-^{-1} \ln \left( \gamma_0 x \right),\w \right) dx   \right|^2 d\xi d\w \\
   & = \dfrac{1}{4\pi} \int_{\R \times \Ss} |\xi| |e^{iz_*(0)\xi/h} Z - e^{-iz_*(0)\xi/h} \oZ|^2 d\xi d\w,
\end{align*}
where 
\begin{equation}\label{eq:kcw}
Z = \int_\R e^{-ix\xi} \chi(hx) u^*_-\left( \kappa_-^{-1} \ln \left( \gamma_0 x \right),\w \right) dx.
\end{equation}
Write
\begin{equation*}
|e^{iz_*(0)\xi/h} Z - e^{-iz_*(0)\xi/h} \oZ|^2 = 2 |Z|^2 - 2 \Ree ( e^{2iz_*(0)\xi/h} Z^2 )
\end{equation*}
and concentrate on the second term of the RHS. An integration by parts gives
\begin{equation}\label{eq:1h}
 \int_{\R \times \Ss} |\xi| e^{2iz_*(0)\xi/h} Z^2 d\xi d\w = \dfrac{h}{2iz_*(0)}  \int_{\R \times \Ss} \left( \dfrac{\xi}{|\xi|}Z^2 + 2|\xi| Z \p_\xi Z  \right)e^{2iz_*(0)\xi/h} d\xi d\w
\end{equation}
and in addition
\begin{equation}\begin{gathered}\label{eq:jh23}
\left| \dfrac{h}{2iz_*(0)}  \int_{\R \times \Ss} \operatorname{sgn}(\xi)Z^2 e^{2iz_*(0)\xi/h} d\xi d\w \right| \leq h \int_{\R \times \Ss} |Z|^2 d\xi d\w \\
\leq h \int_{\R \times \Ss} \left|\chi(hx) u^*_-\left( \kappa_-^{-1} \ln \left( \gamma_0 x \right),\w \right) \right|^2 dx d\w  \lesssim h\int_{hx \in \supp(\chi)}    \lr{x}^{-\gamma} dx   = O(h^\gamma).\end{gathered}
\end{equation}
We used from the first to the second line that $u_*^- \in X_\scatt$, which implies $u^*_-(\kappa_-\ln(\cdot),\cdot) \in S^{-\delta}$ for some $\delta > 0$. To treat the second term in the RHS of \eqref{eq:1h} we see that
\begin{equation*}
\left|\dfrac{h}{2iz_*(0)}  \int_{\R \times \Ss} 2|\xi| Z \p_\xi Z e^{2iz_*(0)\xi/h} d\xi d\w\right| \lesssim h |Z|_2 |\xi \p_\xi Z|_2.
\end{equation*}
We showed in \eqref{eq:jh23} that $h |Z|_2^2 = O(h^\gamma)$. In addition by Plancherel's formula and using that $u^*_-(\kappa_-\ln(\cdot),\cdot) \in S^{-\delta}$ for some $\delta > 0$ we have
\begin{equation*}\begin{gathered}
h|\xi \p_\xi Z|_2^2   = h \int_{\R \times \Ss}\left| \p_x \left(x \chi(hx) u^*_-\left( \kappa_-^{-1} \ln \left( \gamma_0 x \right),\w \right) \right) \right|^2 dxd\w \\
     \lesssim h \int_{\R}\left| \chi(hx) \lr{x}^{-\gamma} \right|^2 dx  
     + h \int_{\R} \left| hx \chi'(hx) \lr{x}^{-\gamma} \right|^2 dx   
    + h \int_{\R} \left| \chi(hx) \lr{x}^{-\gamma} \right|^2 dx  \\
    = O(h^\gamma) + O(h^\gamma) + O(h^\gamma) = O(h^\gamma).
   \end{gathered}
\end{equation*}
It follows that
\begin{equation*}\begin{gathered}
\lr{u_\BB,|D_{x,0}|u_\BB} = \dfrac{1}{4\pi}\int_{\R \times \Ss} 2|\xi| |Z|^2 d\xi d\w + O(h^\gamma) \\ = \dfrac{1}{2\pi}\int_{\R \times \Ss} |\xi| \left| \int_\R e^{-ix\xi} \chi(hx) u^*_-\left( \kappa_-^{-1} \ln \left( \gamma_0 x \right),\w \right) dx \right|^2 d\xi d\w + O(h^\gamma). \end{gathered}
\end{equation*}
Arguments similar to the proof of Lemma \ref{lem:1g} show that
\begin{equation*}\begin{gathered}
\int_{\R \times \Ss} |\xi| \left| \int_\R e^{-ix\xi} \chi(hx) u^*_-\left( \kappa_-^{-1} \ln \left( \gamma_0 x \right),\w \right) dx \right|^2 d\xi d\w \\ = \int_{\R \times \Ss} |\xi| \left| \int_\R e^{-ix\xi} u^*_-\left( \kappa_-^{-1} \ln \left( \gamma_0 x \right),\w \right) dx \right|^2 d\xi d\w + O(h^c).\end{gathered}
\end{equation*}
To compute the integral in the second line we apply \cite[Lemma II.6]{Bachelot2}:
\begin{equation*}\begin{gathered}
\lr{u_\BB, \varphi_+(\Ha_0) u_\BB} = \dfrac{1}{2\pi} \int_{\R \times \Ss} |\xi| \coth\left( \beta_- |\xi|  \right) |\FF u^*_-(\xi,\w)|^2 d\xi d\w + O(e^{-cT}) \\
   = \lr{u^*_-, \varphi_-(D_x^2) u^*_-}+ O(e^{-cT}).\end{gathered}
\end{equation*}
This completes the proof. \end{proof}

\subsubsection{Computation of the second limit.}

\begin{lem}\label{lem:lim2} We have:
\begin{equation*}
\lr{u_\WB,\varphi_+(\Ha_0) u_\WB}  = \lr{u^*_+, \varphi_+(D_x^2) u^*_+} + O(e^{-cT}).
\end{equation*}
\end{lem}

\begin{proof} Recall that $u_\WB$ is given by $u_\WB(x,\w) =  \tchi(x) u^*_+(x-T,\w)$ where $\tchi$ vanishes for $x \leq 0$ and is equal to $1$ for $x \geq 1$ and $u^*_+ \in X_\scatt$. Write
\begin{equation*}
\varphi_+(z) = \beta_+^{-1} + \beta_+^{-1} \sum_{k \in \Vv \setminus \{0\}} \dfrac{z}{z+k^2}, \ \ \Vv = \pi\beta_+^{-1}\Z.
\end{equation*}
The series $\sum_{k \in \Vv \setminus \{0\}} (\Ha_0+k^2)^{-1} \Ha_0 u_\WB$ converges in $\HH$. Indeed,
\begin{equation}\label{eq:2g}
\sum_{ k \in \Vv \setminus \{0\}} \left| (\Ha_0+k^2)^{-1} \Ha_0 u_\WB\right| \leq \sum_{ k \in \Vv \setminus \{0\} } \lr{k}^{-2} |\Ha_0 u_\WB| < \infty
\end{equation}
as $u_\WB$ is smooth and decays exponentially. It follows that
\begin{equation*}
\lr{ u_\WB, \varphi_+(\Ha_0)u_\WB} = \beta_+^{-1} |u_\WB|^2 + \beta_+^{-1}\sum_{k \in \Vv \setminus \{0\}} \lr{u_\WB,(\Ha_0+k^2)^{-1} \Ha_0 u_\WB},
\end{equation*}
where the series converges absolutely. The estimate \eqref{eq:2g} and the bound $|\Ha_0 u_\WB| = O(1)$ imply that the convergence of the series is uniform as $T \rightarrow \infty$. We now study the terms 
\begin{equation*}
|u_\WB|^2, \ \ \ \lr{(\Ha_0+k^2)^{-1} \Ha_0 u_\WB,u_\WB}, \ \ \ T \rightarrow \infty.
\end{equation*}
First note that
\begin{align*}
|u_\WB - u^*_+|^2 & =  \int_{\R \times \Ss} |1-\tchi(x)|^2 |u^*_+(x-T,\w)|^2 dxd\w \\
    & \lesssim \int_{-\infty}^1 e^{-\nu(T-x)} dx = O(e^{-\nu T}).
\end{align*}
We conclude that $|u_\WB|^2 = |u^*_+|^2 + O(e^{-cT})$. We now derive similar estimates for the term $\lr{(\Ha_0+k^2)^{-1} \Ha_0 u_\WB,u_\WB}$. For $k \in \Vv \setminus \{0\}$,
\begin{equation*}\begin{gathered}
(\Ha_0+k^2)^{-1} \Ha_0 -(D_{x,0}^2+k^2)^{-1} D_{x,0}^2 \\ = (\Ha_0+k^2)^{-1} (\Ha_0-D_{x,0}^2)   - (\Ha_0+k^2)^{-1} (\Ha_0-D_{x,0}^2) (D_{x,0}+k^2)^{-1} D_{x,0}^2.\end{gathered}
\end{equation*}
Therefore
\begin{equation}\label{eq:2h}\begin{gathered}
\left|\lr{ u_\WB, (\Ha_0+k^2)^{-1} \Ha_0 u_\WB} - \lr{ u_\WB,(D_{x,0}^2+k^2)^{-1} D_{x,0}^2 u_\WB} \right| \\
\leq \left|\lr{  u_\WB, (\Ha_0+k^2)^{-1} (\Ha_0-D_{x,0}^2) u_\WB}\right| \\
\ \ \ + \left|\lr{  D_{x,0}^2 u_\WB, (\Ha_0+k^2)^{-1} (\Ha_0-D_{x,0}^2) (D_{x,0}+k^2)^{-1} u_\WB}\right| \\
\lesssim |(\Ha_0-D_{x,0}^2) u_\WB|\dfrac{|u_\WB|}{|k|^2}  + |(\Ha_0-D_{x,0}^2) (D_{x,0}+k^2)^{-1} D_{x,0}^2 u_\WB|\dfrac{|u_\WB|}{|k|^2}.
\end{gathered}
\end{equation}
The operator $\Ha_0-D_{x,0}^2$ is a second order differential operator whose coefficient decay like $e^{-\kappa_+\lr{x}}$. The function $u_\WB$ has support in $[0,T+C]$ and all its derivative decay like $e^{-\nu\lr{x-T}}$. Thus $|(\Ha_0-D_{x,0}^2) u_\WB| = O(e^{-cT})$ and $|(\Ha_0-D_{x,0}^2) u_\WB| = O(e^{-cT})$. We now focus on the term $|(\Ha_0-D_{x,0}^2) (D_{x,0}+k^2)^{-1} D_{x,0}^2 u_\WB|$. The kernel of the operator $(D_{x,0}^2+k^2)^{-1}$ is given by \eqref{eq:kerne}:
\begin{equation*}
K_{ik}(x,y) = \dfrac{e^{-k|x-y|}-e^{-k|x+y-2z_*(0)|}}{2k}.
\end{equation*}
In addition $u_\WB$ and its derivative decay like $e^{-\nu\lr{T-x}}$. It follows that $(\Ha_0-D_{x,0}^2)(D_{x,0}^2+k^2)^{-1} D_{x,0}^2u_{\WB}$ can be estimated by:
\begin{equation*}
|(\Ha_0-D_{x,0}^2)(D_{x,0}^2+k^2)^{-1} D_{x,0}^2u_{\WB}(x)| \lesssim e^{-\kappa_+ x} \int_0^{T+C} \dfrac{e^{-k|x-y|}-e^{-k|x+y-2z_*(0)|}}{2k} e^{-\nu y} dy.
\end{equation*}
This decays exponentially in $T$; therefore $|(\Ha_0-D_{x,0}^2)(D_{x,0}^2+k^2)^{-1} D_{x,0}^2 u_{\WB}| = O(e^{-c T})$. Plugging this estimates in \eqref{eq:2h} we obtain that uniformly in $k \in \Vv\setminus \{0\}$,
\begin{equation*}
\lr{(\Ha_0+k^2)^{-1 }\Ha_0 u_\WB, u_\WB} = \lr{(D_{x,0}^2+k^2)^{-1} D_{x,0}^2 u_\WB, u_\WB} + O(e^{-cT}/|k|^2).
\end{equation*}
Sum over $k \in \Vv \setminus \{0\}$ to conclude that
\begin{equation*}
\lr{u_{\WB},\varphi_+(\Ha_0) u_{\WB}} = \lr{ u_{\WB},\varphi_+(D_{x,0}^2) u_{\WB}} + O(e^{-cT}).
\end{equation*}

We next show that $\lr{ u_\WB, \varphi_+(D_{x,0}^2) u_\WB} = \lr{u^*_+, \varphi_+(D_x^2) u^*_+} + O(e^{-cT})$. Using the extension operator $\EE$ defined in \eqref{eq:2s} we have
\begin{equation*}\begin{gathered}
\lr{\varphi_+(D_{x,0}^2) u_{\WB}, u_{\WB}}  = \lr{\varphi_+(D_x^2) \EE u_{\WB}, \EE u_{\WB}} \\
  = \dfrac{1}{2\pi} \int_{\R \times \Ss} |\xi|\coth\left( \beta_+ |\xi| \right) \left|\int_\R e^{-ix\xi} \tchi(x) u^*_+(x-T) dx\right|^2 d\xi d\w \\
  = \dfrac{1}{2\pi} \int_{\R \times \Ss} |\xi|\coth\left( \beta_+ |\xi| \right) \left|\int_\R e^{-ix\xi} \tchi(x+T) u^*_+(x) dx\right|^2 d\xi d\w.\end{gathered}
\end{equation*}
To conclude we need to replace the term $\tchi(x)$ with $1$. We note that
\begin{equation*}\begin{gathered}
 \int_{\R \times \Ss} |\xi|\coth\left( \beta_+ |\xi|  \right) \left|\int_\R e^{-ix\xi} (1-\tchi(x+T)) u^*_+(x) dx\right|^2 d\xi d\w \\
\lesssim  \int_{\R \times \Ss} (1+|\xi|^2) \left|\int_\R e^{-ix\xi} (1-\tchi(x+T)) u^*_+(x) dx\right|^2 d\xi d\w.\end{gathered}
\end{equation*}
Using Plancherel's formula this is equal to
\begin{equation}\label{eq:2u}
\int_{\R \times \Ss} |\p_x ((1-\tchi(\cdot+T)) u^*_+)|^2 dx d\w + \int_{\R \times \Ss} |(1-\tchi(\cdot+T)) u^*_+|^2 dx d\w .
\end{equation}
As $1-\tchi(\cdot+T)$ is supported on $x \lesssim -T$ and as $u^*_+ \in X_\scatt$, \eqref{eq:2u} is $O(e^{-cT})$. We finally conclude that
\begin{equation*}\begin{gathered}
\lr{u_\WB, \varphi_+(\Ha_0) u_\WB }  =  \dfrac{1}{2\pi} \int_{\R \times \Ss} |\xi|\coth\left( \beta_+ |\xi| \right) \left|\int_\R e^{-ix\xi} u^*_+(x,\w) dx\right|^2 d\xi d\w + O(e^{-cT})\\
     = \lr{u^*_+, \varphi_+(D_x^2) u^*_+ } + O(e^{-cT}).\end{gathered}
\end{equation*}
This completes the proof. \end{proof}

\subsubsection{Computation of the third limit}

\begin{lem} We have:
\begin{equation*}
\lr{v_\BB, \psi_+(\Ha_0) v_\BB } =  \lr{D_x u^*_-, \psi_-(D_x^2) D_x u^*_- } + O(e^{-cT}).
\end{equation*}
\end{lem}

\begin{proof} We first prove that as $T \rightarrow \infty$, $\lr{v_\BB,\Ha_0^{-1}  v_\BB} = O(e^{-cT})$. By the Lowner-Heinz theorem $0 \leq \lr{ v_\BB,\Ha_0^{-1} v_\BB} \leq \lr{v_\BB,D_{x,0}^{-2} v_\BB}$. Mimicking the proof of Lemma \ref{lem:lim1} up to equation \eqref{eq:kcw},
\begin{equation*}
\lr{v_\BB,D_{x,0}^{-2} v_\BB} = h \int_{\R \times \Ss} |\xi|^{-2} \left| e^{iz_*(0)\xi/h} Y - e^{-iz_*(0)\xi/h} \oY \right|^2 d\xi d\w,
\end{equation*}
where
\begin{equation*}
Y = \int_\R e^{-ix\xi} \p_x \left(\chi(hx) \dfrac{1}{\kappa_- x} (\p_x u^*_-)\left( \kappa_-^{-1} \ln \left( \gamma_0 x \right),\w \right) \right) dx.
\end{equation*}
By an integration by parts, $Y=-i \xi Z$ where $Z$ was defined in \eqref{eq:kcw}. It follows that 
\begin{align}\label{eq:2k}\begin{split}
\lr{v_\BB,D_{x,0}^{-2} v_\BB} & = h \int_{\R \times \Ss} | e^{iz_*(0)\xi/h} Z - e^{-iz_*(0)\xi/h} \oZ| d\xi d\w \\ & \leq 2 h \int_{\R \times \Ss} |Z|^2 d\xi d\w  = O(h^\gamma)\end{split}
\end{align}
where the last equality comes from \eqref{eq:jh23}. It follows that $\lr{v_\BB,\Ha_0^{-1}  v_\BB} = O(e^{-cT})$. 

The function $\psi_+$ satisfies $\psi_+(z) =(\beta_+z)^{-1} + O(1)$ near $0$ and $\psi_+(z) = \lr{z}^{-1/2} + O(z^{-1})$ near $+\infty$. It follows that there exists a constant $c > 0$ such that for every $z \in [0,\infty)$,
\begin{equation*}
(z^{1/2}+1)^{-1} - c z^{-1} \leq \psi_+(z) \leq (z^{1/2}+1)^{-1} + cz^{-1}. 
\end{equation*}
The spectral theorem and the estimate $\lr{v_\BB,\Ha_0^{-1}  v_\BB} = O(e^{-cT})$ imply
\begin{equation*}
\lr{ v_\BB, \psi_+(\Ha_0) v_\BB} = \lr{v_\BB, (\Ha_0^{1/2}+1)^{-1}v_\BB} + O(e^{-cT}). 
\end{equation*}
We now focus on the term $\lr{ v_\BB, (\Ha_0^{1/2}+1)^{-1}v_\BB}$. The operator $\Ha_0$ satisfies $D_{x,0}^2 \leq \Ha_0 \leq D_{x,0}^2 + c\Delta_\Ss + c$ and the Lowner-Heinz inequality yields
\begin{equation}\label{eq:1f}
\lr{v_\BB, (D_{x,0}^2 + c\Delta_\Ss+c)^{-1/2} v_\BB} \leq \lr{ v_\BB, (\Ha_0^{1/2}+1)^{-1}v_\BB} \leq \lr{ v_\BB, |D_{x,0}|^{-1}v_\BB}.
\end{equation}
The operators $D_{x,0}^2$ and $\Delta_\Ss$ commute and are nonnegative. Moreover for every $\xi \geq 0, \ell \geq 0$ we have
\begin{equation*}
\dfrac{1}{\xi} - \dfrac{1}{\sqrt{\xi^2+\ell^2}} = \dfrac{\sqrt{\xi^2+\ell^2}- \xi}{\xi \sqrt{\xi^2 + \ell^2}} \leq \dfrac{\ell+\xi-\xi}{\xi \sqrt{\xi^2 + \ell^2}} \leq \dfrac{\ell}{\xi^2}.
\end{equation*}
This inequality implies
\begin{equation*}
\dfrac{1}{|\xi|} - \dfrac{\ell}{\xi^2}  \leq \dfrac{1}{\sqrt{\xi^2+\ell^2}}
\end{equation*}
and the operator-valued version is 
\begin{equation*}
|D_{x,0}|^{-1} - (c\Delta_\Ss+c)^{1/2} D_{x,0}^{-2} \leq (D_{x,0}^2 + c\Delta_\Ss+c)^{-1/2}.
\end{equation*}
Consequently \eqref{eq:1f} becomes
\begin{equation*}
\lr{v_\BB,|D_{x,0}|^{-1}v_\BB} - \lr{v_\BB,(c\Delta_\Ss+c)^{1/2} D_{x,0}^{-2}v_\BB} \leq  \lr{ v_\BB, (\Ha_0^{1/2}+1)^{-1}v_\BB} \leq \lr{ v_\BB, |D_{x,0}|^{-1}v_\BB}.
\end{equation*} We recall that $v_\BB = \p_x u_\BB$ and $u_\BB$ vanishes near $x = z_*(0)$. Thus
\begin{equation*}
\lr{v_\BB,(c\Delta_\Ss+c)^{1/2} D_{x,0}^{-2}v_\BB} = \lr{u_\BB,(c\Delta_\Ss+c)^{1/2} u_\BB} = O(e^{-cT}).
\end{equation*}
Hence the lemma holds if we can show that
\begin{equation*}
\lr{ v_\BB, |D_{x,0}|^{-1} v_\BB} =  \lr{v_\BB,\psi_-(D_x^2) v_\BB} + O(e^{-cT}).
\end{equation*}
Mimicking the beginning of the proof of Lemma \ref{lem:lim1} up to equation \eqref{eq:kcw} yields
\begin{equation*}\begin{gathered}
\lr{v_\BB,|D_{x,0}|^{-1}  v_\BB} = \int_{\R \times \Ss} |\xi|^{-1} \left| e^{iz_*(0)\xi/h} Y - e^{-iz_*(0)\xi/h} \oY \right|^2 d\xi d\w \\
 = \int_{\R \times \Ss} |\xi| | e^{iz_*(0)\xi/h} Z - e^{-iz_*(0)\xi/h} \oZ | d\xi d\w \\
   = \lr{u^*_-, \varphi_-(D_x^2) u^*_-} = \lr{D_xu^*_-, \psi_-(D_x^2) D_xu^*_-} + O(e^{-cT}).\end{gathered}
\end{equation*}
where the equality in the last line directly comes from the computations in the proof of Lemma \ref{lem:lim1}. This completes the proof. \end{proof}

\subsubsection{Computation of the fourth limit}

\begin{lem}\label{lem:lim4} We have:
\begin{equation*}
\lr{v_\WB, \psi_+(\Ha_0) v_\WB} = \lr{D_x u^*_-, \psi_+(D_x^2) u^*_-} + O(e^{-cT}).
\end{equation*}
\end{lem}

\begin{proof} First write $\psi_+(z) = (\beta_+z)^{-1} + \phi_+(z)$ where $\phi_+$ is holomorphic and bounded on a neighborhood of $[0,\infty)$. By methods similar to the proof of Lemma \ref{lem:lim2} we have
\begin{equation*}
\lr{ v_\WB, \phi_+(\Ha_0)v_\WB} = \lr{D_xu^*_-, \phi_+(D_x^2)D_x u^*_-} + O(e^{-cT}).
\end{equation*}
To prove the lemma it suffices to prove that
\begin{equation}\label{eq:1r}
\lr{\Ha_0^{-1} v_\WB, v_\WB} = \lr{D_{x,0}^{-2} v_\WB, v_\WB} + O(e^{-cT}).
\end{equation}
Recall that $R_0(\lambda), R(\lambda)$ are the operators defined by
\begin{equation*}
R_0(\lambda) = (D_{x,0}^2-\lambda^2)^{-1}, \ \ \ R(\lambda) = (\Ha_0-\lambda^2)^{-1}.
\end{equation*}
These are holomorphic families of operators near $\lambda=0$. Thus to prove Equation \eqref{eq:1r} it suffices to show  that
\begin{equation*}
\lim_{\epsilon \rightarrow 0} \lr{ v_\WB, (R(i\epsilon) - R_0(i\epsilon))v_\WB} = O(e^{-cT}).
\end{equation*}
For $\epsilon > 0$,
\begin{equation*}\begin{gathered}
 R(i\epsilon) - R_0(i\epsilon) = R(i\epsilon) (D_{x,0}^2-\Ha_0) R_0(i\epsilon)  \\
   = R_0(i\epsilon) (D_{x,0}^2-\Ha_0) R_0(i\epsilon) + \left( R(i\epsilon) - R_0(i\epsilon) \right) (D_{x,0}^2-\Ha_0) R_0(i\epsilon) \\
   = R_0(i\epsilon) (D_{x,0}^2-\Ha_0) R_0(i\epsilon) - R_0(i\epsilon) (D_{x,0}^2-\Ha_0) R(i\epsilon) (D_{x,0}^2-\Ha_0) R_0(i\epsilon).\end{gathered}
\end{equation*}
It yields
\begin{equation}\label{eq:1p}\begin{gathered}
\lr{\left(R(i\epsilon) - R_0(i\epsilon)\right)v_\WB, v_\WB}   = \lr{ (D_{x,0}^2-\Ha_0) R_0(i\epsilon) v_\WB, R_0(i\epsilon)v_\WB} \\
    - \lr{  R(i\epsilon) (D_{x,0}^2-\Ha_0) R_0(i\epsilon) v_\WB, (D_{x,0}^2-\Ha_0)R_0(i\epsilon)  v_\WB}.
\end{gathered}
\end{equation}
We want to prove that the limits as $\epsilon \rightarrow 0$ of both the terms on the RHS decay exponentially as $T$ goes to infinity. Start with the first term. We have
\begin{equation*}
\lr{ (D_{x,0}^2-\Ha_0) R_0(i\epsilon) v_\WB, R_0(i\epsilon)v_\WB} = \lr{ e^{\az x} (D_{x,0}^2-\Ha_0) e^{\az x} L_\az(i\epsilon) e^{\az x} v_\WB, L_\az(i\epsilon) e^{\az x} v_\WB }, 
\end{equation*}
where $L_\az(\lambda)$ is defined by
\begin{equation*}
L_\az(\lambda) = e^{-\az x} (D_{x,0}^2 - \lambda^2)^{-1} e^{-\az x} : \HH \rightarrow \HH \ \ \Im \lambda > -\az/2.
\end{equation*}
By Lemma \ref{lem:1c}, $L_\az(\lambda)$ is a holomorphic family of operators for $\Ime \lambda > -\az/2$ and $L_\az(i\epsilon)$ converges to $L_\az(0)$ for the topology of bounded operators on $\HH$. It yields
\begin{equation*}\begin{gathered}
\lim_{\epsilon \rightarrow 0} \lr{(D_{x,0}^2-\Ha_0) R_0(i\epsilon) v_\WB, R_0(i\epsilon)v_\WB }  \\ = \lr{ e^{\az x} (D_{x,0}^2-\Ha_0) e^{\az x} L_\az(0) e^{\az x} v_\WB, L_\az(0) e^{\az x} v_\WB }   = \lr{ (D_{x,0}^2-\Ha_0) D_{x,0}^{-2} v_\WB, D_{x,0}^{-2} v_\WB }.\end{gathered}
\end{equation*}
Since $v_\WB = \p_x u_\WB$ and the kernel of $D_{x,0}^{-2}$ is given by $\min(x,y)-z_*(0)$,
\begin{align*}
D_{x,0}^{-2} v_\WB(x,\w) & = \int_{z_*(0)}^\infty \left(\min(x,y)-z_*(0) \right) \p_y u_\WB(y,\w) dy 
   & = -\int_{z_*(0)}^x u_\WB(y,\w) dy.
\end{align*}
This leads to the pointwise bound $|D_{x,0}^{-2} v_\WB(x,\w)| \lesssim \lr{x}$. The operator $D_{x,0}^2-\Ha_0$ has coefficients decaying exponentially and $u_\WB(y,\w)$ and its derivatives decay like $e^{-\nu|y-T|}$. Thus
\begin{equation}\label{eq:1q}
\left|(D_{x,0}^2-\Ha_0)D_{x,0}^{-2} v_\WB(x,\w)\right| \lesssim \lr{x}\int_{z_*(0)}^x e^{-\kappa_+ x} e^{-\nu|y-T|} dy.
\end{equation}
For $x \leq T/2$, this is bounded by $e^{-\nu T/2}$, while for $x \geq T/2$ this is bounded by $e^{-\kappa_+ x}$. This yields the bound
\begin{equation}\label{eq:1t}
\left|\lr{ (D_{x,0}^2-\Ha_0) D_{x,0}^{-2} v_\WB, D_{x,0}^{-2} v_\WB }\right| \lesssim \int_{z_*(0)}^{T/2} \lr{x} e^{-\nu \lr{x}} dx  + \int_{T/2}^\infty \lr{x} e^{-\kappa_+ x} dx  \lesssim e^{-cT}.
\end{equation}
This proves that the first term in the RHS of \eqref{eq:1p} is $O(e^{-cT})$. Concentrate now on the second term of the RHS of \eqref{eq:1p}. Since the operator $R_0(i\epsilon)$ is positive we have
\begin{equation*}
\lr{R(i\epsilon) (D_{x,0}^2-\Ha_0) R_0(i\epsilon) v_\WB, (D_{x,0}^2-\Ha_0) R_0(i\epsilon) v_\WB } \geq 0.
\end{equation*}
The Lowner-Heinz inequality implies $R(i\epsilon) \leq R_0(i\epsilon)$. Hence
\begin{equation*}\begin{gathered}
    \lr{R(i\epsilon) (D_{x,0}^2-\Ha_0) R_0(i\epsilon) v_\WB, (D_{x,0}^2-\Ha_0) R_0(i\epsilon) v_\WB } \\
\leq   \lr{R_0(i\epsilon) (D_{x,0}^2-\Ha_0) R_0(i\epsilon) v_\WB, (D_{x,0}^2-\Ha_0) R_0(i\epsilon) v_\WB} \\
=  \lr{ L_\az(i\epsilon) e^{\az x} (D_{x,0}^2-\Ha_0) e^{\az x} L_\az(i\epsilon) e^{\az x} v_\WB, e^{\az x} (D_{x,0}^2-\Ha_0) e^{\az x} L_\az(i\epsilon) e^{\az x}  v_\WB}.\end{gathered}
\end{equation*}
By Lemma \ref{lem:1c},
\begin{equation*}\begin{gathered}
   \lim_{\epsilon \rightarrow 0} \lr{ L_\az(i\epsilon) e^{\az x} (D_{x,0}^2-\Ha_0) e^{\az x} L_\az(i\epsilon) e^{\az x} v_\WB, e^{\az x} (D_{x,0}^2-\Ha_0) e^{\az x} L_\az(i\epsilon) e^{\az x}  v_\WB} \\
=   \lr{ L_\az(0) e^{\az x} (D_{x,0}^2-\Ha_0) e^{\az x} L_\az(0) e^{\az x} v_\WB, e^{\az x} (D_{x,0}^2-\Ha_0) e^{\az x} L_\az(0) e^{\az x}  v_\WB} \\
=   \lr{ D_{x,0}^{-2} (D_{x,0}^2-\Ha_0) D_{x,0}^{-2} v_\WB, (D_{x,0}^2-\Ha_0) D_{x,0}^{-2}  v_\WB}.\end{gathered}
\end{equation*}
By \eqref{eq:1v}, this is controlled by
\begin{equation}\label{eq:1o}
\int_\Ss \left| \int_{z_*(0)}^\infty \lr{x} |(D_{x,0}^2-\Ha_0) D_{x,0}^{-2} v_\WB(x,\w) dx| \right|^2 d\w.
\end{equation}
The bound \eqref{eq:1q} shows that \eqref{eq:1o} decays like $e^{-cT}$ for some $c>0$. Thus
\begin{equation}\label{eq:1s}
0 \leq \lim_{\epsilon \rightarrow 0} \lr{R(i\epsilon) (D_{x,0}^2-\Ha_0) R_0(i\epsilon) v_\WB, (D_{x,0}^2-\Ha_0) R_0(i\epsilon) v_\WB } \lesssim e^{-cT}. 
\end{equation}
Putting together \eqref{eq:1t} and \eqref{eq:1s} leads to \eqref{eq:1r}. This completes the proof. \end{proof}

\subsection{Proof of \eqref{eq:2d}}\label{subsec:e}

In this section we end the proof of Theorem \ref{thm:4} by proving that \eqref{eq:2d} holds.

\begin{lem} Let $u$ solution of the boundary value problem \eqref{eq:hgc}. As $T \rightarrow \infty$,
\begin{equation*}\begin{gathered}
\lr{ru(0),\varphi_+(\Ha_0)ru(0)} + \lr{r\p_tu(0),\psi_+(\Ha_0)r\p_tu(0)} = \\ r_-^2 \lr{u_\BB,\varphi_+(\Ha_0)u_\BB} + r_+^2\lr{u_\WB,\varphi_+(\Ha_0)u_\WB} \\ + r_-^2 \lr{v_\BB,\psi_+(\Ha_0)v_\BB} + r_+^2\lr{v_\WB,\psi_+(\Ha_0)v_\WB}   + O(e^{-cT}).\end{gathered}
\end{equation*}
\end{lem}

\begin{proof} We first show the four following estimates:
\begin{equation}\label{eq:2a}\begin{gathered}
\lr{\epsi^0_\BB, \varphi_+(\Ha_0) \epsi^0_\BB} = O(e^{-cT}), \ \ \
 \lr{\epsi^0_\WB, \varphi_+(\Ha_0) \epsi^0_\WB} = O(e^{-cT}), \\
\lr{\epsi^1_\BB, \psi_+(\Ha_0) \epsi^1_\BB} = O(e^{-cT}), \ \ \ 
 \lr{\epsi^1_\WB, \psi_+(\Ha_0) \epsi^1_\WB} = O(e^{-cT}),\end{gathered}
\end{equation}
where $\epsi_\WB^i$, $\epsi_\BB^i$, $i=0,1$ are given by Lemma \ref{lem:1d}. Let $\Phi^+$ such that $\varphi_+ = \Phi_+^2$ then $|\Phi_+(z)| \lesssim \lr{z}^{1/4}$. Lemmas \ref{lem:1f} and \ref{lem:1d} imply
\begin{equation*}\begin{gathered}
\lr{\epsi^0_\BB, \varphi_+(\Ha_0) \epsi^0_\BB} = |\Phi_+(\Ha_0)\epsi^0_\BB|^2 \lesssim |\epsi^0_\BB|_{H^{1/2}}^2 = O(e^{-cT}) \\
\lr{ \epsi^0_\WB, \varphi_+(\Ha_0)\epsi^0_\WB} = |\Phi_+(\Ha_0)\epsi^0_\WB|^2 \lesssim |\epsi^0_\WB|_{H^{1/2}}^2 \leq T |\epsi^0_\WB|_{C^\infty_0} = O(e^{-cT}).\end{gathered}
\end{equation*}
This proves the first two estimates of \eqref{eq:2a}. For the third estimate we recall that $\epsi_\BB^1$ is compactly supported and $|\epsi_\BB^1|_{H^{-1/2}} = O(e^{-cT})$. Lemma \ref{lem:9q} applies and yields
\begin{equation*}
\lr{\epsi^1_\BB, \psi_+(\Ha_0)\epsi^1_\BB} \leq C|\epsi_\BB^1|_{H^{-1/2}}^2 = O(e^{-cT}).
\end{equation*}
This shows the third estimate of \eqref{eq:2a}. For the fourth estimate we write $\psi_+(z) = (\beta_+z)^{-1} + \phi_+(z)$ where $\phi_+$ is uniformly bounded on $[0,\infty)$. By the spectral theorem $\psi_+(\Ha_0) = \beta_+^{-1}\Ha_0^{-1} + \phi_+(\Ha_0)$, where $\phi_+(\Ha_0)$ is a bounded operator. We  bound $\lr{ \epsi^1_\WB, \phi_+(\Ha_0)\epsi^1_\WB}$ by $|\epsi^1_\WB|^2$ and $\lr{\epsi^1_\WB,\Ha_0^{-1}\epsi^1_\WB}$ by \eqref{eq:1m}. This gives:
\begin{equation*}\begin{gathered}
0 \leq \lr{\epsi^1_\WB, \psi_+(\Ha_0)\epsi^1_\WB} = \lr{\epsi^1_\WB, \Ha_0^{-1}\epsi^1_\WB} + \lr{\epsi^1_\WB, \phi_+(\Ha_0)\epsi^1_\WB} \\
    \lesssim \int_\Ss \left|\int_{z_*(0)}^\infty \lr{x} |\epsi_\WB^1(x,\w)|dx\right|^2 d\w + |\epsi^1_\WB|^2.\end{gathered}
\end{equation*}
By Lemma \ref{lem:1d}, point $(i)$, $\epsi^1_\WB$ has support in $[0,T+C]$ and is uniformly bounded by $O(e^{-cT})$. It follows that
\begin{equation*}
\int_\Ss \left|\int_{z_*(0)}^\infty \lr{x} |\epsi_\WB^1(x,\w)| dx\right|^2 d\w + |\epsi^1_\WB|^2 = O(e^{-cT}).
\end{equation*}
This ends the proof of \eqref{eq:2a}.

Recall the following version of the Cauchy-Schwartz inequality: if $A$ is a nonnegative symmetric operator then for all $v,w \in D(A)$,  $|\lr{v,Aw}|^2 \leq \lr{v,Av} \lr{w,Aw}$. We use this inequality to evaluate $\lr{ u_\BB, \varphi_+(\Ha_0)\epsi_\BB^0}$. By Lemma \ref{lem:lim1} we have $\lr{ u_\BB, \varphi_+(\Ha_0)u_\BB} = O(1)$. By \eqref{eq:2a} we have $\lr{ \epsi_\BB^0, \varphi_+(\Ha_0)\epsi_\BB^0} = O(e^{-cT})$. Consequently the above version of the Cauchy-Schwartz inequality implies $
|\lr{ u_\BB, \varphi_+(\Ha_0)\epsi_\BB^0}| = O(e^{-cT})$. The same argument works to evaluate the terms that are crossed with one of the $\epsi_{\BB,\WB}^{0,1}$ and one of the $u_{\BB,\WB}$: $\lr{ u_\WB, \varphi_+(\Ha_0)\epsi_\BB^0}$, $\lr{ u_\BB, \varphi_+(\Ha_0)\epsi_\WB^0}$, ..., $\lr{v_\WB, \psi_+(\Ha_0)\epsi_\WB^1}$. It leads to 
\begin{equation*}\begin{gathered}
\lr{ru(0),\varphi_+(\Ha_0)ru(0)} + \lr{r\p_tu(0),\psi_+(\Ha_0)r\p_tu(0)}   = \\ r_-^2 \lr{u_\BB,\varphi_+(\Ha_0)u_\BB} + r_+^2\lr{u_\WB,\varphi_+(\Ha_0)u_\WB}  +r_+^2\lr{v_\BB,\varphi_+(\Ha_0)v_\BB} + r_-^2\lr{v_\WB,\varphi_+(\Ha_0)v_\WB} \\
  + 2 r_- r_+ \lr{u_\BB, \varphi_+(\Ha_0)u_\WB} + 2r_-r_+ \lr{v_\BB, \psi_+(\Ha_0)v_\WB}   + O(e^{-cT}).
\end{gathered}
\end{equation*}
Thus the proof will be over if we can show that
\begin{equation}\begin{gathered}\label{eq:2b}
\lr{u_\BB, \varphi_+(\Ha_0)u_\WB} = O(e^{-cT}), \ \ \ 
\lr{v_\BB, \psi_+(\Ha_0)v_\WB} = O(e^{-cT}).
\end{gathered}
\end{equation}
We first note that
\begin{equation*}
 \blr{  u_\BB , \varphi_+(\Ha_0)u_{\WB} } = \blr{u_\BB, \varphi_+(\Ha_0) u_{\WB}} \leq |u_\BB| \cdot |\varphi_+(\Ha_0) u_\WB|.
\end{equation*}
We have $|u_{\BB}| = O(e^{-cT})$. In addition $\varphi_+$ satisfies $|\varphi_+(z)| \leq \lr{z}^{1/2}$. This yields $|\varphi_+(\Ha_0) u_\WB| \leq |u_\WB|_{H^1} = O(1)$. It follows that $\lr{\varphi_+(\Ha_0)u_\BB, u_\WB} = O(e^{-cT})$ and the first claim of \eqref{eq:2b} holds. In order to prove that the second claim holds we recall that
\begin{equation*}
\psi_+(z) = \dfrac{1}{\beta_+z} + \sum_{k \in \Vv \setminus \{0\}} \dfrac{1}{z+k^2}, \ \ \ \Vv = \pi \beta^{-1}_+ \Z.
\end{equation*}
Thus 
\begin{equation*}
\lr{v_\BB,\psi_+(\Ha_0)v_\WB} = \lr{v_\BB,\Ha_0^{-1}v_\WB} + \sum_{k \in \Vv \setminus \{0\}} \lr{v_\BB,(\Ha_0+k^2)^{-1}v_\WB}
\end{equation*}
where the series converges absolutely. By an application of the Cauchy-Schwartz inequality, 
\begin{equation*}
|\lr{v_\BB,\Ha_0^{-1}v_\WB}|^2 \leq |\lr{v_\BB,\Ha_0^{-1}v_\BB}| |\lr{v_\WB,\Ha_0^{-1}v_\WB}|.
\end{equation*}
By \eqref{eq:1m} and the fact that the support of $v_\WB$ is included in $[0,T+C] \times \Ss$, $|\lr{v_\WB,\Ha_0^{-1}v_\WB}| = O(T^2)$. By \eqref{eq:2k} and the Lowner-Heinz inequality, $|\lr{v_\BB,\Ha_0^{-1}v_\BB}| = O(e^{-cT})$. It leads to $\lr{v_\BB,\Ha_0^{-1}v_\WB} = O(e^{-cT})$. For the terms of the form $\lr{\lr{v_\BB,(\Ha_0+k^2)^{-1}v_\WB}}$ where $k \in  \Vv\setminus \{0\}$ we note that by the spectral theorem and the Cauchy-Schwartz inequality,
\begin{equation*}\begin{gathered}
|\lr{v_\BB,(\Ha_0+k^2)^{-1}v_\WB}|  = |\lr{(\Ha_0+k^2)^{-3/8}v_\BB,(\Ha_0+k^2)^{-5/8}v_\WB}| \\
   \leq |\lr{(\Ha_0+k^2)^{-3/4}v_\BB,v_\BB}|^{1/2} |\lr{(\Ha_0+k^2)^{-5/4}v_\WB,v_\WB}|^{1/2} \lesssim \dfrac{|v_\WB|}{|k|^{5/4}} |\lr{|D_{x,0}|^{-3/2}v_\BB,v_\BB}|^{1/2}.\end{gathered}
\end{equation*}
In the last line we applied the bound $|(\Ha_0+k^2)^{-5/4}| \lesssim |k|^{-5/2}$ (coming from the spectral theorem) and the bound $(\Ha_0+k^2)^{-3/4} \leq |D_{x,0}|^{-3/2}$ (coming from the Lowner-Heinz inequality). In the proof of Lemma \ref{lem:lim1} we showed that $\lr{D_{x,0}^{-2}v_\BB,v_\BB} = O(e^{-cT})$ and $\lr{|D_{x,0}|^{-1}v_\BB,v_\BB} = O(1)$. Hence by interpolation
\begin{equation*}
|\lr{|D_{x,0}|^{-3/2}v_\BB,v_\BB}|^{1/2} = O(e^{-cT}).
\end{equation*}
It follows that for $k \in \Vv\setminus \{0\}$,
\begin{equation*}
|\lr{v_\BB,(\Ha_0+k^2)^{-1}v_\WB}| \lesssim  |v_\WB| O(e^{-cT} |k|^{-5/4}).
\end{equation*}
Sum over $k$ and note that $|v_\WB| = O(1)$ as $T \rightarrow \infty$ to obtain the second claim of \eqref{eq:2b}. This ends the proof of the lemma. \end{proof}

Theorem \ref{thm:4} follows then from the results of \S \ref{subsec:2}. This proves Theorem \ref{thm:haw}.

\section{Appendix}

Here we prove two approximation lemma.

\begin{lem}\label{lem:1e} Let $\chi \in C_0^\infty(\R^+,[0,1])$, $\zeta \in C^\infty(\R^+,\R^+)$ and $f \in S^{-\delta}(\R \times \Ss)$ for some $\delta \in (0,1/2)$. Assume that $\eta \leq 1+y \zeta(y)$ for some $\eta \in (0,1)$ and define
\begin{equation*}
F_h(y,\w) = \chi(y) \left(f\left(\dfrac{y+y^2\zeta(y)}{h},\w\right) - f\left(\dfrac{y}{h},\w\right)\right), \ \ G_h(y,\w) = \chi(y)f\left(\dfrac{y+y^2\zeta(y)}{h},\w\right).
\end{equation*}
Then $|F_h|_{H^{1/2}} = O(h^\delta)$ and $|G_h|_{H^{-1/2}} = O(h^{\delta})$.
\end{lem}

\begin{proof} We first note that derivatives with respect to $\w$ of $F_h$ are uniformly bounded in $h$ and thus not relevant. Without loss of generality we assume that $F_h$ does not depend on $\w$. We first prove a bound on $|F_h|_2$. Fix $A > 0$ such that $\supp(\chi)\subset [0,A]$. The substitution $y \mapsto y/h$ yields
\begin{equation*}\begin{gathered}
|F_h|_2^2 = \int_0^A \chi(y)^2 \left|f\left(\dfrac{y+y^2\zeta(y)}{h}\right) - f\left(\dfrac{y}{h}\right)\right|^2 dy \\
= h \int_0^{A/h} \chi(hy)^2 \left|f(y+hy^2\zeta(yh)) - f(y)\right|^2 dy \\ = h \int_0^{A/h} h^2 \chi(hy)^2 \zeta(hy)^2 y^4 \left|\int_0^1 f'(y+s h y^2\zeta(hy)) ds \right|^2 dy.
\end{gathered}
\end{equation*}
Note that for all $s \in [0,1]$ and $h,y > 0$,
\begin{equation*}
h(y+sh y^2 \zeta(hy)) \geq \min(hy,hy+(hy)^2\zeta(hy)) \geq \eta hy
\end{equation*}
and thus $y+ sh y^2 \zeta(hy) \geq \eta y$. In addition both $\chi$ and $\zeta$ are uniformly bounded and $f' \in S^{-1-\delta}$. Thus using the Cauchy-Schwartz inequality, the decay of $f'$ and the fact that $0 < \delta < 1/2 < 3/2$, 
\begin{equation*}
\begin{gathered}
|F_h|_2^2 \leq C h^3   \int_0^{A/h} \lr{y}^4 \lr{y}^{-2-2\delta} dy \\
\leq C h^3 \left|\left[lr{y}^{3-2\delta}\right]_0^{A/h}\right| \leq Ch^3 h^{-3+2\delta} = O(h^{2\delta}).
\end{gathered}
\end{equation*} 
This shows $|F_h|_2 = O(h^\delta)$. We now provide a bound on $|F'_h|_2$. Note that
\begin{equation}\label{eq:3q}\begin{gathered}
F_h'(y) = \chi'(y) \left(f\left(\dfrac{y+y^2\zeta(y)}{h}\right) - f\left(\dfrac{y+y^2\zeta(y)}{h}\right)\right) \\
+ \dfrac{\chi(y)}{h} \left(f'\left(\dfrac{y+y^2\zeta(y)}{h}\right) - f'\left(\dfrac{y}{h}\right)\right)
+ \dfrac{y\psi(y)}{h} f'\left(\dfrac{y+y^2\zeta(y)}{h}\right).
\end{gathered}
\end{equation}
where $\psi(y) = \chi(y)(2 \zeta(y) + y \zeta'(y))$. The first term in \eqref{eq:3q} is $O_{L^2}(h^\delta)$ by just replacing $\chi$ by $\chi'$ in the estimation of $|F_h|_2$. The term
\begin{equation}\label{eq:3p}
\chi(y) \left(f'\left(\dfrac{y+y^2\zeta(y)}{h}\right) - f'\left(\dfrac{y}{h}\right)\right)
\end{equation}
is exactly of the same form as $F_h$ with $f$ replaced by $f'$. As $f' \in S^{-1-\delta}$ and $1+\delta < 3/2$ the above calculation shows that the $L^2$-norm of \eqref{eq:3p} is $O(h^{1+\delta})$. Multiplying by the extra factor $h^{-1}$ yields that the $L^2$-norm of the second term in \eqref{eq:3q} is $O(h^\delta)$. For the third term in \eqref{eq:3q} we note that
\begin{equation*}\begin{gathered}
\int_0^A \left|\dfrac{y\psi(y)}{h} f'\left(\dfrac{y+y^2\zeta(y)}{h}\right)\right|^2 dy = h \int_0^{A/h} |y\psi(hy)|^2 |f'(y+hy^2\zeta(hy)|^2 dy \\
\leq C h \int_0^{A/h}   \lr{y}^{-2\delta} dy \leq Ch h^{-1+2\delta} = O(h^{2\delta})\end{gathered}
\end{equation*}
where here again we used $\delta < 1/2$. It follows that $|F_h'|_2 = O(h^\delta)$ and thus by interpolation $|F_h|_{H^{1/2}} = O(h^\delta)$.

We now estimate the $H^{-1/2}$ norm of $G_h$. As $L^{3/2} \hookrightarrow H^{-1/2}$ it suffices to show that $G_h = O_{3/2}(h^\delta)$. We have
\begin{equation*}
\begin{gathered}
|G_h|_{3/2}^{3/2} = h\int_0^{A/h} |f(y+hy^2\zeta(y),\w)|^{3/2} dy d\w \leq Ch \int_0^{A/h} \lr{y}^{-3\delta/2} dy  \\
\leq C h  \left|\left[lr{y}^{1-3\delta/2}\right]_0^{A/h}\right| \leq Ch h^{-1+3\delta/2} = O(h^{3\delta/2}).
\end{gathered}
\end{equation*}
Here again we used $\delta < 1/2$. This proves $|G_h|_{3/2} = O(h^\delta)$.\end{proof}

\begin{lem}\label{lem:1g} Let $\chi_0 \in C_0^\infty(\R)$ with $\chi_0 = 1$ near $0$; let $\rho \in C^\infty(\R)$ vanishing near $0$ and equal to $1$ near infinity. For $f \in S^{-\delta}, \delta < 1/2$ and $h \leq 1, \ell \geq 1$ we define 
\begin{equation*}
F_{\ell,h}(x,\w) = \chi_0(x) \rho \left(\dfrac{x}{\ell h}\right) f\left(\dfrac{
x}{h},\w \right).
\end{equation*}
Then $|F_{\ell,h}|_{H^{1/2}} = O(h^{\delta/2}) + O(\ell^{-\delta/2})$ uniformly in $h \leq 1, \ell \geq 1$.
\end{lem}

\begin{proof} Here again the angular variable plays no role and we assume that $f$ does not depend on $\w$. We start by getting an estimate on $|F_{\ell,h}|_2$. The function $F_{\ell,h}$ has compact support independent of $\ell,h$ and thus by Holder's inequality $|F_{\ell,h}|_2 \leq C|F_{\ell,h}|_{2/\delta}$. It follows that
\begin{equation*}\begin{gathered}
|F_{\ell,h}|_2 \leq C \left(\int_{\R} \left|\chi_0(x) \rho\left(\dfrac{x}{\ell h}\right) f\left(\dfrac{x}{h}\right)\right|^{2/\delta} dx\right)^{\delta/2} \\
\leq C h^{\delta/2} \int_\R f(x)^{2/\delta} dx = O(h^{2/\delta}).\end{gathered}
\end{equation*}
We now prove an estimate on $||D_x|^{1/2}F_{\ell,h}|_2$. We have
\begin{equation*}\begin{gathered}
2\pi||D_x|^{1/2}F_{\ell,h}|_2^2 = \int_\R |\xi| \left| \int_\R e^{-ix\xi}\chi_0(x) \rho\left(\dfrac{x}{\ell h}\right) f\left(\dfrac{x}{h}\right) dx \right|^2 d\xi \\
= \int_\R |\xi| \left| \int_\R e^{-ix\xi}\chi_0(h x) \rho\left(x/\ell\right) f(x) dx \right|^2 d\xi.\end{gathered}
\end{equation*}
Small values of $\xi$ do not cause any trouble and we focus on $\xi$ away from $0$. Integration by parts yields
\begin{equation*}\begin{gathered}
\int_{|\xi| \geq 1} |\xi| \left| \int_\R e^{-ix\xi}\chi_0(h x) \rho\left(x/\ell\right) f(x) dx \right|^2 d\xi \\
 = \int_{|\xi| \geq 1} \dfrac{d\xi}{|\xi|^3}\left| \int_\R e^{-ix\xi} \p_x^2\left(\chi_0(h x) \rho\left(x/\ell\right) f(x)\right) dx \right|^2  \leq \left(\int_\R \left| \p_x^2\left(\chi_0(h x) \rho(x/\ell) f(x) \right)\right| dx\right)^2.\end{gathered}
\end{equation*}
We now estimate this last term as follows. Derivatives carried on $\chi_0(h\cdot)$ (resp. $\rho(\cdot/\ell)$) create a factor $h$ (resp. a factor $\ell^{-1}$). The functions $\chi_0'(h\cdot), \chi_0''(h\cdot)$ (resp. $\rho(\cdot/\ell), \rho'(\cdot/\ell), \rho''(\cdot/\ell)$) are supported in the annulus $x \sim h^{-1}$ (resp. $x \sim \ell$). Derivatives carried on $f$ create additional decay: $f' \in S^{-1-\delta} \subset L^1$. All these observations put together show that
\begin{equation*}
\left| \p_x^2 \chi_0(h \cdot) \rho(\cdot/\ell) f \right|_1 = O(\ell^{-\delta}) + O(h^\delta).
\end{equation*} 
It finally follows that $|F_{\ell,h}|_{H^{1/2}} = O(h^{\delta/2} + \ell^{-\delta})$ and the lemma follows. \end{proof}


\begin{thebibliography}{2}
\bibitem[AB09]{AnBl} L. Anderson and P. Blue. Hidden symmetries and decay for the wave equation on the Kerr spacetime, to appear in Ann. of Mat.
\bibitem[Ba97a]{Bachelot0}  A. Bachelot. Scattering of scalar fields by spherical gravitational collapse. J. Math. Pures Appl. (9) 76 (1997), no. 2, 155-210.
\bibitem[Ba97b]{Bachelot2} A. Bachelot. Quantum vacuum polarization at the black hole horizon. Ann. Inst. H. Poincar\'e Phys. Th\'eor. 67 (1997), no. 2, 181-222.
\bibitem[Ba99]{Bachelot1} A. Bachelot. The Hawking effect. Ann. Inst. H. Poincar\'e Phys. Th\'eor. 70 (1999), no. 1, 41-99.
\bibitem[Ba00]{Bachelot3}  A. Bachelot. Creation of fermions at the charged black hole horizon. Ann. Henri Poincar\'e 1 (2000), no. 6, 1043-1095.
\bibitem[BH08]{BonHaf} J.F. Bony and D. H\"afner. Decay and non-decay of the local energy for the wave
equation in the De Sitter-Schwarzschild metric.  Comm. Math. Phys. 282 (2008), no. 3, 697-719.
\bibitem[BG14]{BouGer} P. Bouvier and C. Gerard. Hawking Effect for a Toy Model of Interacting Fermions. Ann. H. Poincar\'e (2015), no. 5, 1191-1230.
\bibitem[BR81]{BraRob} O. Bratteli and D. Robinson. Operator algebras and quantum-statistical mechanics, I and II. Texts and Monographs in Physics. Springer-Verlag, New York-Berlin, 1981.
\bibitem[DR07]{DafRod1} M. Dafermos and I. Rodnianski. The wave equation on Schwarzschild-de Sitter space
times. Preprint, arXiv:07092766.
\bibitem[DR11]{DafRod2} M. Dafermos and I Rodnianski. A proof of the uniform boundedness of solutions to the wave equation on slowly rotating Kerr backgrounds. Invent. Math. 185 (2011), no. 3, 467-559.
\bibitem[DRS14]{DRS14} M. Dafermos, I. Rodnianski  and Y. Shlapentokh-Rothman. Decay for solutions of the wave equation on Kerr exterior spacetimes III: The full subextremal case $|a| < M$,  arXiv:1402.7034.
\bibitem[Dy11a]{Dyatlov1} S. Dyatlov. Quasi-normal modes and exponential energy decay for the Kerr-de Sitter black hole. Comm. Math. Phys. 306 (2011), no. 1, 119-163.
\bibitem[Dy11b]{Dyatlov2} S. Dyatlov. Exponential energy decay for Kerr-de Sitter black holes beyond event horizons. Math. Res. Lett. 18 (2011), no. 5, 1023-1035.
\bibitem[Dy12]{Dyatlov3} S. Dyatlov. Asymptotic distribution of quasi-normal modes for Kerr-de Sitter black holes. Ann. Henri Poincar\'e 13 (2012), no. 5, 1101-1066.
\bibitem[DZ15]{DyaZwo} S. Dyatlov  and M. Zworski. Mathematical theory of scattering resonances, available online. 
\bibitem[FKSY06]{FKSY06} F. Finster, N. Kamran, J. Smoller, and S.-T. Yau, Decay of solutions of the wave equation in
the Kerr geometry, Comm. Math. Phys. 264 (2006), 465-503.
\bibitem[FKSY08Err]{FKSY08Err} F. Finster, N. Kamran, J. Smoller, and S.-T. Yau. Erratum: Decay of solutions of the
wave equation in the Kerr geometry, Comm. Math. Phys. 280 (2008), 563-573.
\bibitem[FH90]{FH}  K. Fredenhagen and R. Haag. On the derivation of Hawking radiation associated with the formation of a black hole. Comm. Math. Phys. 127 (1990), no. 2, 273-284.
\bibitem[Fr80]{Friedlander} F.G. Friedlander. Radiation Fields and hyperbolic scattering theory. Math.
Proc. Cambridge Philos. Soc. 88: 483-515 (1980).
\bibitem[Ga14a]{Ga14a} O. Gannot. Quasinormal modes for Schwarzschild-AdS black holes: exponential convergence to the real axis. Comm. Math. Phys. 330 (2014), no. 2, 771-799. 
\bibitem[Ga14b]{Ga14b} O. Gannot. A global definition of quasinormal modes for Kerr-AdS Black Holes, arXiv:1407.6686.
\bibitem[GH77]{GibHaw} G. Gibbons and S. Hawking. Cosmological event horizons, thermodynamics, and particle creation. Phys. Rev. D, no. 15, 1977.
\bibitem[Ni13]{JP} J.-P. Nicolas. Conformal scattering on the Schwarzschild metric. Preprint, arXiv:1312.1386.
\bibitem[H\"a09]{Hafner} D. H\"afner. Creation of fermions by rotating charged black holes. M\'em. Soc. Math. Fr. (N.S.) No. 117 (2009), 158 pp.
\bibitem[Ha75]{Hawking} S.H. Hawking. Particle creation by black holes. Comm. Math. Phys. 43 (1975), no. 3, 199-220.
\bibitem[H\"o07]{Hormander} L. H\"ormander. The analysis of linear partial differential operators III. Classics in Mathematics, Springer, Berlin, 2007.
\bibitem[Me04a]{Melnyk1} F. Melnyk. The Hawking effect for spin 1/2 fields. Comm. Math. Phys. 244 (2004), no. 3, 483-525.
\bibitem[Me04b]{Melnyk2} F. Melnyk. The Hawking effect for a collapsing star in an initial state of KMS type. J. Phys. A 37 (2004), no. 39, 9225-9249.
\bibitem[MSV14]{MelSaBVas} R. Melrose, A. S\'a Barreto and A. Vasy. Asymptotics of solutions of the wave equation on de Sitter-Schwarzschild space. Comm. Partial Differential Equations 39 (2014), no. 3, 512-529.
\bibitem[SZ97]{SaZw} A. Barreto and M. Zworski. Distribution of resonances for spherical black holes. Math. Res. Lett. 4 (1997), no. 1, 103-121.
\bibitem[Si00]{Simon} B. Simon. Resonances in one dimension and Fredholm determinants. J. Funct. Anal. 178 (2000), no. 2, 396-420.
\bibitem[TT11]{TT11} D. Tataru and M. Tohaneanu. A local energy estimate on Kerr black hole backgrounds. Int. Math. Res. Not. IMRN 2011, no. 2, 248-292.
\bibitem[Va13]{Vasy} A. Vasy. Microlocal analysis of asymptotically hyperbolic and Kerr-de Sitter spaces (with an appendix by Semyon Dyatlov). Invent. Math. 194 (2013), no. 2, 381-513.
\bibitem[Zw12]{Zworski} M. Zworski. Semiclassical analysis. Graduate Studies in Mathematics, 138. American Mathematical Society, Providence, RI, 2012.
\end{thebibliography}
\end{document}